\documentclass{article}

\usepackage{amsmath,amssymb,amsfonts}%
\usepackage{amsthm}%
\usepackage{mathrsfs}%
\usepackage[title]{appendix}%
\usepackage{xcolor}%
 \usepackage{manyfoot}%
\usepackage[all,2cell]{xy}
 \usepackage{tikz-cd}
\usepackage{amsfonts}
 \input diagxy
 \usepackage{csquotes}
 \usepackage{cancel}
\usepackage{appendix}
\usepackage{lipsum}
\usepackage{orcidlink}

\usetikzlibrary{arrows.meta} 
\tikzset{
  commutative diagrams/.cd,
  arrow style=tikz,
  diagrams={>={Computer Modern Rightarrow[length=5pt,width=3pt]}},
}

\hypersetup{colorlinks, citecolor=teal, breaklinks, linkcolor=red, urlcolor=blue}

 \theoremstyle{thmstyleone}%
\newtheorem{theorem}{Theorem}[section]
\newtheorem{definition}[theorem]{Definition}
\newtheorem{remark}[theorem]{Remark} 

\newtheorem{lemma}[theorem]{Lemma}
\newtheorem{examples}[theorem]{Examples}
\newtheorem{example}[theorem]{Example}
\newtheorem{proposition}[theorem]{Proposition}
\newtheorem{corollary}[theorem]{Corollary}
\newtheorem{notation}[theorem]{Notation}
\newtheorem{notations}[theorem]{Notations}

\newcommand{\Z}{\mathbb{Z}} 
\newcommand{\R}{\mathbb{R}} 
\newcommand{\xra}{\xrightarrow}
\newcommand{\sra}{\rightarrow}
\newcommand{\ra}{\longrightarrow}

\newcommand{\Ra}{\Longrightarrow}
\newcommand{\hra}{\hookrightarrow}
\newcommand{\st}{\stackrel}
\newcommand{\tm}{\times}
\newcommand{\colim}{\mathsf{colim}}
\newcommand{\lm}{\mathsf{lim}}

\newcommand{\ma}{\mathcal{A}}
\newcommand{\mk}{\mathcal{K}}
\newcommand{\mw}{\mathcal{W}}
\newcommand{\mi}{\mathcal{I}}
\newcommand{\mj}{\mathcal{J}}
\newcommand{\me}{\mathcal{E}}

\newcommand{\md}{\mathcal{D}}

\newcommand{\mb}{\mathcal{B}}
\newcommand{\mc}{\mathcal{C}}

\newcommand{\mr}{\mathcal{R}}

\newcommand{\mo}{\mathcal{O}}

\newcommand{\bn}{\mathbb{N}}
\newcommand{\bs}{\mathbb{S}}
 
\newcommand{\Vt}{\mathsf{Vect}}
\newcommand{\U}{\mathsf{U}}
\newcommand{\Ab}{\mathsf{Ab}}
\newcommand{\Fin}{\mathsf{Fin}}
\newcommand{\Tor}{\mathsf{Tor}}
\newcommand{\Po}{\mathsf{P}}
\newcommand{\Dis}{\mathsf{Dis}}

\newcommand{\Pt}{\mathsf{pt}}         
            
\newcommand{\Top}{\mathsf{Top}}
\newcommand{\Set}{\mathsf{Set}}
\newcommand{\Q}{\mathsf{P}} 
\newcommand{\Cat}{\mathsf{Cat}} 
 
\newcommand{\N}{\mathsf{N}}             

\newcommand{\hcTop}{\mathsf{h_cTop}}
\newcommand{\hwTop}{\mathsf{h_wTop}}
\newcommand{\hkTop}{\mathsf{h_kTop}}
\newcommand{\hTop}{\mathsf{hTop}}

\newcommand{\T}{\mathsf{T}}
\newcommand{\ZZ}{\mathcal{Z}}
\newcommand{\fTop }{\mathsf{fTop }}
 \newcommand{\uTop }{\mathsf{uTop }}
\newcommand{\kfTop }{\mathsf{kfTop }}
\newcommand{\khTop }{\mathsf{khTop }}
\newcommand{\kuTop }{\mathsf{kuTop }}
\newcommand{\khcTop }{\mathsf{kh_cTop }}
\newcommand{\khkTop }{\mathsf{kh_kTop }}
\newcommand{\khwTop }{\mathsf{kh_wTop }}

 \newcommand{\Rng }{\mathsf{Rng}}

\newcommand{\Sq}{\mathsf{Square}}

\newcommand{\Topo}{\mathsf{top}}
\newcommand{\Sier}{\mathsf{Sier}}
\newcommand{\Alex}{\mathsf{Alex}}
\newcommand{\Seq}{\mathsf{Seq}}
\newcommand{\Rcomp}{\mathsf{Rcomp}}
\newcommand{\Comp}{\mathsf{Comp}}
\newcommand{\kTop}{\mathsf{kTop}}
\newcommand{\Lcomp}{\mathsf{Lcomp}}
\newcommand{\lkTop}{\mathsf{lkTop}}
      
\newcommand{\Hom}{\mbox{Hom}}

\newcommand*{\absl}[1]{\left\lvert\:#1\:\right\rvert}       
\newcommand*{\abs}[1]{\left|#1\right|} 
\newcommand*{\Disc} {\mbox{Disc} } 
\newcommand*{\Codisc}{\mbox{Codisc} }

\newcommand{\map}{\mbox{map}}

 \newcommand{\oppair}[4]{\xymatrix{#1 \ar@<.5ex>[r]^-{#3} &{#2}\ar@<.5ex>[l]^-{#4}}}
\newcommand{\oppairi}[4]{\xymatrix@1{#1 \ar@<.5ex>[r]^{#3} &{#2}\ar@<.5ex>[l]^{#4}}}
 
\begin{document}

\title{Kan extendable subcategories and fibrewise topology}
 
\author{Moncef Ghazel\orcidlink{0000-0002-0006-945X} \footnote{Faculty of Sciences of Tunis,
          University of Tunis El manar\\  email: moncef.ghazel@fst.utm.tn}}

\date{ }

\maketitle

\begin{abstract}
 We use pointwise Kan extensions to generate new subcategories out of old ones. We investigate the properties of these newly produced categories and give sufficient conditions for their  cartesian closedness to hold. Our methods are of general use.  Here we apply them particularly to the study of the properties of certain categories of fibrewise topological spaces.  In particular, we prove that the categories of fibrewise compactly generated spaces, fibrewise sequential spaces and fibrewise Alexandroff spaces are cartesian closed provided that the base space satisfies the right separation axiom.
\end{abstract}



 \tableofcontents

\section*{Introduction}
In this paper, a subcategory of any category is always assumed to be full.\\
\indent A  subcategory $\mb$ of a category $\mc$ is said to be reflective if the inclusion functor $\mb \ra \mc$ has a left adjoint. Examples of such are the  subcategories of Hausdorff spaces, Tychonoff spaces, compact spaces and realcompact spaces of the category $\Top$ of topological spaces.
The reflective hull of a subcategory $\mw$ of $\mc$ is the smallest replete, reflective subcategory of $\mc$ containing $\mw$. Such a subcategory does not always exist,  for the intersection of all replete reflective subcategories of $\mc$ containing $\mw$ may not be reflective, as is shown by Adámek and Rosický \cite{A}. The existence of reflective hulls and their properties have been extensively studied by several authors 
\cite{A,BK,HST, HS, KK,T}. \\
\indent We, in Theorem \ref{pt3},  show that a replete reflective subcategory of $\mc$ containing $\mw$ as a codense subcategory is necessarily the reflective hull of $\mw$, and is therefore unique when it exists. We call such a subcategory the strong reflective hull of $\mw$.
Coreflective subcategories, coreflective hulls and strong coreflective hulls are dually defined.
The notion of the so-called strong reflective hull is strictly stronger than that of the reflective hull, in the sense that there are examples of reflective hulls which are not strong (see Remark \ref{pr12} and Example \ref{pex2}).\\
\indent When it exists, a (pointwise) right Kan extension $R$ of a functor $F:\ma \ra \mb$ along itself has a monad structure. This monad $R$ is called the codensity monad of $F$; for $R$ reduces to the identity functor $1_{\mb}$ iff the functor $F$ is codense. One has a dual notion of density comonad of the functor $F$. \\
\indent A monad $(T,\eta, \mu)$ is said to be idempotent when its  multiplication $\mu:T^2 \ra T$ is an isomorphism. Similarly, a comonad is said to be idempotent if its comultiplication is an isomorphism. \\  
 \indent We define a  subcategory $\mw$ of $\mc$  to be left Kan extendable if the inclusion functor $\mw \ra \mc$ has an idempotent density comonad $(L,\epsilon,\delta)$. When this is the case, then the category of $L$-coalgebras is denoted by $\mw_l[\mc]$ and 
the forgetful functor  $U:\mw_l[\mc] \ra \mc$  is fully faithful and injective on objects. Consequently, $\mw_l[\mc]$ is viewed as a subcategory of $\mc$. Dually, the subcategory $\mw$ of $\mc$ is said to be right Kan extendable provided that the inclusion functor $\mw \ra \mc$ has an idempotent codensity monad $(R,\eta,\mu)$. In this case, the category of $R$-algebras is denoted by 
$\mw_r[\mc]$.\\
\indent The two notions of strong reflective hull and right Kan extendability are closely related:  a subcategory $\mw$ of $\mc$ has a strong reflective hull iff $\mw$ is  right Kan extendable in $\mc$. When this is the case, then the strong reflective hull of $\mw$ is precisely the subcategory  $\mw_r[\mc]$ of  $\mc$ (dual of Theorem \ref{pt1}). \\
\indent As applications, we prove that the subcategory  of $\Top$ whose only object is the square of the unit interval has a strong reflective hull which is the subcategory of compact Hausdorff spaces. Similarly, we prove that the subcategory of $\Top$ whose only object is the square of the real line is the subcategory of realcompact spaces. Consequently, one recovers the Stone–Čech compactification and the Hewitt realcompactification procedures.\\
\indent Fibrewise topology is a branch of topology which   studies the slice categories of  $\Top$. It plays an important role in homotopy theory as shown by Crabb and James in their book \cite{CM}, and is now considered as a subject in its own right. One of the main objectives of this paper is to extend some of the categorical properties of certain subcategories of $\Top$ to  their fibrewise counterparts.\\ 
\indent 
It is  a well known fact that the subcategories of $\Top$  of Fréchet spaces,  Hausdorff spaces, Urysohn spaces,  completely Hausdorff spaces, weak Hausdorff spaces and  $k$-Hausdorff spaces are reflective. Let $B$ be a topological space and let $\Top_B$ be the category of fibrewise spaces over $B$. We use the theory of Kan extendable subcategories to present a general theorem allowing one to recognize reflective subcategories of  $\Top_B$. We then use it to prove, in a harmonized and systematic manner, that the fibrewise  versions of the above  subcategories of $\Top$ are again reflective subcategories of $\Top_B$. \\
\indent It is a classical result of Herrlich and Strecker that any subcategory $\mw$ of $\Top$ containing a nonempty space is, in our terminology, left Kan extendable  (\cite[Proposition 2.17]{HH}, \cite[Theorem 12]{HST} and \cite[page 283]{HEST}). Moreover, if the objects of $\mw$ are exponentiable in $\Top$ and if $\mw$ satisfies an  additional condition, then a celebrated theorem of Day   asserts $\mw_l[\Top]$ is cartesian closed \cite[Theorem 3.1]{D}. In the most famous application, one
takes $\mw$ to be the subcategory $\Comp$ of compact Hausdorff spaces to deduce that the category  of compactly generated spaces, which is the
strong reflective hull of $\Comp$, is cartesian closed \cite[Corollary 3.3]{D}. Similarly, by taking $\mw$ to be the subcategory of $\Top$ whose only object is the one-point compactification of a discrete countable space, we  deduce that the subcategory of $\Top$ of sequential spaces is cartesian closed and, by taking $\mw$ to be the subcategory of $\Top$ whose only object is the Sierpinski space, one gets the fact that the subcategory of Alexandroff spaces is cartesian closed.    \\
\indent The category $\Top_B$ is not cartesian closed.  Lots of work with varying success has been done to provide a convenient substitute for it. In \cite[Theorem 2.2]{BP1}, Booth proves that the category of fibrewise quasi-topological spaces, in which the category of fibrewise topological spaces embeds, is cartesian closed. In a later work, Booth and Brown defined a partial map version of the compact-open topology and use it to describe a fibrewise mapping-space satisfying certain exponential laws \cite[Section 3]{BB}.
Variants of the Booth-Brown topology on the mapping space were used by James to show that an exponential law holds in certain situations (\cite[Proposition 5.6]{J1} and \cite[Proposition 10.15]{J2}).        \\
\indent Let $\mw$ be a left Kan extendable subcategory of $\mc$  whose objects are exponentiable in $\mc$. We show that under mild conditions, the subcategory $\mw_l[\mc]$ of $\mc$ is cartesian closed  (Theorem \ref{pt10}). \\
\indent We here prove that a subcategory $\mw$ of $\Top_B$, which is suitable in a specified sense, is left Kan extendable (Theorem \ref{t2}).   
 We then use Theorem \ref{pt10} to derive a fibrewise version of Day's theorem.  As application, we prove that the category of fibrewise compactly generated spaces is cartesian closed provided that the base $B$ is $\T_1$ (Theorem \ref{t7});  a result  which is not proved neither in \cite{BP1,BP, BB} nor in \cite{J1,J2} and is new to author's knowledge.
Further applications
 include the cartesian closedness of the category of
fibrewise sequential spaces (Proposition \ref{p22}) and that of fibrewise Alexandroff
spaces (Proposition \ref{p37}), provided that the base $B$ satisfies the right separation axiom.

 \indent The paper is structured as follows: Section \ref{s1} contains a brief discussion of  reflective subcategories and their properties that are being used throughout. In particular, the concept of strong reflective hull is introduced and its connection with  the ordinary reflective hull is investigated. In Section \ref{s2},  we recall the basic definitions and facts about codensity monads and their idempotency. These are used in Section \ref{s4} to define the notion of Kan extendable subcategories  and study their properties. In Section \ref{s5}, we use the theory of Kan extendable subcategories to derive the Stone–Čech compactification and the Hewitt realcompactification procedures. In Section \ref{s6}, we prove that  subcategories of fibrewise topological spaces over $B$ which satisfy certain separation axioms are reflective  subcategories of $\Top_B$. 
In Section \ref{s9}, we investigate the concept of fibrewise compact spaces. We in particular prove that a fibrewise compact fibrewise Hausdorff space over a $\T_1$ base $B$ is an exponential object of $\Top_B$, a fact that is needed to give one of the main applications of the paper.
In Section \ref{s7}, we introduce the subcategories of $\Top_B$   of fibrewise weak Hausdorff spaces and fibrewise $k$-Hausdorff spaces and prove that they are reflective in $\Top_B$.  
In Section \ref{s8}, we present a sufficient condition for a subcategory of $\Top_B$ to be left Kan extendable.
In Section \ref{s10}, we state conditions that ensure the cartesian closedness of the strong coreflective hull of a subcategory (Theorem \ref{pt10}). The fibrewise Day's theorem is presented and proved in Section \ref{s11} (Theorem \ref{t6}). It is then used in  Sections \ref{s12} to prove that the category  $\kTop_B$ of fibrewise compactly generated topological spaces over a $\T_1$ base $B$ is cartesian closed. Properties of certain subcategories of $\kTop_B$ are inspected in Section \ref{s13}. Sections \ref{s14} and \ref{s15} are devoted to the study of fibrewise sequential spaces and fibrewise Alexandroff spaces respectively. In Appendix \ref{A}, limits in a slice category  are investigated.
Specializations of the results of Appendix \ref{A} to either a slice category of sets or a category of fibrewise topological spaces are given in appendices  \ref{B} and \ref{C}.

 \subsection*{Conventions and notations}
   
Throughout this paper, the product of two categories $\ma$ and  $\mb$
is denoted by $\ma\tm\mb$. A subcategory $\mb$ of a category $\mc$ is always assumed to be full. Given two objects $X,Y \in\mc$, the set of morphisms from $X$ to $Y$ is denoted by $\mc(X,Y)$. When it exists, the cartesian product of $X$ and $Y$ in $\mc$ is denoted by $X\tm_{\mc} Y$. If $X$ and $Y$ are in the subcategory $\mb$ of  $\mc$, then their cartesian product $X\tm_{\mb} Y$ in $\mb$,  when it exists  may be different from their product $X\tm_{\mc} Y$  in the larger category $\mc$ and should not be confused with it. Given a family of objects $(X_i)_{i\in I}$ of $\mc$, when they exist, the product and coproduct over $I$ of the $X_i$'s are denoted by $\underset{\mc}{\overset{i\in I}{\prod}}X_i$ and $\underset{\mc}{\overset{i\in I}{\coprod}}X_i$ respectively.\\  
\indent Throughout this paper, $B$ denotes a fixed topological space. The slice category $\Top/B$ 
is called the category of fibrewise topological spaces over $B$ and denoted simply by $\Top_B$. (see Appendix \ref{C}).

\subsection*{Acknowledgments}
 I would like to greatly thank Inès Saihi for
the careful revision  of an early draft and the
several corrections she made.
   I am also deeply grateful to the referees whose comments considerably helped streamline the paper and improve its readability.
\section{Reflective subcategories}\label{s1}
In this section, we briefly recall the notion of reflective subcategories and discuss some of their relevant properties. \\

\begin{definition}\label{pd4}

Let $\mc$ be a category.
\begin{enumerate}
 
	\item A subcategory $\mc_0$ of $\mc$ is said to be reflective  if the inclusion functor $\mc_0 \st{U}\hra \mc$ is a right adjoint functor. In this case, a left adjoint functor $F$ of $U$ is called a reflector and the adjoint pair $(F\dashv U)$ is called a reflection of $\mc$ on $\mc_0$.
The unit $1_{\mc}\st{\eta}\Ra UF$ and counit $FU\st{\epsilon}\Ra 1_{\mc_0}$ of the adjunction $(F\dashv U)$ are also called the unit and counit of the reflection $(F\dashv U)$ of $\mc$ on $\mc_0$.
  
 \item Dually, a subcategory $\mc^0$ of $\mc$ is said to be coreflective  if the inclusion functor $\mc^0 \st{U}\hra \mc$ is a left adjoint functor. In this case, a right adjoint $G$ of $U$ is called a coreflector and the adjoint pair $(U\dashv G)$ is called a coreflection of $\mc$ on $\mc^0$. 
The unit $1_{\mc^0}\st{\eta}\Ra GU$ and counit $UG\st{\epsilon}\Ra 1_{\mc}$ of the adjunction $(U\dashv G)$ are also called the unit and counit of the coreflection $(U\dashv G)$ of $\mc$ on $\mc^0$.  
 	\end{enumerate}
\end{definition}
Under the conditions of Definition \ref{pd4}.1, the objects of $\mc_0$ are often identified with their images in $\mc$ by the inclusion functor $U$. In particular,  the components of the unit $\eta$ of the reflection $(F\dashv U)$ are viewed as maps $  \eta_C:C \ra F(C)$ in $\mc$.


\begin{lemma}\label{pl1}\mbox{} 
\begin{enumerate}

\item  Let $(F\dashv U)$ be a reflection of $\mc$ on $\mc_0$ with unit
$1_{\mc}\st{\eta}\Ra UF$ and counit $FU\st{\epsilon}\Ra 1_{\mc_0}$. Then

\begin{enumerate}

\item The natural transformation $\epsilon$ is an isomorphism.
\item An object $C\in \mc$ is isomorphic to an object in $\mc_0$ iff the map $\eta_C: C\ra F(C)$ is an isomorphism. 
\end{enumerate}
\item Dually, let $(U\dashv G)$ be a coreflection of $\mc$ on $\mc^0$ with unit
$1_{\mc}\st{\eta}\Ra GU$ and counit $UG\st{\epsilon}\Ra 1_{\mc_0}$. Then

\begin{enumerate}
 
\item The natural transformation $\eta$ is an isomorphism.
\item An object $C\in \mc$ is isomorphic to an object in $\mc^0$ iff the map $\epsilon_C: G(C)\ra C$ is an isomorphism. 
\end{enumerate}
\end{enumerate}
\end{lemma}
 
\begin{proof}\mbox{}

\begin{enumerate}
 
\item  
\begin{enumerate}
 
\item  The functor $U:\mc_0\hra \mc$ is fully faithful, therefore by \cite[Theorem 1 page 90]{ML}, the counit $  FU \st{\epsilon}\Ra 1_{\mc_0}$ is an isomorphism. 
\item If $C\in \mc$ is such that $\eta_C: C\ra F(C)$ is an isomorphism, then obviously, $C$ is isomorphic to the object $F(C)$ of $\mc_0$. Conversely, assume that there is an isomorphism $C_0\st{f}\ra C$, where $C_0 \in \mc_0$.
  The counit $  FU \st{\epsilon}\Ra 1_{\mc_0}$ is an isomorphism, therefore $  UFU \st{U\epsilon}\Ra U$ is an isomorphism. By \cite[Theorem 1 page 82]{ML} the composite 
 $$  U \st{\eta U}\Ra UFU \st{U\epsilon}\Ra U$$ 
is the identity natural transformation. 
Therefore $U \st{\eta U}\Ra UFU $ is an isomorphism. It follows that $\eta_{C_0}: {C_0}\ra F({C_0})$ is an isomorphism. In the following commutative diagram 
$$\begin{tikzcd}
 C_0\arrow[r,"\eta_{C_0}"]\arrow[d,"f"']& F(C_0) \arrow[d,"F(f)"]\\
C \arrow[r,"\eta_{C}"'] & F(C) 
\end{tikzcd}$$
$f$, $F(f)$ and $\eta_{C_0}$ are isomorphisms. Therefore $\eta_{C}$ is an isomorphism. 
\end{enumerate}
\item    The second property is dual to the first.
\end{enumerate}
\qedhere  
\end{proof}
\begin{definition}\label{pd3}
Let $\mb$ be a subcategory of a category $\mc$ and $C\st{f}\ra C'$ a morphism in $\mc$. Then
\begin{enumerate}
 
\item   $f$ is said to be $\mb$-monic if given two maps $\alpha$, $\beta$ from an object $B$ in $\mb$ to $C$, then $f\alpha= f\beta \Ra \alpha=\beta$.
\item Dually, $f$ is said to be $\mb$-epic if given two maps $\alpha$, $\beta$ in $\mc$ from $C'$ to an object $B$ in $\mb$, then $\alpha f= \beta f \Ra \alpha=\beta$.  
	\end{enumerate}
\end{definition}
\begin{lemma}\label{pl4}\mbox{}   
 \begin{enumerate}
 
\item Assume that $(F\dashv U)$ is a reflection of a category $\mc$ on a subcategory $\mc_0$ with unit
$1_{\mc}\st{\eta}\Ra UF$. Then for every $C\in \mc$, the morphism $C\st{\eta_C}\ra F(C)$ is $\mc_0$-epic.
\item Dually, assume that $(U\dashv G)$ is a coreflection of a category $\mc$ on a subcategory $\mc^0$ with counit $UG\st{\epsilon}\Ra 1_{\mc^0}$.  Then for every $C\in \mc$, the morphism $G(C)\st{\epsilon_C}\ra C$ is $\mc^0$-monic.  
	\end{enumerate}
\end{lemma}
   
\begin{proof}\mbox{}
\begin{enumerate}
 
\item Let $C \in \mc$ and $C_0 \in \mc_0$. The map $\mc_0(F(C),C_0)\xra{\mc_0(\eta_C,C_0) }\mc(C,C_0)$ is bijective, therefore injective. Thus  $\eta_C$ is $\mc_0$-epic.
\item The second property is dual to the first.   
\end{enumerate}
\qedhere  
\end{proof}  
 \begin{proposition}\label{pp13}(Riehl, \cite[Proposition 4.5.15]{RE})
\begin{enumerate}
 
\item Let $\mc_0\hra \mc$ be a reflective subcategory, then
 
\begin{enumerate}
 
	\item  The inclusion functor $\mc_0\hra \mc$ creates all limits that $\mc$ admits.
	\item The subcategory $\mc_0$ has all colimits that $\mc$ admits, formed by applying the reflector to the colimit in $\mc$.
\end{enumerate}
 
In particular if $\mc$ is either complete or cocomplete, then so is $\mc_0$.
\item Let $\mc^0\hra \mc$ be a coreflective subcategory, then
 
\begin{enumerate}
 
\item  The inclusion functor $\mc^0\hra \mc$ creates all colimits that $\mc$ admits.
\item The subcategory $\mc^0$ has all limits that $\mc$ admits, formed by applying the coreflector to the limit in $\mc$.
\end{enumerate}

In particular if $\mc$ is either complete or cocomplete, then so is $\mc^0$.
\end{enumerate}
\end{proposition}
\indent The following result is a generalization of \cite[Proposition 3]{HS}. 
 
\begin{lemma}\label{pl3}
Let $\mc_0$ be a subcategory of a category $\mc$ which is either reflective or coreflective. Then the retract in $\mc$ of an object in $\mc_0$ is isomorphic to an object of $\mc_0$.
\end{lemma}
 
\begin{proof} We only need to prove the property in the reflective case.  Let $$A\st{i}\ra X\st{r}\ra A$$ be a retraction in $\mc$ of an object $X\in \mc_0$. The diagram
\begin{equation}\label{peq8}
\begin{tikzcd}
A\ar[r,"i"]&X \ar[r,shift left=.5ex,"ir"]
  \ar[r,shift right=.5ex,swap,"1_X"]
&X
\end{tikzcd}
\end{equation}
is an equalizer in $\mc$. For $iri=i=1_Xi$. Let $f:Y \ra X$ be such that $irf=f$. Assume that  $g:Y \ra A$ is such that $ig=f$. 
$$\begin{tikzcd}
&Y\ar[d,"f"]\ar[dl,"g=rf"',densely dotted]&\\
A\ar[r,"i"']&X \ar[r,shift left=.5ex,"ir"]
  \ar[r,shift right=.5ex,swap,"1_X"]
&X
\end{tikzcd}$$
Then $ig=f=irf$. The morphism $i$ is monic, thus $g=rf$ and $g$ is unique. Now define $g=rf$, $ig=irf=f$. It follows that (\ref{peq8}) is an equalizer.
By Proposition \ref{pp13}.1.(a), $A$ is isomorphic to an object of $\mc_0$.
\qedhere  
\end{proof}
Recall that a subcategory $\ma$ of a category $\mc$ is said to be replete if
  any object of $\mc$ which is isomorphic to an object of $\ma$  is itself in  $\ma$. 
 
\begin{definition}\label{pd5}
Let $\mw$ be a subcategory of a category $\mc$.

\begin{enumerate}
 
\item A subcategory $\mc_0$ of $\mc$ is called the reflective hull of $\mw$ in $\mc$ if it is the smallest replete, reflective subcategory of $\mc$ containing $\mw$.
\item Dually, a subcategory $\mc^0$ of $\mc$ is called the coreflective hull of $\mw$ in $\mc$ if it is the smallest replete, coreflective subcategory of $\mc$ containing $\mw$.
	
\end{enumerate}
\end{definition}
  A reflective (resp. coreflective) hull of a subcategory may not always exist, as is shown in \cite{A}, but if it does, then it is unique. 
A subcategory $\mw$ of a category $\mc$ has a reflective (resp. coreflective) hull iff the intersection of all reflective (resp. coreflective), replete subcategories of $\mc$ containing 
$\mw$ is again a reflective (resp. coreflective) subcategory of $\mc$. In which case, this intersection is precisely the reflective (resp. coreflective) hull of $\mw$.\\
\indent Let $F:\ma \ra \mc$ be a functor. For $C\in\mc$, let $F/C$ be standard comma category,
	$D_C:F/C\ra\ma$ be the functor which takes an arrow-object $F(A)\st{\sigma} \ra C$ to  $A$ and  $F_C$ be the composite functor
	\begin{equation}\label{peq9}
F/C \st{D_C}\ra\ma \st{F}\ra \mc
\end{equation}
Recall that the functor $F$ is said to be dense if for each $C\in \mc$, $F_C$ has a colimit and the natural map  $\colim F_C\ra C$ is an isomorphism. 
If $\ma$ is a subcategory of $\mc$ and $J:\ma \ra \mc$ is the inclusion functor, then for $C\in \mc$, the comma category $J/C$ is also denoted by $\ma/C$. The functor $J_C$ is the composite
\begin{equation}\label{peq9}
\ma/C \st{D_C}\ra\ma \st{J}\ra \mc.
\end{equation}
The subcategory $\ma$ of $\mc$ is said to be dense in $\mc$ if the functor $J$ is dense. One has dual notions of codense functor and codense subcategory. 
\begin{theorem}\label{pt3}
Let $\mw$ be a subcategory of a category $\mc$.
\begin{enumerate}
 
\item  Assume that $\mc_0$ is a replete reflective  subcategory  of $\mc$ in which $\mw$ is codense. Then $\mc_0$ is the reflective hull of $\mw$ in $\mc$.
\item Dually, assume that $\mc^0$ is a replete coreflective subcategory  of $\mc$ in which $\mw$ is dense. Then $\mc^0$ is the coreflective hull of $\mw$ in $\mc$.
\end{enumerate}
\end{theorem}
 
\begin{proof}
We prove the first property, the second one is the dual of the first.
Let $\mc'_0$ be a replete reflective subcategory of $\mc$ containing 
$\mw$.
Define $\mc_0 \st{U_0}\ra\mc$ and $\mc'_0 \st{U'_0}\ra\mc$ to be the inclusion functors and let $X\in \mc_0$. Define $X/\mw$ to be the subcategory of the under category $X/\mc$ whose objects are arrows $X\sra V$ with $V\in \mw$. Let $J^X:X/\mw\ra \mc$ be the functor which takes an arrow-object $X\sra V$ to its codomain $V$. The functor $J^X$ takes values in $\mw$ which is contained in $\mc_0$ and $\mc'_0$, therefore $J^X$ factors through $\mc_0$ and $\mc'_0$ as shown in the following commutative diagram 
$$\begin{tikzcd}[row sep=0.8cm, column sep=1cm]
& \mc_0  \arrow[dr, "U_0" ] & \\
X/\mw  \arrow[dr, "J'^{X}_{0}"'] \arrow[ur, "J^{X}_{0}"]\arrow[rr,"J^X"] 
& & \mc  \\
& \mc'_0   \arrow[ur, "U'_0"']
\end{tikzcd}$$
The subcategory $\mw$ is codense in $\mc_0$, thus $\lm J^{X}_{0} = X$.  Being a right adjoint, $U_0$ preserves limits. Thus $J_X=U_0J^{X}_{0}$ has a limit and $\lm J^{X} = X$. We have $J_X=U'_0J'^{X}_{0}$ and $\mc'_0$ be a replete. By Proposition \ref{pp13}.1.(a), $J'^{X}_{0}$ has a limit, $X\in \mc'_0$ and $\lm J'^{X}_{0} = X$. It follows that $\mc_0$ is a subcategory of $\mc'_0$. Therefore $\mc_0$ is the reflective hull of $\mw$ in $\mc$.
\qedhere  
\end{proof}
Recall that subcategories are always assumed to be full. 
\begin{remark}\label{pr1} 
 Given a subcategory $\mw$ of $\mc$.   
Theorem \ref{pt3} shows that:
\begin{enumerate}
 
\item   There is at most one replete reflective subcategory of $\mc$ in which $\mw$ is codense. When it exists, it is certainly the reflective hull of $\mw$, and is  called the \textbf{\textit{strong reflective hull}}
 of $\mw$ in $\mc$.

\item Dually,  there is at most one replete coreflective subcategory of $\mc$ in which $\mw$ is dense. When it exists, it is certainly the coreflective hull of $\mw$, and is  called the \textbf{\textit{strong coreflective hull}}
 of $\mw$ in $\mc$. \\
\end{enumerate}
\end{remark} 
\begin{corollary}\label{pc10} 
 Given a subcategory $\mw$ of $\mc$. 
 
\begin{enumerate}
 
\item The subcategory $\mw$ has a strong reflective hull iff it has a reflective hull in which it is codense.   

\item Dually, $\mw$ has a strong coreflective hull iff it has a coreflective hull in which it is dense.  
\end{enumerate}
\end{corollary}
  
\begin{proof}
This is a consequence of Theorem \ref{pt3}.
\qedhere  
\end{proof}
We next give an example of a subcategory which has a coreflective hull but has no strong  coreflective hull.

\begin{example}\label{pex2} 
  Let  $\Vt$  be the category of $\Z/2\Z$-vector spaces and $\ZZ$ the subcategory of $\Vt$ containing $\Z/2\Z$ as its unique object.
Let $\mc$ be a replete coreflective subcategory of $\Vt$ containing $\ZZ$.
By Proposition \ref{pp13}.2.(a),  $\mc$ contains $\Z/2\Z\oplus \Z/2\Z$. By \cite[ page 247 Exercise 1]{ML}, $\Z/2\Z\oplus \Z/2\Z$
is dense in $\Vt$. Thus by  Proposition \ref{pp13}.2.(a), $\mc=\Vt$. It follows that $\Vt$
is the unique coreflective subcategory of $\Vt$ containing $\ZZ$ and it is consequently its coreflective hull.
Let $A=\Z/2\Z\oplus \Z/2\Z \in \Vt$ and let
$J_{A}:\ZZ/A \ra \Vt $ be the functor which takes an arrow-object $\Z/2\Z \ra A$ in $\ZZ/A$ to its domain $\Z/2\Z$. Clearly, $\colim J_{A}\cong \Z/2\Z\oplus \Z/2\Z\oplus \Z/2\Z$. It follows that the subcategory $\ZZ$ of $\Vt$ is not dense in $\Vt$.  By Corollary \ref{pc10}.2, $\ZZ$ has no strong coreflective hull. \\
\end{example}
We close this section with the following observation. 

\begin{remark}\label{pr4}  
Let $\mc$ be a cartesian closed category with internal $\hom$ functor
$$\begin{array}{rccl}
  (.)^{(.)}:&\mc^{op}\tm \mc &\ra &  \mc\\
     &(Y,Z)&\longmapsto&Z^Y										
\end{array}$$
Let $\mc_0$ be a reflective subcategory of $\mc$ and assume that for every $Y,Z\in \mc_0$, the power object $Z^{Y}\in \mc_0$. Then:

\begin{enumerate}
 
\item By Proposition \ref{pp13}.1.(a), for every $X,Y\in \mc_0$, $X\tm_{\mc_0} Y$ exists and is isomorphic to the product $X\tm_{\mc} Y$ of $X$ and $Y$ in $\mc$.
\item  The category $\mc_0$ is cartesian closed with internal $\hom$ functor induced by that of $\mc$.
\end{enumerate}
\end{remark}
 
\section{Idempotent codensity monads}\label{s2}
 Here, we   recall the concepts of monads, codensity monads and their algebras. Details may be found in  \cite[Ch. VI]{ML},   \cite[Ch. 5]{RE}, \cite[Ch. 4]{BR}, \cite[page 67]{DE} and \cite[Section 2]{CF}. These notions are needed to define the main concept of this paper, which that of Kan extendable subcategories.\\
\indent Let $\mc$ be a category.  
\begin{itemize}
 
\item The category $\mc^{\mc}$ of endofunctors of $\mc$ is a monoidal category with composition of functors as its monoidal product.
\item A monad on $\mc$ is an unital associative monoid in $\mc^{\mc}$. It consists then of a triple $(T,\eta,\mu)$, where $T:\mc\ra \mc$ is a functor, $\mu:T^2 \ra T$ is an associative multiplication with unit $\eta:1_{\mc} \ra T$. 

\end{itemize}

  Let $(T,\eta,\mu)$ be a monad on the category $\mc$.
\begin{itemize}
 
\item The monad $(T,\eta,\mu)$ is said to be  idempotent if the multiplication  $\mu:T^2 \ra T$ is an isomorphism.
\item An algebra over $T$ is a pair $(A,u)$ consisting of an object $A \in \mc$ and a morphism $u:TA \ra A$  rendering commutative the diagrams:

$$ \begin{array}{lccr}
        \begin{tikzcd}
T^2A  \arrow[r, "Tu"]\arrow[d,"\mu_A"'] & TA \arrow[d, "u"] \\
TA\arrow[r, "u"']& A 
\end{tikzcd}&&&\begin{tikzcd}
A  \arrow[r, "\eta_A"]\arrow[dr,"1_{A}"'] & TA \arrow[d, "u"] \\
& A 
\end{tikzcd}      
	\end{array}$$

\item Given two $T$-algebras $(A,u)$ and $(B,v)$. A morphism of $T$-algebras from $A$ to $B$ is an arrow  $f:A \ra B$ rendering commutative the diagram 

$$\begin{tikzcd}
TA  \arrow[r, "Tf"]\arrow[d,"u"'] & TB \arrow[d, "v"] \\
A\arrow[r, "f"']& B 
\end{tikzcd}$$

\item Algebras over $T$ and their morphisms form a category denoted by $\mc^{T}$. It admits a   forgetful 
 functor $U: \mc^{T} \ra \mc$ which is right adjoint to the free $T$-algebras functor $F: \mc \ra \mc^{T}$. 
\end{itemize}

\begin{proposition}\label{pp4} (Borceaux, \cite[Proposition 4.1.4]{BR})\\
Let  $(T,\eta,\mu)$ be a monad on a category $\mc$. Let $U: \mc^{T} \ra \mc$ be the forgetful functor and $F: \mc \ra \mc^{T}$ the free $T$-algebras functor. Then $(F\dashv U)$ is an adjunction with unit the unit  $\eta:1_{\mc}\ra T=UF$ of the monad $T$.  
\end{proposition}

In order to recall the notion of idempotent monad, we state the following result which is part of  \cite[Proposition 7.2]{KL} of Kelly and Lack.

\begin{proposition}\label{pp1} 
Let  $(T,\eta,\mu)$ be a monad on a category $\mc$. Then the following properties are equivalent:
\begin{enumerate}
 
\item The monad $T$ is idempotent.
\item The natural transformation $\eta T:T \ra T^2$ is an isomorphism.
\item  The natural transformation $T \eta :T \ra T^2$ is an isomorphism.
\item The functors $\mu$ and $\eta T$ are mutually inverse.
\item The functors $\mu$ and $T \eta$ are mutually inverse.
\item The functors $\eta T$ and $T \eta$ are equal.

\item  For each object $A$ of $\mc$, a map $u:TA \ra A$ defines an algebra structure on $A$ iff it is inverse to $\eta_A$.
\item The forgetful functor $U:\mc^T \ra \mc$ is full. 
\item The forgetful functor $U:\mc^T \ra \mc$ is full and faithful. 

\end{enumerate}
\end{proposition}

\begin{proposition}\label{pp2} 
Let $(T,\eta,\mu)$ be an idempotent monad on a category $\mc$ and let $A$ be an object of $\mc$. Then the following three conditions are equivalent: 
\begin{enumerate}
 
\item The object $A$ of $\mc$ has a $T$-algebra structure.
\item The unit map $\eta_A: A \ra TA$ is an isomorphism.
\item The object $A$ of $\mc$ is isomorphic to a certain $T$-algebra.
\end{enumerate}
In particular, $\mc^T$ is a replete subcategory of $\mc$.
\end{proposition}
   
\begin{proof}\mbox{}
\begin{itemize}
 
	\item $1 \Longleftrightarrow 2:$ The monad $T$ is idempotent. Therefore by  Proposition \ref{pp1}.7, an object  $A$ of $\mc$ has a $T$-algebra structure iff $\eta_A: A \ra TA$ is an isomorphism.

\item $2\Ra 3:$ If $\eta_A: A \ra TA$ is an isomorphism, then the object $A$ of $\mc$ is isomorphic to the free algebra $TA$.
\item $3\Ra 2:$  Assume that $f:A \ra B$ is an isomorphism, where $B$ is a $T$-algebra. In the following commutative diagram

$$\begin{tikzcd}
 A  \arrow[r, "f"]\arrow[d,"\eta_A"'] & B \arrow[d, "\eta_B"] \\
TA\arrow[r, "Tf"']& TB 
\end{tikzcd}$$
The maps $f$, $Tf$ and $\eta_B$ are isomorphisms, therefore $\eta_A$ is an isomorphism. It follows that $A$ is a $T$-algebra.

\item The monad $T$ is idempotent. By Proposition \ref{pp1}.9, the forgetful functor $U:\mc^T \ra \mc$ is fully faithful. By Proposition \ref{pp1}.7, any object of $\mc$ admits at most one algebra structure. Therefore  $U$ is injective on objects. The category $\mc^T$ may then be identified to a subcategory of $\mc$.  The fact that $\mc^T$ is a replete  follows from the fact that properties $1.$ and  $3.$ are equivalent.

	\end{itemize}
 
	\end{proof}
  Let $G:\ma \ra \mb$ be a functor and assume that $G$ has a pointwise right Kan extension $R$ along itself  with counit $\epsilon: RG \ra G$.
Then one has a diagram
\begin{center}
\begin{tikzcd}[row sep=1.5cm, column sep=1.5cm]
\ma  \ar[dr, "G"', ""{name=K,near end}]
            \ar[rr, "G", ""{name=X, below,bend right}]&& \mb\\
& \mb    \ar[ur, bend left, "R", ""{name=R, below}]
                \ar[ur, bend right, "1_{\mb}"', ""{name=R'}]
%
\arrow[Rightarrow, "\eta"', from=R', to=R]
\arrow[Leftarrow, from=X, "\epsilon"']
\end{tikzcd}
\end{center}
where  $\eta: 1_{\mb}  \ra R$ is the unique natural transformation rendering commutative the diagram 
\begin{equation}\label{peq1}
\begin{tikzcd}
G \arrow[rr, Rightarrow, "1_G"] \arrow[dr, Rightarrow,"\eta G"']
& &G \\
& RG \arrow[ur , Rightarrow ,"\epsilon"']
\end{tikzcd}
\end{equation} 
\indent Let $B\in \mb$, $D^B:B/G\ra\ma$ be the functor which takes an arrow-object $ B\st{\sigma}\sra G(A)$ to  $A$ and  $G^B$ be the composite functor
	\begin{equation}\label{peq22}
B/G \st{D^B}\ra\ma \st{G}\ra \mb.
\end{equation}
Then by \cite[Theorem 1 page 237]{ML},
\begin{equation}\label{peq26}
 R(B)=\lim G^B.
\end{equation}
By  \cite[(6), page 238]{ML}, the unit
\begin{equation}\label{peq24}
 \eta_B: B \ra R(B)
\end{equation}
is the map induced by the cone
\begin{equation}\label{peq23}
 B\st{\lambda^B}\Ra G^B 
\end{equation}
whose component $\lambda^B_{\sigma}$ along an arrow-object $B\st{\sigma}\sra G(A)$ is the map $\sigma: B \ra G(A)$. 
By the universal property of $(R,\epsilon)$, there exists a unique natural transformation $\mu:R^2 \ra R$ rendering commutative the diagram
$$\begin{tikzcd}
 R^2G\arrow[r, Rightarrow ,"R\epsilon"]\arrow[d, Rightarrow ,"\mu G"']& RG \arrow[d, Rightarrow ,"\epsilon"]\\
RG \arrow[r, Rightarrow ,"\epsilon"'] & G 
\end{tikzcd}$$
Then the triple $(R,\eta,\mu)$ is a monad called the codensity monad of the functor $G$.\\
 
  \noindent Assume that a functor $G:\ma \ra \mb$ has an idempotent codensity monad $(R,\eta,\mu)$. Then as explained in the proof of the final statement of Proposition \ref{pp2}, the forgetful functor $U:\mb^R \ra \mb$ is fully faithful and injective on objects. The category $\mb^R$ is then identified to its image by $U$, which is, by Proposition \ref{pp2}, a replete subcategory of $\mc$.  

\begin{examples}\label{pex7} \mbox{}
\begin{enumerate}
 
\item The unit of a monoidal category is a unital associative monoid.

\item Let $\mb$ be a category. The trivial monod $I_{\mb}$ on $\mb$ is the unit of the monoidal category of endofunctors  $\mb^{\mb}$. It is the identity functor $1_{\mb}$ with the identity natural transformation of $1_{\mb}$ as its unit and its multiplication, and it is idempotent. The trivial comonad on $\mb$ is dually defined. 

\item Clearly, a functor $G:\ma \ra \mb$ is codense iff the trivial monad $I_{\mb}$ is a codensity monad of $G$ \cite[Proposition 1 page 246]{ML}.         
\end{enumerate}
\end{examples}

\begin{theorem}\label{pt11}
 Let $G:\ma \ra \mb$ be a fully faithful functor which has an idempotent  codensity monad $(R,\eta,\mu)$. Then 
\begin{enumerate}
 
 \item The functor $G$ takes values in the subcategory of $R$-algebras. That is,  $G(\ma) \subset \mb^R$.
  \item The functor $G_0: \ma \ra  \mb^R$ induced by $G$ is a codense functor.
\end{enumerate}
\end{theorem}
    
\begin{proof}
 Let $\epsilon: RG \ra G$ be the counit of the pointwise right Kan extension $R$ of $G$ along itself.  
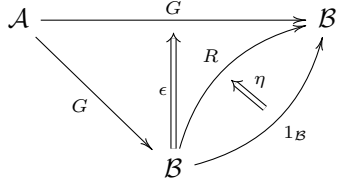
\begin{center}
\begin{tikzcd}[row sep=1.5cm, column sep=1.5cm]
\ma  \ar[dr, "G"', ""{name=K,near end}]
            \ar[rr, "G", ""{name=X, below,bend right}]&& \mb\\
& \mb    \ar[ur, bend left, "R", ""{name=R, below}]
                \ar[ur, bend right, "1_{\mb}"', ""{name=R'}]
%
\arrow[Rightarrow, "\eta"', from=R', to=R]
\arrow[Leftarrow, from=X, "\epsilon"']
\end{tikzcd}
\end{center}

\begin{enumerate}
 
 \item The functor $G$ is fully faithful. By \cite[Corollary 3, page 239]{ML}, $\epsilon$ is an isomorphism. Let $A\in \ma$, by (\ref{peq1}), the composite 
$$G(A) \xra{\eta_{G(A)}}  RG(A)\st{\epsilon_A}\ra G(A)$$
is $1_{G(A)}$. It follows that $\eta_{G(A)}$ is an isomorphism. By Proposition \ref{pp2},  $G(A) \in \mb^R$.
\item Let $B$ be an $R$-algebra and $G_0: \ma \ra  \mb^R$ be the functor induced by $G$. Let
 $$
\begin{array}{lcr}
  G^B:B/G \ra \mb & \mbox{ and } & G_0^B:B/G_0 \ra \mb^R  
 \end{array}
$$
 be as defined by (\ref{peq22}). Moreover, let 
$$
\begin{array}{lcr}
B\st{\lambda^B}\Ra G^B   & \mbox{ and } & B\st{\lambda_0^B}\Ra G_0^B   
 \end{array}
$$
be as defined by (\ref{peq23}). As explained above, $\lim G^B=R(B)$ and the map $B\ra \lim G^B$
 induced by the cone $\lambda^B$ is just the unit $\eta_B:B\ra R(B)$, which is an isomorphism by Proposition \ref{pp2}. It follows that $\lambda^B$ is a limiting cone.\\
The category  $B/G_0$ is isomorphic to $B/G$ and may be identified with it. The functor $G^B$ factors through $G_0^B$ as follows
	\begin{center}
 \begin{tikzcd} 
& B/G_0=B/G \arrow[dl, "G^B_0"'] \arrow[dr, "G^B"] & \\
\mb^R \arrow[rr, "U"']& & \mb
\end{tikzcd}
\end{center}
By Proposition \ref{pp13}.1.(a), the functor $U$ creates limits. In particular, any cone setting above a limit cone is itself a limit cone (see \cite[page 90]{RE} ).
 We have $U(\lambda_0^B)=\lambda^B$. Thus the cone $\lambda_0^B$	is a limiting cone. It follows that $G_0$ is a codense functor.

\end{enumerate}
 
\end{proof}
The notions of comonads, idempotent comonads, coalgebras over them and density comonads  are dually defined and satisfy the appropriate dual properties.

 
\section{Left and right Kan extendable subcategories and their properties}\label{s4}

 In this section, we introduce the key notion of left Kan extendable subcategories and investigate some of its properties. We conclude the section by briefly introducing the dual notion of right Kan extendable subcategories. 

\begin{definition}\label{pd8}\footnote{The author is greatly grateful to Richard Garner for helping him introduce this final form of the definition of Kan extendability. }
A subcategory $\mw$ of a category $\mc$ is said to be left Kan extendable provided that:
\begin{enumerate}
 
\item The inclusion functor $J: \mw \ra \mc$ has a density comonad $(L,\epsilon,\delta)$.
\item The comonad $(L,\epsilon,\delta)$ is idempotent. 
\end{enumerate}
\end{definition}
Let $\mw$ be a left Kan extendable subcategory of $\mc$ and let $(L,\epsilon,\delta)$ be the idempotent density comonad of the inclusion functor 
$J:\mw \ra \mc$.  The category  of $L$-coalgebras is denoted by $\mw_l[\mc]$. It is, by the dual of Proposition \ref{pp2},  a replete subcategory of $\mc$ and is called the subcategory of $\mw$-generated objects of $\mc$.

\begin{examples}\label{pex1} We here give examples of left Kan extendable subcategories. 
\begin{enumerate}
 
\item  Let $\Ab$  be the category of abelian groups. The subcategory $\Fin$  of $\Ab$ of finite abelian groups is left Kan extendable, $\Fin_l[\Ab]$ is the subcategory $\Tor$ of torsion abelian  groups \cite[page 42]{LS}. The functor  $\Ab \ra \Tor$ which takes an abelian group to its torsion subgroup is a coreflector.

\item Let  $\Po$ the subcategory of the category $\Top$ consisting of just one object which is the one point topological space. The subcategory $\Po$ is left Kan extendable in $\Top$ and 
$\Po_l[\Top]$ is the category $\Dis$ of discrete topological spaces \cite[page 18]{BO}.
Furthermore, the discretization functor $\Top \ra \Dis$ is a coreflector.
\item The simplicial category $\Delta$ has objects $[n] = \{0, 1,..., n\}$, $n \geq 0$. A map  in $\Delta$ is an order preserving function $\alpha : [n] \ra [m]$.
Let $S$ be the category of simplicial sets and let $\Delta^n\in S$ be the standard $n$-simplex. Fix a non-negative integer $n$ and let $\mw_n$ be the (full) subcategory of $S$ whose objects are $\Delta^k$, $k\leq n$. Then $\mw_n$ is left Kan extendable in $S$ and $\mw_n[S]$ is the subcategory $S^n$ of $S$ of simplicial sets of dimension   $\leq n$. Furthermore, the left Kan extension of the inclusion functor $\mw_n \ra S$
along itself is just the functor $n$-skeleton functor $Sk^n: S \ra S$
as defined in \cite[page 11]{JT}. 
\end{enumerate}
\end{examples}

\begin{proposition}\label{pp17} 
Let $\mw$ be a left Kan extendable subcategory of  $\mc$, $(L,\epsilon,\delta)$ the density comonad of the inclusion functor $J:\mw \ra \mc$ and $U:\mw_l[\mc]\ra \mc$ the forgetful functor. Then: 

\begin{enumerate}
 
\item  The subcategory $\mw_l[\mc]$   is the strong coreflective hull of $\mw$ in $\mc$.
 
\item  The free $L$-coalgebra functor $F_{L}: \mc \ra \mw_l[\mc]$ is a coreflector.
	\item  The coreflection $(U \dashv F_{L})$ has  $\epsilon$ as its counit. 
	
\end{enumerate}
\end{proposition}

\begin{proof}
This follows from the duals of Propositions \ref{pp4}, \ref{pp1}, \ref{pp2} and the dual of Theorem \ref{pt11}. 
\qedhere  
\end{proof}

\begin{proposition}\label{pp6} 
Let $\mw$ be a left Kan extendable subcategory of $\mc$. 
\begin{enumerate}
 
	\item The inclusion functor $\mw_l[\mc]\st{U}{\ra}\mc$ creates all colimits that $\mc$ admits.
	\item The subcategory $\mw_l[\mc]$ has all limits that $\mc$ admits formed by applying the  coreflector $F_{L}$ to the limit in $\mc$.
	\end{enumerate}
	In particular, if $\mc$ is either complete, cocomplete or bicomplete, then so is $\mw_l[\mc]$.
	\end{proposition}
	   
	\begin{proof} This follows from   Proposition  \ref{pp17} and Proposition \ref{pp13}.2.
	\qedhere  
	\end{proof}

\begin{corollary}\label{pc6} Let $\mw$ be a left Kan extendable subcategory of $\mc$ and $C\in\mc$. Then the following two properties are equivalent:
\begin{enumerate}
 
	\item The object $C$ is $\mw$-generated.
	\item  There exists a functor $F:\mk\ra\mw$ such that $C\cong\colim JF$, where  $J:\mw\hra\mc$ is the inclusion functor.
	\end{enumerate}

\end{corollary}
  
\begin{proof} \mbox{}

\begin{itemize}
 
	\item $2 \Ra 1:$ This follows from  Proposition \ref{pp6}.1. 
	\item  $1 \Ra 2:$ Let $(L,\epsilon,\delta)$ be the density comonad of the inclusion functor $J:\mw \hra \mc$. Define
  $D_C:\mw/C\ra\mw$ to be the functor which associates to an arrow-object $V\ra C$ its domain $V$ and let $J_C$ be the composite functor 
	\begin{equation*} 
	  J_C:\mw/C\st{D_C}{\ra}\mw\st{J}{\hra}\mc
			\end{equation*}
	\end{itemize}
Then, $L(C)\cong \colim J_C$. By the dual of Proposition \ref{pp2}, $\epsilon_C:L(C) \ra C$
is an isomorphism. Therefore $C\cong \colim J_C$.
\qedhere   
\end{proof}

The next result  presents a criteria for the existence of a strong  reflective hull of a subcategory.  

\begin{theorem}\label{pt1}
 Let $\mw$ be a subcategory of a category $\mc$ and $J:\mw \ra \mc$ the inclusion functor.
 The following two properties are equivalent: 
\begin{enumerate}

	\item  The subcategory $\mw$ of $\mc$ is  left Kan extendable.
	\item  The subcategory $\mw$ of $\mc$ has a strong coreflective hull.
	\item The functor $J$ has a
	density comonad  $(L, \epsilon, \delta)$ and the morphism $\epsilon_C : L(C) \ra C$ is
$\mw$-monic for all $C \in  \mc$.
	 
	\end{enumerate}
When these conditions are satisfied, then $\mw_{l}[\mc]$ is the strong coreflective hull of $\mw$. 
\end{theorem}

 
\begin{proof} \mbox{} 

\begin{itemize}

	\item $1 \Ra 2$:   By  Proposition \ref{pp17}.1,    $\mw_{l}[\mc]$  is the strong coreflective hull of $\mw$ in $\mc$.  
	 \item  $2 \Ra 3$:
	Let $\mc^0$ be the strong coreflective hull of $\mw$ in $\mc$, 
	$\mw\st{J^0}\ra \mc^0$   the inclusion functor and $(U\dashv G)$ a coreflection of $\mc$ on $\mc^0$.
 The subcategory $\mw$ is dense in $\mc^0$, thus $J^0$ has a trivial density comonad (dual of Example  \ref{pex7}.3).\\
The functor $1_{\mc^0}:\mc^0 \ra \mc^0$ is a left pointwise Kan extension of $J^0$ along itself.
The functor	$G$ is a right adjoint of $U$, thus by \cite[Proposition 6.5.2]{RE}, $G$ is a left pointwise Kan extension of  $1_{\mc^0}$ along $U$. Therefore   $G$ is a left pointwise Kan extension of $J^0$ along $J=UJ^0$. The functor $U$ is a left adjoint functor, therefore it preserves left pointwise Kan extensions. It follows that $L=UG$ is a density comonad of $J$.

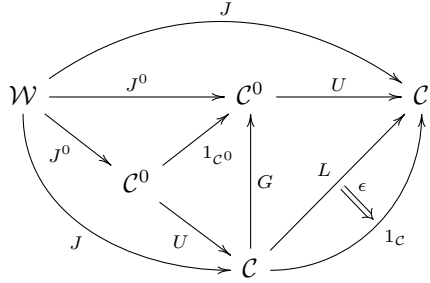
$$\begin{tikzcd}
\mw\arrow[out=35,in=145, looseness=1]{rrrr}{J}\arrow[out=270,in=180, looseness=1]{ddrr}[swap]{J}\arrow{rr}{J^0}\arrow{dr}[swap]{J^0}&&\mc^0\arrow{rr}{U}&&\mc\\
&\mc^0\arrow{dr}[swap]{U}\arrow{ur}[swap]{1_{\mc^{0}}}\\
&&\mc\arrow{uu}[swap]{G}\arrow[uurr,"L",""{name=L}]
\arrow[out=0,in=270, looseness=1,uurr,"1_{\mc}"',""{name=T}]
\arrow[Rightarrow,from=L, to=T,shorten >=-1.5mm,shorten <=1mm,"\epsilon"'swap]
\end{tikzcd}$$\\

	Let $\epsilon$ be the counit of the comonad $L$. By Proposition \ref{pp17}.3, $\epsilon:L=UG \Ra 1_{\mc}$  is the counit of the coreflection $(U,G)$.
	By Lemma \ref{pl4}.2, $\epsilon_C : L(C) \ra C$ is
$\mw$-monic for all $C \in  \mc$.

\item  $3 \Ra 1$:	
For $C\in\mc$, let $\mw/C$ be the  subcategory of the over category $ \mc/C$ whose objects are arrows $V\ra C$ with  domain $V\in\mw$. Define
  $D_C:\mw/C\ra\mw$ to be the functor which associates to an arrow-object $V\ra C$ its domain $V$ and let $J_C$ be the composite functor 
	 \begin{equation*}
J_C:\mw/C\xra{D_C}\mw\st{J}{\ra}\mc\\
\end{equation*}

The functor $J$ has a density comonad  $(L, \epsilon, \delta)$.	
Therefore by the dual of \cite[Theorem 3 page 244]{ML},

\begin{equation*}
\forall C\in \mc,\quad  \colim J_C  \mbox{ exists and }       L(C)=\colim J_C.\\
\end{equation*}

	A morphism $C\st{f}{\ra} C'$ in $\mc$ induces a functor
$\mw/C \xra{\mw/f} \mw/C'$ rendering commutative the diagram

$$
\xymatrix{
\mw/C \ar[rr]^{ \mw/f} \ar[rd]_{J_{C}}   &       &\mw /C'\ar[ld]^{J_{C'}}     \\
                        &\mc
}
$$
This last diagram induces a map $$ \colim J_C\ra\colim J_{C'}$$ which is just
 $$L(f):L(C)\ra L(C').$$
Let $\eta:J\Ra LJ$ be the unit of the left Kan extension $L$ of $J$ along itself.
The functor $J$ is fully faithful.   By the dual of \cite[Corollary 3 page 239]{ML}, $\eta$  is an isomorphism.
We may therefore assume that 
$$L(V)=V  \mbox{ and }  \eta_{V}=1_{V}: V\ra V ,\quad    \mbox{ for all }  V\in \mw.$$
 In which case, by the commutativity of the diagram which is  dual to (\ref{peq1}),
$$\epsilon_{V}=1_{V}: V\ra V ,\quad \mbox{ for all }  V\in \mw.$$
Therefore by the naturality of $\epsilon$, for each $C\in \mc$ and each arrow-object  $\sigma: V\ra C$ in $\mw/C$, there exists a map $\tilde{\sigma}:V\ra L(C)$ rendering commutative the diagram 
$$\begin{tikzcd} 
 & V \arrow[dr, "\sigma"]\arrow[dl,dashed,"\exists !\tilde{\sigma}"'] \\
L(C)  \arrow[rr,"\epsilon_C"'] && C
\end{tikzcd} 
$$
Moreover, $\tilde{\sigma}$ is unique since $\epsilon_C$ is $\mw$-monic.
It follows that the functor  
 $$ \mw/L(C)\xra{\mw/\epsilon_C}\mw/C$$ is an isomorphism.
The following diagram commutes
$$
\xymatrix{
 \mw/L(C) \ar[rr]^{\mw/\epsilon_C} \ar[rd]_{J_{L(C)}}  && \mw/C \ar[ld]^{J_C}     \\
                        &\mc
}
$$
Therefore the map $L(\epsilon_C): L^2(C) \ra L(C)$ is an isomorphism.
By the dual of Proposition \ref{pp1}, $L$ is an idempotent comonad and $\mw$ is left Kan extendable in $\mc$.

\end{itemize}
\qedhere  
\end{proof}

\begin{example}\label{pex3}   The subcategory $\ZZ$ of $\Vt$ of Example \ref{pex2} has no strong coreflective hull. By Theorem \ref{pt1}, $\ZZ$ is not left Kan extendable. 
\end{example}

\begin{corollary}\label{pc4}
Let  $\mw$ be a left Kan extendable subcategory of a category $\mc$ and assume that $\mc^0$  is a replete coreflective subcategory of $\mc$ containing $\mw$. Then
  $\mw$ is left Kan extendable as a subcategory $\mc^0$.
   Furthermore, $\mw_l[\mc^0]= \mw_l[\mc]$.
 
\end{corollary}
    
\begin{proof} By Theorem \ref{pt1}, $\mw_l[\mc]$ is the coreflective hull of $\mw$. The subcategory $\mc^0$ is a replete, coreflective subcategory of $\mc$ containing $\mw$, thus $\mw_l[\mc]$ is a subcategory of $\mc^0$, which is replete coreflective as a subcategory of $\mc^0$,  in which $\mw$ is dense. By Theorem \ref{pt3}.2,  $\mw_l[\mc]$ is the strong coreflective hull of $\mw$ in $\mc^0$. By Theorem \ref{pt1}, $\mw$ is left Kan extendable in $\mc^0$ and  $\mw_l[\mc^0]= \mw_l[\mc]$.
\qedhere   
\end{proof}

\begin{corollary}\label{pc5}
Let $\mw$, $\mw'$ be left Kan extendable subcategories of $\mc$.
\begin{enumerate}

\item If  $\mw' \subset \mw_l[\mc]$, then 
 $\mw'_l[\mc]$ is a coreflective subcategory of $\mw_l[\mc]$.
\item If $\mw' \subset \mw_l[\mc]$ and $\mw \subset \mw'_l[\mc]$, then 
 $\mw_l[\mc]=\mw'_l[\mc]$.
\end{enumerate}
\end{corollary}
   
\begin{proof}\mbox{}
\begin{enumerate}

\item  By Corollary \ref{pc4}, $\mw'$ is left Kan extendable as a subcategory of $\mw_l[\mc]$ and  $\mw'_l[\mw_l[\mc]]=\mw'_l[\mc]$.
By Proposition \ref{pp17}, $\mw'_l[\mc]$ is a coreflective subcategory of $\mw_l[\mc]$. 
\item This follows from 1.
\end{enumerate}
\qedhere    
\end{proof}

 \begin{theorem}\label{pt4}
Let  $\mc_{0}$ be a reflective subcategory of $\mc$, $\mw$ a left Kan extendable subcategory of $\mc$ contained in $\mc_{0}$ and $(U\dashv F_{L})$ the coreflection of $\mc$ on $\mw_l[\mc]$ given by  Proposition \ref{pp17}. Assume further that $F_{L}(\mc_0) \subset \mc_0$. Then:
\begin{enumerate}
 
\item  The subcategory $\mw$ is left Kan extendable as a subcategory of $\mc_0$.
\item  A reflection  of $\mc$ on $\mc_{0}$ induces a reflection of $\mw_l[\mc]$ on $\mw_l[\mc_0]$.
 \item  We have $\mw_l[\mc_0]= \mc_0\cap\mw_l[\mc] $.
 \item  The coreflection $(U\dashv F_{L})$ of $\mc$ on $\mw_l[\mc]$ induces a coreflection of $\mc_0$ on $\mw_l[\mc_0]$. 
 
 \end{enumerate}
\end{theorem}
 
\begin{proof} 
Let $J:\mw \ra \mc$, $J_0:\mw \ra \mc_0$ be the inclusion functors and $(F\dashv V)$ a reflection of $\mc$ on $\mc_{0}$. We may, by Lemma \ref{pl1}.1, assume that  the composite 
  $ \mc_0 \st{V}\ra \mc \st{F}\ra  \mc_0$  is the identity $1_{\mc_0}$ functor. We then have $FJ=J_0$.  
\begin{enumerate}
 
\item Let  $(L,\epsilon,\delta)$ be a  density comonad of $J$. The subcategory $\mc_0$ of $\mc$ contains $\mw$, thus $J$ has a pointwise left Kan extension along $J_0:\mw \ra \mc_0$ which is $L_{/\mc_0}$. The functor $F$ is  left adjoint, it is then cocontinuous and therefore preserves left pointwise Kan extensions. Thus $J_0=FJ$ has a  pointwise left Kan extension along itself which is $L_0=FL_{/\mc_0}$. We have 
 $L_{/\mc_0}(\mc_0)=L(\mc_0)=F_{L}(\mc_0)\subset \mc_0$ and $F_{/\mc_0}=1_{\mc_0}$, therefore $L$ induces an endofunctor of $\mc_0$ which is simply the functor $L_0$.
Let $\epsilon_0:L_0\ra 1_{\mc_0}$ be the natural transformation induced by $\epsilon$. Clearly, $\epsilon_0$
is the counit of the density comonad $L_0$. By Theorem \ref{pt1}, $\epsilon$ is $\mw$-monic, hence $\epsilon_0$ is $\mw$-monic. Again, by Theorem \ref{pt1},
  $\mw$ is left Kan extendable as a subcategory of $\mc_{0}$.
$$
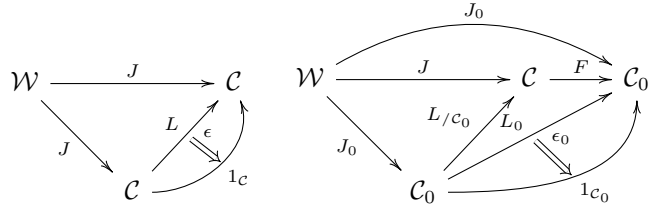
\begin{matrix} 
\begin{tikzcd}[row sep=large]
\mw\arrow{rr}{J}\arrow{dr}[swap]{J}&&\mc\\
&\mc\arrow[ur,"L",""{name=L}]\arrow[out=0,in=-70, looseness=1,ur,"1_{\mc}"',""{name=T}]
\arrow[Rightarrow,from=L, to=T,shorten >=-2mm,shorten <=1mm,"\epsilon"'swap]
\end{tikzcd}
 &\begin{tikzcd}[row sep=large]
\mw\arrow[out=30,in=150, looseness=1]{rrr}{J_0}\arrow{rr}{J}\arrow{dr}[swap]{J_0}&&\mc\arrow{r}{F}&\mc_0\\
&\mc_0\arrow{ur}{L_{/\mc_0}}\arrow[urr,"L_0",""{name=L}]\arrow[out=0,in=-90, looseness=1,urr,"1_{\mc_0}"',""{name=T}]
\arrow[Rightarrow,from=L, to=T,shorten >=-2mm,shorten <=1mm,"\epsilon_0"'swap]
\end{tikzcd}
\end{matrix}$$

\item We just need to prove that $\mw_l[\mc_0]\subset \mw_l[\mc]$ and 
$F(\mw_l[\mc]) \subset \mw_l[\mc_0]$. Let $X_0\in \mw_l[\mc_0]$.
By the dual of Proposition \ref{pp2},  $\epsilon_{X_0}=(\epsilon_0)_{X_0}$ is an isomorphism,
therefore $X_0 \in \mw_l[\mc]$ and $\mw_l[\mc_0]\subset \mw_l[\mc]$. 
Let $X \in \mw_l[\mc]$ and $J_X:\mw/X \ra \mc$ be the functor which takes an arrow-object $V \sra X$ in $\mw/X$ to its domain $V$. The functor $F$ preserves colimits, thus $\colim FJ_X\cong F(X)$. The functor $FJ_X$ takes values in $\mw$, by Corollary \ref{pc6}, $\colim FJ_X \in \mw_l[\mc_0]$. It follows that $F(X) \in \mw_l[\mc_0]$ and $F(\mw_l[\mc]) \subset \mw_l[\mc_0]$.

\item We have $\mw_l[\mc_0]\subset \mw_l[\mc]$, thus $\mw_l[\mc_0]\subset \mc_0\cap\mw_l[\mc] $. The induced functor $F_{/\mc_0}=1_{\mc_0}$, thus
$\mc_0\cap\mw_l[\mc]=F(\mc_0\cap\mw_l[\mc])\subset F(\mw_l[\mc])\subset \mw_l[\mc_0]$. Therefore $\mw_l[\mc_0]= \mc_0\cap\mw_l[\mc] $.

\item One has $F_{L}(\mc_0)\subset \mc_0\cap \mw_l[\mc]=\mw_l[\mc_0]$. Thus $ F_{L} (\mc_0)\subset \mw_l[\mc_0]$ and the result follows. 
	
\end{enumerate}
 \qedhere   
\end{proof}
We next introduce the dual notion of right Kan extendable subcategories.

\begin{definition}\label{pd2}  
A subcategory $\mw$ of a category $\mc$ is said to be right Kan extendable provided that:
\begin{enumerate}
 
\item The inclusion functor $J: \mw \ra \mc$ has a codensity  monad $(R,\eta,\mu)$.
\item The monad $(R,\eta,\mu)$ is idempotent. 
\end{enumerate}
\end{definition}

Examples of right Kan extendable subcategories are given in the next section.\\
\indent Let $\mw$ be a right Kan extendable subcategory of a category $\mc$ and $(R,\eta,\mu)$ the codensity monad of the inclusion functor $J:\mw \ra \mc$.  Define $\mw_r[\mc]$ to be the category of $R$-algebras. Then by proposition \ref{pp2}, $\mw_r[\mc]$ may be  viewed as a replete subcategory of $\mc$. It is called the subcategory of $\mw$-cogenerated objects of $\mc$. 
\begin{corollary}\label{pc19}
Let $\mw$ be a right Kan extendable subcategory of  $\mc$, $(R,\eta,\mu)$ the codensity monad of the inclusion functor $J:\mw \ra \mc$ and $U:\mw_r[\mc]\ra \mc$ the forgetful functor. Then 

\begin{enumerate}
 
\item  The subcategory $\mw_r[\mc]$  is  the strong reflective hull of $\mw$ in $\mc$.
 \item The free $R$-algebra functor $F^{R}: \mc \ra \mw_r[\mc]$ is a  reflector.

	\item  The reflection $(F^{R} \dashv U)$ has  $\eta$ as its unit. 
	
\end{enumerate}
\end{corollary}
    
\begin{proof}
This is the dual of Proposition  \ref{pp17}.  
\qedhere   
\end{proof}
 
\section{Compactifications}\label{s5}
   Stone–Čech compactification and Hewitt realcompactification are  procedures  exhibiting the subcategories of compact Hausdorff  and realcompact spaces as  reflective subcategories of $\Top$. Our objective in this section is to show how can these two facts be established using the notion of Kan extendable subcategories. We begin with the following technical result.\\
	
\begin{lemma}\label{pl5} 
Let $\mc$ be a complete category, $\mc_0$ is a subcategory of $\mc$ and   $J:\mc_0 \ra \mc $ the inclusion functor. Assume that for any small category $\mi$ and any functor $F:\mi \ra \mc_0$, the limit of the composite functor  $JF$ is in  $\mc_0$. Then: 
\begin{enumerate}
 
\item  The subcategory $\mc_0$ is complete.
\item The functor  $J:\mc_0 \ra \mc $ preserves and creates  small limits.
\end{enumerate}
\end{lemma}
 
 Observe that such a subcategory $\mc_0$ is necessarily replete.
  
\begin{proof} Clear.
\qedhere  
\end{proof}

\begin{theorem}\label{pt5} 
Let $\mc$ be a complete category, $\mw$ a small subcategory of $\mc$,   $\mc_0$ a subcategory of $\mc$ containing $\mw$ as a codense subcategory and $J_0:\mc_0\ra\mc$ the inclusion functor. Assume further that for any small category $\mi$ and any functor $F:\mi \ra \mc_0$, the limit of the composite functor $J_0F$ is in $\mc_0$. Then $\mc_0$ is the strong reflective hull of $\mw$.
  \end{theorem}
	    
\begin{proof}
The category $\mc$ is complete and $\mw$ is small, therefore the inclusion functor 
$J:\mw\ra\mc$ has a  codensity monad $(R,\eta,\mu)$. By hypothesis, $R(\mc)\subseteq \mc_0$. Moreover, the subcategory $\mw$ is codense in $\mc_0$. Therefore by lemma \ref{pl5},
for each    $X\in \mc, $
\begin{equation}\label{peq28}
 \mbox{ the morphism }  \eta_X:X \ra R(X)  
\mbox{ is an isomorphism iff }
  X\in \mc_0 .
\end{equation}
 
As observed above, the functor $R$ takes $\mc$ into $\mc_0$. Therefore by (\ref{peq28}),  $\eta R:R \ra R^2$ is an isomorphism. By Proposition \ref{pp1}, $R$ is an idempotent monad. It follows  that $\mw$ is right Kan extendable.\\
\indent By Proposition \ref{pp2}, and object $X\in \mc$ has an $R$-algebra structure iff $\eta_X:X \ra R(X)$ is an isomorphism.  Therefore by  (\ref{peq28}), $\mw_r[\mc]=\mc_0$. By the dual of Theorem \ref{pt1}, $\mc_0$ is the strong reflective hull of $\mw$.
\qedhere   
\end{proof}
  As before, let $\Comp$ be the subcategory of $\Top$ of compact Hausdorff spaces and let  
	$I$  the unit interval, $I^2=I\tm_{\Top} I$ and $\Sq$ the subcategory of $\Top$ having $I^2$ as its unique object. The following result strengthens  the standard Stone–Čech compactification 

\begin{corollary}\label{pc7} 
 The subcategory $\Comp$ of $\Top$  is the strong reflective hull of $\Sq$. 

\end{corollary}
   
\begin{proof} 
Let $J:\Comp\ra\Top$ be the inclusion functor,  $\mi$ be a small category  and $F:\mi \ra\Comp$ a functor. The limit of $JF$ is clearly in $\Comp$.
 By (Isbell, \cite[Theorem 2.6]{I1}), $\Sq$ is a codense subcategory of $\Comp$. By Theorem \ref{pt5}, $\Comp$ is the strong reflective hull of $\Sq$. 
\qedhere   
\end{proof}
\begin{remark}\label{pr12}
An algebraic example of a coreflective hull which is not strong is given in Example \ref{pex2}. We next provide another example which is  topological.\\
Let $\U$ be the subcategory of $\Top$ having the unit interval $I$ as its unique object.  The subcategory $\Comp$  of $\Top$ is reflective and contains $\U$.  Let $\Topo$ be a reflective subcategory of $\Top$ containing $\U$. Clearly, $\Topo$ contains $\Sq$, therefore it contains the reflective hull of $\Sq$ which is $\Comp$. It follows that $\Comp$ is precisely the reflective hull of $\U$. By \cite[Theorem 2.6]{I1},  $\U$ is not dense in $\Comp$. Therefore  $\U$ has no strong reflective hull.  
\end{remark}

Let $\Rng$ be the category of commutative rings and let $$C:\Top^{op}\ra \Rng$$ be the functor which takes a space $X$ to the ring of real-valued continuous maps defined  on $X$. Recall that a topological space is said to be realcompact if it is homeomorphic to a closed subspace of a product of real lines  \cite[11.12]{GJ}.
  Let $\Rcomp$ be the subcategory of $\Top$ of realcompact spaces. 

\begin{theorem}\label{pt6}(\cite[Theorem 10.6]{GJ})\\ 
 The restriction functor $C_{/}:\Rcomp^{op}\ra \Rng$  of $C$ is fully faithful. \\ 
 \end{theorem}

Let $\Q$ be the subcategory of $\Rcomp$ having precisely one object which is $\R^2=\R\tm_{\Top} \R$. 

\begin{theorem}\label{pt7}
 
 The subcategory $\Q$ of $\Rcomp$ is codense.
 \end{theorem}
    
\begin{proof}
The proof is based on Theorem \ref{pt6}, and is strictly similar to  Isbell's proof of the fact that  
  $\Sq$ is  codense in $\Comp$ \cite[Theorem 2.6]{I1}.
	\qedhere    
\end{proof}

The following result strengthens  the standard Hewitt Realcompactification. 

\begin{corollary}\label{pc8} 
The subcategory $\Rcomp$ of $\Top$  is the strong reflective hull of $\Q$.
 
\end{corollary}
   
\begin{proof}
$\Top$ is complete and $\Rcomp$ is a subcategory $\Top$ containing $\Q$ as a codense subcategory. The product in $\Top$ of a small set of realcompact spaces is realcompact. Similarly, the equalizer in $\Top$ of two parallel maps in $\Rcomp$ is again in $\Rcomp$. Therefore by Theorem \ref{pt5}, 
$\Rcomp$ is the strong reflective hull of $\Q$. 
\qedhere  
\end{proof}
 
\section{Reflective subcategories of $\Top_B$}\label{s6}
 In this section, we apply the theory developed previously to prove that  subcategories of fibrewise topological spaces over $B$ satisfying certain separation axioms are reflective subcategories of $\Top_B$.

Recall that a subcategory $\ma$ of a category $\mc$ is said to be closed under subobjects if whenever we have a monomorphism $X \ra Y$ in $\mc$  with codomain $Y\in \ma$, then $X$ is isomorphic to an object of $\ma$.
Observe that the next theorem may also be derived from \cite[Theorem 16.8]{AH}.\\

\begin{theorem}\label{t3} 
Let $\Topo_B$ be a  subcategory of $\Top_B$ such that: 
\begin{enumerate}
 
\item $\Topo_B$ is replete and contains the fibrewise topological space  $B$ (over itself).
\item  $\Topo_B$ is closed under subobjects as a subcategory of $\Top_B$. 
\item For every family $(V_i)_{i\in I}$ of objects of $\Topo_B$ (indexed by a small set $I$), the product $\underset{\Top_B}{\overset{i\in I}{\prod}}V_i$ is an object of $\Topo_B$.
 \end{enumerate}
Then $\Topo_B$ is a reflective subcategory of $\Top_B$. In particular, 
$\Topo_B$ is bicomplete. Furthermore, the unit $\eta$ of this reflection is such that the maps $\eta_X: X\sra R(X)$ are quotient maps,   where $R:\Top_B \ra \Topo_B $ is a reflector. 
\end{theorem}
    
\begin{proof} Let $X\in \Top_B$ and let $J^X:X/\Topo_B \ra \Top_B$ be the functor which takes an arrow-object $X \ra V$ to its codomain $V$.
Define $\mr_{X}$ to be the equivalence relation on $X$ given by 
$x_1\mr_{X}x_2$   iff  $f(x_1)=f(x_2)$  for every continuous  fibrewise map $f$ from $X$ to any fibrewise topological space in $\Topo_B$. The projection $p_X:X\ra B$ defines a continuous fibrewise map from $X$ to $B$ as follows:
$$\begin{tikzcd}
X\arrow[dr, "p_X"' ] \arrow[rr, "p_X" ]{}
& & B\arrow[dl, "1_B" ]  \\
& B    
\end{tikzcd}$$
Therefore if $x_1\mr_{X}x_2$, then  $p_X(x_1)=p_X(x_2)$. It follows that the projection $p_X$ factors through $X/\mr_{X}$ as follows: 
$$\begin{tikzcd}
X\arrow[dr, "p_X"' ] \arrow[rr,  "\eta_X" ]{}
& & X/\mr_{X}\arrow[dl, "\widetilde{p}_X" ]  \\
& B    
\end{tikzcd}$$
In other words, $\mr_{X}$ is a fibrewise equivalence relation on $X$,
thus $X/\mr_{X}$ is a fibrewise topological space over $B$ and
 the quotient map $\eta_X:X\ra X/\mr_{X}$ is a fibrewise map. 
Define $A_X= \{\{x_1,x_2\}\subset X |\; p_X(x_1)=p_X(x_2) \text{ and }  x_1 \cancel{\mr}_X x_2 \}$     and let $(f_i)_{i\in I}$ be a family of maps in $\Top_B$ such that: 
\begin{itemize}
 
\item  $f_i:X \ra V_i$, where $V_i\in \Topo_B$  for all $i\in I$.
\item   $I$ is small and nonempty.
\item  For each $\{x_1,x_2\}\in A_X$, there exists $i\in I$ such that $f_i(x_1) \neq f_i(x_2)$.
 \end{itemize}
 Define $f:X \ra  \underset{\Top_B}{\overset{i\in I}{\prod}}V_i$ to be the map whose $i$-component is $f_i$. Observe that 
$f(x_1)= f(x_2) \Leftrightarrow x_1 \mr_{X} x_2$. Thus there exists a unique continuous fibrewise map  $\tilde {f}: X/\mr_{X}\ra \underset{\Top_B}{\overset{i\in I}{\prod}}V_i$ rendering commutative the diagram
$$
\begin{tikzcd}
X \arrow[r,"f"]\arrow[d,"\eta_X"']& \underset{\Top_B}{\overset{i\in I}{\prod}}V_i\\
X/\mr_{X}\arrow[ru,"\tilde{f}"']&  
\end{tikzcd}
$$
$\underset{\Top_B}{\overset{i\in I}{\prod}}V_i \in \Topo_B$, $\tilde{f}$ is monic and $\Topo_B$ is closed under subobjects. Therefore  $X/\mr_{X}\in \Topo_B$ and the arrow $X \st{\eta_X}\ra X/\mr_X$ is an object of $X/\Topo_B$ which is initial. Then clearly,  $\lim J^X\cong X/\mr_X \in \Topo_B$  exists. It follows that the inclusion functor $ \Topo \st{J}\ra \Top_B$ has a  codensity monad $R$  given by $R(X) \cong X/\mr_{X}\in \Topo_B$, with unit the natural transformation $\eta: 1_{\Top}\Ra R$ whose component along $X$ is the quotient map $X \st{\eta_X}\ra X/\mr_X$ which is epic. By the dual of Theorem \ref{pt1}, $\Topo_B$ is right Kan extendable in $\Top_B$. The codensity monad $R$ of the inclusion functor $ J:\Topo_B \ra \Top_B$  takes values in  $\Topo_B$ which is replete. By Proposition \ref{pp2}, $\Topo_B[\Top_{B}]\subset \Topo_B$. If $X\in \Topo_B$, then $\mr_X$ is the trivial equivalence relation and $\eta_X=1_X.$  Therefore by Proposition \ref{pp2}.3, $X\in \Topo_B[\Top_{B}]$ and $\Topo_B \subset \Topo_B[\Top_{B}]$. It follows that $\Topo_B[\Top_{B}]=\Topo_B$. By Corollary \ref{pc19}, $\Topo_B$ is a reflective subcategory of $\Top_B$ with reflector $F^{R}: \Top_B \ra \Topo_B$ the functor induced by $R$. Furthermore, the reflection of $\Top_{B}$ on $\Topo_{B}$ has unit $\eta$ which is an objectwise quotient map.
\qedhere  
\end{proof}
\begin{examples}\label{ex1}
According to James  \cite[Chapter I, section 2]{J2},  a fibrewise topological space $X$ is said to be fibrewise
\begin{itemize}
 
\item Fréchet (or $\T_1$) if each fibre $X_b$ of $X$ is an ordinary  $\T_1$-topological space. The category of fibrewise Fréchet spaces  is denoted by $\fTop_B$. 

\item  Hausdorff (or $\T_2$) if any two distinct points of $X$ laying in the same fibre can be separated by neighborhoods in $X$. The category of fibrewise Hausdorff spaces is denoted by $\hTop_B$.
\end{itemize}
Observe that if $X$ is a fibrewise $T_i$-space over $B$, $i=1,2$ and $B$ is an ordinary $T_i$-space, then $X$ is a $T_i$-space in the ordinary sense.\\
Similarly, define a fibrewise topological space $X$ to be fibrewise
\begin{itemize}
 
\item Urysohn space (or $\T_2\frac{1}{2}$)  if any two distinct points of $X$ laying in the same fibre can be separated by closed neighborhoods in $X$. The category of fibrewise Urysohn spaces is denoted by $\uTop_B$.

\item completely Hausdorff space (or functionally Hausdorff space) if any two distinct points of $X$ laying in the same fibre can be separated by a continuous function (or equivalently, by a continuous fibrewise map into $B\tm_{\Top} \R$).  The category of fibrewise completely Hausdorff spaces is denoted by $\hcTop_B$.  
 \end{itemize}
By Theorem \ref{t3},  the categories $\fTop_B$, $\hTop_B$, $\uTop_B$ and $\hcTop_B$ are reflective subcategories of $\Top_B$.  
\end{examples}

 A one point space $\Pt$ is a terminal object of $\Top$. Therefore one has  the standard isomorphism
\begin{equation}\label{peq16}
 P: \Top_{\Pt} \ra \Top. 
\end{equation}
 By substituting $\Pt$ for $B$,  Theorem \ref{t3}  reduces to the following. 
	
\begin{corollary}\label{pc12} 
Let $\Topo$ be a  subcategory of $\Top$ such that: 
\begin{enumerate}
 
\item $\Topo$ is replete and contains a nonempty space.
\item  $\Topo$ is closed under subobjects as a subcategory of $\Top$. 
\item For every family $(V_i)_{i\in I}$ of objects of $\Topo$ (indexed by a small set $I$), the product $\underset{\Top}{\overset{i\in I}{\prod}}V_i$
  is an object of $\Topo$.
 \end{enumerate}
Then $\Topo$ is a reflective subcategory of $\Top$. In particular, 
$\Topo$ is bicomplete. Furthermore, the unit $\eta$ of this reflection is such that the map $\eta_X: X\ra R(X)$, $X\in \Top$, is a quotient map, where $R:\Top \ra \Topo $ is the reflector. 
\end{corollary}

\begin{examples}\label{pex5}
 A space $X\in \Top$ is Fréchet (resp. Hausdorff,  Urysohn, completely Hausdorff) if it corresponds, under the isomorphism $P$ of (\ref{peq16}), 
to a fibrewise Fréchet (resp. Hausdorff,  Urysohn, completely Hausdorff) space over $\Pt$. The subcategory of $\Top$ of such spaces is reflective and is denoted by $\fTop$ (resp. $\hTop$, $\uTop$, $\hcTop$).   
\end{examples}

\section{Fibrewise compact spaces}\label{s9}

Let $\mw$ be a left Kan extendable subcategory of a category $\mc$. One of the main objectives of this paper is to present sufficient conditions, under which, the category $\mw_l[\mc]$ is cartesian closed. Among other conditions, one requires that the objects of $\mw$ be exponentiable as objects of the category $\mc$.
To be able to apply this result to prove that the category of fibrewise compactly generated spaces over a $\T_1$-base $B$  is cartesian closed, one then needs to prove that a fibrewise compact fibrewise Hausdorff space over $B$ is an exponentiable object of $\Top_B$. This last result is precisely what this section is after.\\
\indent We begin by introducing the notion of fibrewise compact spaces and recalling   their relevant properties.
The main references of what is discussed here are the books of Bourbaki \cite[Chapter I, Section 10]{BO} and James \cite[Chapter I]{J2}.  \\
\indent Recall that a continuous map $f:X\ra Y$ between two topological spaces $X$ and $Y$ is said to be proper if 
the product map $f\tm_{\Top} 1_Z: X\tm_{\Top} Z\ra Y\tm_{\Top} Z$ is closed for all $Z\in \Top$ \cite[Section 10.1]{BO}. A fibrewise  space $X$ over the fixed topological space $B$ is said to be fibrewise compact if its projection $p:X \ra B$ is a proper map. \\
\indent The next proposition is an immediate consequence of \cite[Proposition 5.b, Section 10.1]{BO}. 
 
\begin{proposition}\label{pp20} 
A continuous map Let $f:X\ra Y$ and $g:Y\ra Z$ be continuous maps. If $g\circ f$ is proper, then the map $f(X) \ra Z$ induced by $g$ is proper. \\ 
\end{proposition}

The next theorem presents a criteria for a continuous map to be proper. 

\begin{theorem}\label{t5} \cite[Theorem 1, Section 10.2]{BO}\\ 
A continuous map $f:X\ra Y$ is proper  iff $f$ is closed and $f^{-1}(y)$ is compact for all $y\in Y$.  
\end{theorem}

\begin{proposition}\label{p10}
 Let $f:X\ra Y$ and $f':X'\ra Y'$ be two continuous fibrewise maps over $B$. Assume that $f$ and $f'$ are proper. Then the map 
$$f\tm_{\Top_B}f': X\tm_{\Top_B} X'\ra  Y\tm_{\Top_B} Y'$$ is proper. 
\end{proposition}
   
\begin{proof}
The maps $f$ and $f'$ are proper. By \cite[Proposition 4, Section 10.1]{BO}, the product map
$f\tm_{\Top}f': X\tm_{\Top} X'\ra  Y\tm_{\Top} Y'$ is proper. The commutative diagram
$$ 
\begin{tikzcd}
X\tm_{\Top_B} X' \arrow[hookrightarrow, r] \arrow[d, "f\tm_{\Top_B} f'"'] & X\tm_{\Top} X' \arrow[d,"f\tm_{\Top} f'"] \\
Y\tm_{\Top_B} Y' \arrow[hookrightarrow,r] & Y\tm_{\Top} Y' 
\end{tikzcd}
$$
is a pullback diagram in $\Top$. By \cite[Proposition 3, Section 10.1]{BO}, the map 
$$f\tm_{\Top_B}f': X\tm_{\Top_B} X'\ra  Y\tm_{\Top_B} Y'$$
is proper.
\qedhere  
\end{proof}
\begin{corollary}\label{c7}
 Let $X$ and $Y$ be two fibrewise compact spaces over $B$.  Then  
$X\tm_{\Top_B} Y$ fibrewise compact.  
\end{corollary}

\begin{proposition}\label{p3}\cite[Proposition 2.7]{J2}\\
A fibrewise space $X$ is fibrewise Hausdorff iff its diagonal $\Delta_X$ is closed in $X\tm_{\Top_B}X$. 
\end{proposition}

\begin{definition}\label{d1}\cite[Definition 2.15]{J2}\\
A fibrewise topological space $p:X\ra B$ is fibrewise regular if for each point $x_0\in X$, and for each open neighborhood $V$ of $x_0$ in $X$, there exist an open neighborhood $\Omega$ of $b_0=p(x_0) $ in $B$ and an open neighborhood $U$ of $x_0$ in $X$ such that  $\overline{U}\cap X_{\Omega} \subset V$. 
  \end{definition}

	\begin{proposition}\label{p4}  \cite[Proposition 3.19]{J2} \\
Let $\phi: K \ra X$ be a continuous fibrewise map, where $K$ is fibrewise compact and $X$ is fibrewise Hausdorff over $B$. Then $\phi$ is a proper map. In particular,
\begin{enumerate}
 
\item  $\phi(K)$ is closed in $X$.
\item  $\phi(K)$ is fibrewise compact fibrewise Hausdorff over $B$. 
\end{enumerate}
\end{proposition}

\begin{corollary}\label{c1}\cite[Corollary 3.20]{J2}\\
A fibrewise compact subspace of a fibrewise Hausdorff space is closed. 
\end{corollary}

\begin{corollary}\label{c4} 
A subspace of a fibrewise compact fibrewise Hausdorff space is fibrewise compact iff it is closed.
\end{corollary}
  
\begin{proof}	
The result  follows from Corollary \ref{c1} and  Theorem \ref{t5}.
\qedhere  
\end{proof}

\begin{proposition}\label{p32}(\cite[Proposition 6 page 104]{BO}) \mbox{}
Let $p:X\ra B$ be a proper map and let $K$ be a compact subspace of $B$, then $p^{-1}(K)$ is a compact subspace of $X$. 
\end{proposition}	

\begin{proposition}\label{p1} \cite[Proposition 3.22]{J2} \\
Every fibrewise compact, fibrewise Hausdorff space over $B$ is fibrewise regular. 
\end{proposition}

The next result reduces to the standard tube lemma \cite[Lemma 26.8]{MJ} in the case where $B$ is a one point space. 

\begin{lemma}\label{l1}(A fibrewise tube lemma)\\
Let $X$ and $K$ be two fibrewise  spaces over $B$ with $K$ fibrewise compact. Let $x_0\in X$, $O$ an open subset of $X\tm_{\Top_B} K$ and assume that $\{x_0\}\tm_{\Top_B} K \subset O$. Then there exists an open neighborhood $V$ of $x_0$ in $X$ such that $V\tm_{\Top_B} K \subset O$.
\end{lemma}
    
\begin{proof} \mbox{}

\begin{itemize}

	\item Case 1: $X=B$ and $x_0=b_0\in B$.\\
	Observe that $B\tm_{\Top_B} K=K$, $\{b_0\}\tm_{\Top_B} K= K_{b_0}$  and $O$ is an open subset of $K$ containing $K_{b_0}$.
	Define $C$ be the closed subset of $K$ given by $C=K\setminus O$.
	The projection $p_K:K \ra B$ is a proper map, it is therefore closed. It follows that $p_K(C)$ is closed and does not contain $b_0$. Define 
	$V=B \setminus p_K(C)$.
	Then clearly, $V$ is an open neighborhood of $b_0$ and  
	$V\tm_{\Top_B} K=p_K^{-1}(V) = K_V \subset O$ as desired.
	
\item Case 2: The general case.\\ 
Let $p_X:X \ra B$ be the projection of the fibrewise space $X$ and let $b_0=p_X(x_0)$. We have $\{x_0\}\tm_{\Top_B} K= \{x_0\}\tm_{\Top_B} K_{b_0}\subset O$.
For every $y\in K_{b_0}$, there exist open neighborhoods $U_y$ of 
$x_0$ in $X$ and $W_y$ of $y$ in $K$ such that $U_y\tm_{\Top_B} W_y \subset O.$ The family $(W_y)_{y\in K_{b_0}}$ is an open cover of $K_{b_0}$ which is compact. There exist $y_1,y_2,\dots,y_n\in K_{b_0}$ such that $K_{b_0}\subset \bigcup^{n}_{i=1} W_{y_i}$. Define $U= \bigcap^{n}_{i=1}U_{y_i}$ and $W=\bigcup^{n}_{i=1} W_{y_i}$. Then $U$ is an open neighborhood of $x_0$, $W$ is an open subset of $K$ containing $K_{b_0}$ and 
	$U\tm_{\Top_B} W\subset O$. By Case 1, there exists an open subset $\Omega$ of $B$ such that
	$K_{\Omega}\subset W$. Define $V= X_{\Omega}\cap U$, then
	$$V\tm_{\Top_B} K=U\tm_{\Top_B} K_{\Omega}\subset U\tm_{\Top_B} W \subset O.$$
\end{itemize}
\qedhere  
\end{proof}
 We next present a special case of the fibrewise compact-open topology  defined in \cite[page 64]{J2},   (see also \cite[page 152]{N}).\\
\indent Let $K, Y \in \Top_B$ with $K$ fibrewise compact, fibrewise Hausdorff space. A subspace of $K$ (or $Y$) may be viewed as a fibrewise space over $B$.  For $\Omega$ open in $B$, $C$ closed in $K$ and $O$ open in $Y$, let 
\begin{equation}\label{eq36} 
(C,O,\Omega)=\underset{\Set}{\overset{b\in \Omega}{\coprod}}\{\gamma \in \Top(K_b,Y_b)  \mid \gamma(C_b)\subset O_b\}.
\end{equation}
Define $\map_B(K,Y)$ to be the topological space whose underlying set is $\underset{\Set}{\overset{b\in B}{\coprod}}\Top(K_b,Y_b)$
and whose topology is generated \footnote{The topology of $\map_B(K,Y)$ is then the coarsest topology on the set $\underset{\Set}{\overset{b\in B}{\coprod}}\Top(K_b,Y_b)$ containing  $(C,O,\Omega)$'s as open subsets.} by the subsets $(C,O,\Omega)$, where  $\Omega$ is open in $B$, $C$ is closed in $K$ and $O$ is open in $Y$. \\ 
\indent Our definition agrees with that of James mentioned above with the difference that in our case,  $\map_B(K,Y)$ is only defined when  $K$ is fibrewise compact, while in \cite{J2},  $\map_B(X,Y)$ is defined for any fibrewise space $X$ in precisely the same way. \\
 Open subsets given by (\ref{eq36}) are called elementary open subsets of $\map_B(K,Y)$. For $b\in B$, let $\map(K_b,Y_b)$ be the subspace of $\map_B(K,Y)$ whose underlying set is $\Top(K_b,Y_b)$\footnote{Observe that if $K_b$ is empty, then $\Top(K_b,Y_b)$ contains precisely one element.}.
Define 
\begin{equation}\label{eq44} 
p_{_{\map_B(K,Y)}}:\map_B(K,Y) \ra B
\end{equation} 
to be the map whose fibre over $b$ is $\map(K_b,Y_b)$. Let $\Omega$ be open in $B$, then
$$p_{_{\map_B(K,Y)}}^{-1}(\Omega)=\underset{\Set}{\overset{b\in B}{\coprod}}\Top(K_b,Y_b)=(K,Y,\Omega)$$
is open in  $\map_B(K,Y)$. It follows that $p_{_{\map_B(K,Y)}}$ is 
continuous. The space $\map_B(K,Y)$ is therefore viewed as a fibrewise space over $B$.  

\begin{example}\label{ex2}

For $b\in B$, let
  $B^b$ be the fibrewise subspace of $B$ having $b$ as its unique point.
Then $B^b$ is fibrewise Hausdorff.  Assume that $B$ is $\T_1$. Then by Theorem \ref{t5}, $B^b$ is fibrewise compact. If $Z\in \Top_B$, then $\map_B(B^b,Z)$ is a fibrewise space over $B$. It is such that  
\begin{equation}\label{eq45} 
\map_B(B^b,Z)_{b'}\cong \left\{ \begin{array}{ll}
Z_b     &\mbox{ if } b'=b \\[2mm]
 \mbox{One point space  } &\mbox{ if }b' \neq b
\end{array} \right.
\end{equation}
\end{example}
  The next proposition is a special case of that of James \cite[Corollary 9.13]{J2}.\\

\begin{proposition}\label{p2}
Let $K$, $Y$ be fibrewise topological spaces over $B$ with $K$ fibrewise compact fibrewise Hausdorff. Then the evaluation map 
$$ev:\map_B(K,Y)\tm_{\Top_B}  K \ra Y$$
is continuous. 
\end{proposition}
   
\begin{proof}
Let $b_0\in B$, $\gamma_0 \in \map(K_{b_0},Y_{b_0})$, $x_0\in K_{b_0}$,  $O$ open in $Y$ and suppose that $\gamma_0(x_0)\in O$. The map
$\gamma_0:K_{b_0}\ra Y_{b_0}$ is continuous, therefore there exists an open neighborhood $V$ of $x_0$ in $K$ such that $\gamma_0(V\cap K_{b_0})\subset O$. The fibrewise space $K$ fibrewise compact, fibrewise Hausdorff, by Definition \ref{d1}, $K$ is regular. There exists an open neighborhood $\Omega$ of $b_0\in B$ and an open neighborhood $U$ of $x_0$ in $K$ such that 
$\overline{U}\cap K_{\Omega} \subset V$.
Define $W=U\cap K_{\Omega}$. Then
$(\overline{U},O,\Omega)\tm_{\Top_B} W$ is a neighborhood of $(\gamma_0,x_0)\in\map_{\Top_B}(K,Y)\tm_{\Top_B} K$ and $ev((\overline{U},O,\Omega)\tm_{\Top_B} W) \subset O$. It follows that $ev$ is continuous.
\qedhere   
	\end{proof}
	
Recall that an object $Y$ in a category $\mc$ is said to be exponentiable if for each $X\in \mc$, the binary product $X\tm_{\mc}Y$ exists and 
  the functor $ .\tm_{\mc} Y:\mc\ra\mc$ has a right adjoint.\\
		\indent The following fact is a consequence of   \cite[Proposition 9.7 and  Corollary 9.13]{J2}  of James. 
	
\begin{theorem}\label{t1}
Let $K$ be a fibrewise compact fibrewise Hausdorff space over $B$. Then the functor $$.\tm_{\Top_B}K: \Top_B \ra \Top_B$$ has a right adjoint which is the functor $$\map_B(K,.): \Top_B\ra \Top_B.$$ In particular, $K$ is an exponentiable object of $\Top_B$. 
 \end{theorem}
   
\begin{proof}
Let $X,Y \in \Top_B$ and  
$f:X\tm_{\Top_B} K \ra Y$ a fibrewise function with adjoint (as a  fibrewise map between sets) the fibrewise function $\widehat{f}:X\ra \map_B(K,Y)$.
We need to prove that $f$ is continuous iff $\widehat{f}$ is.\\
 Assume that 
$f:X\tm_{\Top_B} K \ra Y$ is continuous and let $x_0\in X$, 
  $(C,O,\Omega)$ be an elementary open subset of $\map_B(K,Y)$ and assume that 
	$\widehat{f}(x_0)\in (C,O,\Omega)$. Then $f(\{x_0\}\tm_{\Top_B} C) \subset O$.
	By the fibrewise tube Lemma \ref{l1}, there exists an open neighborhood $U$ of $x_0$ such that  $f(U\tm_{\Top_B} C )\subset O$. Define $V=U\cap X_{\Omega}$. Then $\widehat{f}(V)\in (C,O,\Omega)$. It follows that 
	$\widehat{f}$ is continuous.\\
	Conversely, assume that $\widehat{f}:X\ra \map_B(K,Y)$ is continuous. By Proposition \ref{p2}, the evaluation map 
	$$ev:\map_B(K,Y)\tm_{\Top_B} K\ra Y$$ is continuous. Therefore $f$ which is the composite
	\begin{equation}\label{eq67} 
X\tm_{\Top_B}K \xra{\widehat{f}\tm_{\Top_B} 1_K} \map_B(K,Y)\tm_{\Top_B} K \st{ev}\ra Y
\end{equation} 
	is continuous. 
	\qedhere   
\end{proof}
\begin{proposition}\label{p26}
Assume that $B$ is $\T_1$.  Let $K,Z\in \Top_B$ with $K$  fibrewise compact fibrewise Hausdorff and let $\map_B(K,Z)$ be the exponential object defined by (\ref{eq44}).  If  $Z$ is fibrewise $\T_1$, then so is $\map_B(K,Z)$.  
\end{proposition}
  
\begin{proof}\mbox{}
 
 \begin{itemize}

\item Step 1: $K$ is the fibrewise space $B^b$ defined by Example \ref{ex2},  $b\in B$. \\
$B$ is $\T_1$ and the fibre $Z_b$ is closed $\T_1$-subspace of $Z$. Then by Example \ref{ex2}, $\map_B(B^b,Z)$ is a fibrewise $\T_1$-subspace. 
\item Step 2: The general case.\\
Let $\gamma \in \map_B(K,Z)$. we need to show that $\{\gamma\}$ is closed in $\map_B(K,Z)$. Let $b=p(\gamma)$, where $p$ is the projection of the fibrewise space $\map_B(K,Z)$. Then $\gamma \in \Top(K_b,Z_b)$. For each $x\in K_b$, define $$f_x:B^b \ra K $$ to be the fibrewise map given by $f_x(b)=x$ and 
let
$$\map_B(f_x,Z): \map_B(K,Z) \ra \map_B(B^b,Z)$$ be the fibrewise map induced by $f_x$.
Furthermore, let $\gamma_x\in \Top(\{b\},Z_b)$ to be the map given by $\gamma_x(b)=\gamma(x)$. Then $\gamma_x \in \map_B(B^b,Z)$. We have 
$$\{\gamma\}=\underset{x\in K_b}{ \bigcap}\map_B(f_x,Z)^{-1}(\{\gamma_x\}).$$  By Step 1, $\{\gamma\}$ is closed in $\map_B(K,Z)$.  
\end{itemize}
\qedhere   
\end{proof}
\begin{remark}\label{r9}  
  Let $K,Z\in \Top_B$ with $K$  fibrewise compact, fibrewise Hausdorff and $Z$ fibrewise Hausdorff. Then the exponential space $\map_B(K,Z)$ is not in general fibrewise Hausdorff even if $Z$ is  Hausdorff 
	(not just fibrewise Hausdorff)  and $B$ is $\T_1$.
  
\end{remark}

\section{Fibrewise weak and $k$-Hausdorfifications}\label{s7}
Our objective in this section is to prove that if $B$ is $\T_1$, then the subcategories of fibrewise weak Hausdorff spaces and fibrewise $k$-Hausdorff spaces are reflective subcategories of $\Top_B$. We adopt a definition of fibrewise weak Hausdorff spaces that is seemingly weaker than that of James \cite[Definition 1.1]{J1}. Our definition has the advantage that it agrees with  the ordinary definition  of  weak Hausdorff spaces when $B$ is reduced to a point (Strickland, \cite[Definition 1.2]{S}).\\

\begin{definition}\label{d3}\mbox{}
\begin{enumerate}
 
	\item  A fibrewise space $X$ over $B$ is said to be fibrewise weak Hausdorff
if for each open set $\Omega$ of $B$, each fibrewise compact, fibrewise Hausdorff space $K$ over $\Omega$ and each fibrewise map $\alpha:K\ra X_{\Omega}$, the image $\alpha(K)$ is closed in $X_{\Omega}$.
	\item  The  subcategory of $\Top_B$ whose objects are the weak Hausdorff spaces is denoted by $\hwTop_B$.
	\end{enumerate}
\end{definition}

\begin{proposition}\label{p5}
A fibrewise Hausdorff space is fibrewise weak Hausdorff.
\end{proposition}
  
\begin{proof}
Let $X$ be a fibrewise Hausdorff space, $\Omega$ open in $B$, $K$ a fibrewise compact, fibrewise Hausdorff space  over $\Omega$ and $u: K\ra X_{\Omega}$ a continuous fibrewise map. By Proposition \ref{p4}, $u(K)$ is closed in $X_{\Omega}$. Hence $X$ is weak Hausdorff.
\qedhere   
\end{proof}

\begin{proposition}\label{p6}

  Let $f:X\ra Y$  be an injective, continuous fibrewise map with $Y$ fibrewise weak Hausdorff. Then $X$ is fibrewise weak Hausdorff. In particular, a subspace of a fibrewise weak Hausdorff space is fibrewise weak Hausdorff.

\end{proposition}
  
\begin{proof}
Clear.
\qedhere   
 \end{proof}
\begin{proposition}\label{p7}
Assume that the base space $B$ is a $\T_1$-space. Then
every fibrewise weak Hausdorff space over $B$ is fibrewise $\T_1$.
\end{proposition}
  
\begin{proof}
Let $X$ be a fibrewise weak Hausdorff space over $B$ and let $x\in X$. $B$ is $\T_1$, thus the fibrewise subspace $\{x\}$ of $X$ is fibrewise compact, fibrewise Hausdorff space. $X$ is weak Hausdorff, thus $\{x\}$ is closed in $X$ and $X$ is $\T_1$. 
\qedhere  
\end{proof}

\begin{proposition}\label{p8}
Assume that the base space $B$ is a $\T_1$-space. 
 Let $u$ be a fibrewise continuous map from a fibrewise compact, fibrewise Hausdorff space $K$ to a fibrewise weak Hausdorff space $X$. Then:  
\begin{enumerate}
 
\item The map $u: K\ra X$ is proper.
\item  The subspace $u(K)$ is a closed, fibrewise Hausdorff subspace of $X$.  
\end{enumerate}

\end{proposition}
  
\begin{proof}\mbox{}
\begin{enumerate} 
 
 \item We use the characterization of proper maps given by Theorem \ref{t5}: Let $C$ be a closed subset of $K$. $C$ is fibrewise compact, fibrewise Hausdorff space over $B$, $X$ is weak Hausdorff thus $u(C)$ is closed. $u$ is then a closed map. $B$ is $\T_1$, by Proposition \ref{p7}, $X$ is $\T_1$. Let $x\in X$ and $b=p(x)$ where $p$ is the projection of $X$ on $B$. The subset $\{x\}$ is closed in $X$, thus $u^{-1}(x)$ is closed in the  compact space $X_b$. It follows that $u^{-1}(x)$ is compact. Therefore $u$ is proper.  
\item The map $u$ is proper, thus $u(K)$ is closed. By Proposition \ref{p6}, the subspace of a fibrewise weak Hausdorff space is fibrewise weak Hausdorff. We therefore may assume without loss of generalities that $u$ is onto. By the first point, $u$ is proper, thus by Proposition \ref{p10}, the map 
$$u\tm_{\Top_B}u: K\tm_{\Top_B} K\ra  X\tm_{\Top_B} X$$
is proper. $K$ is fibrewise Hausdorff, therefore by Proposition \ref{p3}, the diagonal $\Delta(K)$ of $K$ is closed in $K\tm_{\Top_B} K$. It follows that $\Delta(X)=u\tm_{\Top_B}u(K\tm_{\Top_B} K)$ is closed in $X\tm_{\Top_B} X$. By Proposition \ref{p3}, $X$ is fibrewise Hausdorff.
\end{enumerate}
\qedhere  
\end{proof}
\begin{proposition}\label{p38}
Assume that the base space $B$ is $\T_1$ and let $(X_i)_{i\in I}$ be a family of fibrewise weak Hausdorff spaces indexed by a (small) set $I$.  Then $\underset{\Top_B}{\overset{i\in I}{\prod}}X_i$ is fibrewise weak Hausdorff. 
\end{proposition}
  
\begin{proof}

Let $X=\underset{\Top_B}{\overset{i\in I}{\prod}}X_i$ and $p:X \ra B$  the projection of $X$ on $B$. 
\begin{itemize}
  
\item Step 1: Let $K$ be a fibrewise compact, fibrewise Hausdorff space over $B$, $u:K \ra X$  a continuous fibrewise map, $u_i:K\ra X_i$ the $i$-component of $u$ and $K_i=u_i(K)$, $i\in I$. Each $K_i$ is closed and by Proposition \ref{p8}.2, each $K_i$ is a  fibrewise Hausdorff subspace of $X_i$. It follows that $\underset{\Top_B}{\overset{i\in I}{\prod}}K_i$ is closed, fibrewise Hausdorff subspace of $X$. By Proposition \ref{p4}, $u(K)$ is closed in  $\underset{\Top_B}{\overset{i\in I}{\prod}}K_i$. Thus $u(K)$ is closed in  $X$.
\item Step 2: Let $\Omega$ be an open subset of $B$, 
$K$ a fibrewise compact, fibrewise Hausdorff space over $\Omega$, $Y=X_{\Omega}$ and $u:K \ra Y$ a continuous, fibrewise map. Define $Y_i=p^{-1}_{i}(\Omega)$. Then $Y=\underset{\Top_{\Omega}}{\overset{i\in I}{\prod}}Y_i$.
By Proposition \ref{p6}, each $Y_i$ is weak Hausdorff, thus by Step 1, 
$u(K)$ is closed in $Y$. It follows that $\underset{\Top_B}{\overset{i\in I}{\prod}}X_i$
is weak Hausdorff.
\end{itemize}
\qedhere   
\end{proof}

\begin{theorem}\label{t4} Assume that the base space $B$ is $\T_1$. 
Then the category $\hwTop_B$ is a reflective subcategory of $\Top_B$. In particular, 
$\hwTop$ is bicomplete. 

\end{theorem}
  
\begin{proof}
This follows from Theorem \ref{t3}, Proposition \ref{p6} and Proposition \ref{p38}.
\qedhere   
\end{proof}

 $k$-Hausdorff spaces are defined by Rezk in \cite[Section 4]{RC}. We here introduce the notion of fibrewise $k$-Hausdorff spaces.\\

\begin{definition}\label{d5}\mbox{}
\begin{enumerate}
 
	\item 
	A fibrewise space $X$ over $B$ is said to be fibrewise $k$-Hausdorff
if for each open set $\Omega$ of $B$, each fibrewise compact, fibrewise Hausdorff space $K$ over $\Omega$ and each continuous fibrewise map 
$u: K\ra X_{\Omega} \tm_{\Top_{\Omega}}X_{\Omega}$, the inverse image by $u$ of the diagonal of $X_{\Omega}$ is closed in $K$.
	\item  
	The  subcategory of $\Top_B$ whose objects are the fibrewise $k$-Hausdorff spaces is denoted by $\hkTop_B$.
	\end{enumerate}
\end{definition}

By Proposition \ref{p3}, a fibrewise Hausdorff space is fibrewise $k$-Hausdorff. The product in $\Top_B$ of fibrewise $k$-Hausdorff spaces is fibrewise $k$-Hausdorff.
Similarly, a subobject of a fibrewise $k$-Hausdorff spaces is fibrewise $k$-Hausdorff space.
We can apply Theorem \ref{t3} to get the following result.\\

\begin{proposition}\label{p11}
Assume that the base space $B$ is a $\T_1$-space. The subcategory $\hkTop_B$ of $\Top_B$ is reflective. In particular, 
$\hkTop_B$ is bicomplete.
\end{proposition}\mbox{}

  The next result  generalizes that of Rezk  \cite[Proposition 11.2]{RC}. 

\begin{proposition}\label{p12}
 Assume that the base space $B$ is a $\T_1$-space. Then $\hwTop_B$ is a reflective subcategory of $\hkTop_B$. 
\end{proposition}
   
\begin{proof}
In the light of Theorem \ref{t4}, we just need to prove that $\hwTop_B$ is a  subcategory of $\hkTop_B$. 
 Let $X$ be a fibrewise weak Hausdorff space.
\begin{itemize} 
 
\item Step 1:
Let $f:K\ra X\tm_{\Top_B} X$ be a continuous, fibrewise map, where $K$ is fibrewise compact, fibrewise Hausdorff space. Let $f_1$ and $f_2$ be the components of the map $f$. Define $K_1=f_1(K)$, $K_2=f_2(K)$ and $L=K_1\cup K_2$. The subspace $L$ of $X$  is the image of the continuous, fibrewise map $f_1\coprod_{\Top_B} f_2:K\coprod_{\Top_B} K \ra X$. The  space $K\coprod_{\Top_B} K$ is fibrewise compact fibrewise Hausdorff, thus $L$ is closed and  by  Proposition \ref{p8}.2, $L$ is fibrewise Hausdorff. $f$ factors through $L\tm_{\Top_B} L$ as follows
$$\begin{tikzcd}[row sep=large]
K\arrow{rr}{f}\arrow{dr}[swap]{g}&&X\tm_{\Top_B} X\\
&L\tm_{\Top_B} L\arrow{ur}[swap]{j}
 \end{tikzcd}$$
where $g:K\ra L\tm_{\Top_B} L$ is continuous, fibrewise map and $j$ is the inclusion map.  
Let $\Delta_X$ and $\Delta_L$ be the diagonals of $X$ and $L$ respectively. By Proposition \ref{p3}, $\Delta_L$ is closed in $L\tm_{\Top_B} L$,  thus 
$$f^{-1}(\Delta_X)=g^{-1}(j^{-1}(\Delta_X))=g^{-1}(\Delta_L)$$
  is closed in $K$.
\item Step 2: Let $\Omega$ be open in $B$, $K$ a fibrewise compact, fibrewise Hausdorff space over $\Omega$ and 
$u: K\ra X_{\Omega} \tm_{\Top_{\Omega}}X_{\Omega}$ a continuous, fibrewise map. $X$ is fibrewise weak Hausdorff, by Proposition \ref{p6},  $X_{\Omega}$ is fibrewise weak Hausdorff. Therefore by Step 1, $u(K)$ is closed in $X_{\Omega} \tm_{\Top_{\Omega}}X_{\Omega}$. It follows that $X$ is $k$-Hausdorff.
\end{itemize}
 \qedhere   
\end{proof}
\begin{remark}\label{pr5}
 A space $X\in \Top$ is weak Hausdorff (resp. $k$-Hausdorff) if it corresponds, under the isomorphism $P$ of (\ref{peq16})
to a fibrewise weak Hausdorff (resp. $k$-Hausdorff) space over $Pt$. The subcategory of $\Top$ of such spaces is reflective and is denoted by $\hwTop$ (resp. $\hkTop$)   
\end{remark}

\section{Left Kan extendable subcategories of $\Top_B$}\label{s8}

It is well known that any subcategory of $\Top$ containing a nonempty space has a coreflective hull (\cite[Theorem 12]{HST}, \cite[Proposition 2.17]{HH} and \cite[page 283]{HEST}).
In this section, we prove that any subcategory of $\Top_B$, which is suitable in the sense of the definition below, has a strong coreflective hull. \\

\begin{definition}\label{pd10}
A subcategory $\mw$ of $\Top_B$ is said to be suitable if for every $b\in B$, there exists a fibrewise topological space $E(b)$ in $\mw$ such that 
 \begin{equation}\label{eq14} 
 \left\{ \begin{array}{ll}
E(b)_{b}\neq \emptyset &    \\[2mm]
 E(b)_{c}=\emptyset  &\text{ for all } c\neq b
\end{array} \right. 
\end{equation} 
where $E(b)_{c}$ is the fibre of $E(b)$ over $c \in B$.\\
\end{definition}
 
\noindent Let $\mw$ be a suitable subcategory of $\Top_B$ (See Definition \ref{pd10}). For $X\in \Top_B$, let 
\begin{equation}\label{eq1}
			J_X:\mw/X\ra\Top_B
\end{equation}
be the functor which takes an arrow $V \rightarrow X$ to its domain $V,$     
\begin{equation}\label{eq13}
			\abs{J_X}:\mw/X\ra\Set_{\abs{B}}
\end{equation}
  its underlying functor as defined by (\ref{eq72}) and 
	$$P_{\abs{B}}:\Set_{\abs{B}}\ra \Set$$
	the functor defined by (\ref{eq23}).
For $(V \st{\sigma}\ra X)  \in  \mw/X$, define a map  
$$\begin{array}{rcll}
\lambda_{\sigma}:&\abs{V} &\ra &\abs{X}\\
    &v &\mapsto&\abs{\sigma}(v)
	\end{array}$$
The maps $\lambda_{\sigma}$ define a cone

\begin{equation}\label{peq29}
			P_{\abs{B}}\abs{J_X} \st{\lambda}\Ra \abs{X}
\end{equation}

 \begin{itemize}
 
\item  Let $ V \st{\sigma}\ra X$, $V' \st{\sigma'}\ra X$ be  in  $\mw/X$,  $p_{\sigma}$ and $p_{\sigma'}$  the projections of the fibrewise spaces $V$ and $V'$  and  $v\in V$,  $v'\in V'$. The fact that $\mw$ is  suitable implies that 
	$\lambda_{\sigma}(v)=\lambda_{\sigma'}(v')$ iff the objects $(\sigma,v)$ and $(\sigma',v')$ of $\int P_{\abs{B}}\abs{J_X}$ are in the same connected component.

	\item Let  $x_0\in \abs{X}$ and let $b_0=\abs{p_X}(x_0)$, where  $p_X:X\ra B$ is the projection of the fibrewise space $X$ over $B$. Let $E(b_0)$ be as in (\ref{eq14}) and define $\sigma_0:E(b_0) \ra X$ to be the fibrewise map  given by $\lambda_{\sigma_0}(e)=x_0$  for all $e\in E(b_0)$. Then 
	$\sigma_0 \in \mw/X$ and $x_0\in  \lambda_{\sigma_0}(E(b_0))$.
	\end{itemize}
Therefore by Remark \ref{r1}.1.(c), the cone $P_{\abs{B}}\abs{J_X} \st{\lambda}\Ra \abs{X}$ given by (\ref{peq29}) is a colimiting cone. By Remark \ref{r1}.2, $\abs{J_X}$ has a colimit. 
		Therefore by Lemma \ref{l3}, $J_X$ has a colimit whose underlying   set is $\abs{X}$ and whose topology is the final topology defined by the functions $P_{\abs{B}}\lambda_{\sigma}= \abs{P_B(\sigma)}   :\abs{V} \ra \abs{X}$, $\sigma \in \mw/X$.\\
\indent This proves that the inclusion functor $\mw\st{J}{\hra}\Top_B$ has a  density comonad $(L,\epsilon,\delta)$  satisfying $\abs{L(X)}=\abs{X}$, $\forall X\in\Top_B$.
		Furthermore, the underlying map $\abs{\epsilon_X}$ of the counit $\epsilon_X:L(X)\ra X$ of the subcategory $\mw$ of $\Top_B$ is just the identity map $1_{\abs{X}}$.
		In particular, $\epsilon_X$ is monic and by  Theorem \ref{pt1},
		we have the following result. 
		
\begin{theorem}\label{t2}	
Let $\mw$ be a suitable subcategory of $\Top_B$. Then:
\begin{enumerate}
  
\item 
The subcategory $\mw$ is left Kan extendable in $\Top_B$.
\item The coreflector $ \Top_B \xra{\omega} \mw_l[\Top_B]$ takes a fibrewise  topological space $X$ to the fibrewise topological space $\omega(X)$ having the same underlying  set as $X$ and whose topology is the final topology induced by the functions $\abs{V} \xra{\abs{P_B(\sigma)}} \abs{X}$, $\sigma \in \mw/X$.
\item A fibrewise topological space $X$ over $B$ is $\mw$-generated iff $X$ has the final topology defined by all continuous fibrewise maps  $V \sra X$, where $V$ is a fibrewise space in $\mw$.

\end{enumerate}  
\end{theorem}

\begin{example}\label{ex3}

For $b\in B$, let $B^b$ be the fibrewise subspace of $B$ defined by Example \ref{ex2}.
Let $\md$ be the subcategory of $\Top_B$ whose objects are the fibrewise spaces  $B^b$, $b\in B$. Then $\md$ is a suitable subcategory of  
 $\Top_B$. It is then left Kan extendable and $\md_l[\Top_B]$ is precisely the subcategory $\Dis_B$ of $\Top_B$  of discrete fibrewise spaces over $B$. 
\end{example}
\indent A subcategory $\mw$ of $\Top$ is said to be suitable if it  corresponds, under the isomorphism $P$ of (\ref{peq16}) to a suitable subcategory of $\Top_B.$ That is, if $\mw$ contains a nonempty space. By substituting $\Pt$ for $B$, one partially recovers a result of Herrlich and Strecker  \cite[Proposition 2.17]{HH}. 

\begin{corollary}\label{pc14}	
Let $\mw$ be a  suitable subcategory of $\Top$. Then:
\begin{enumerate} 
 
\item 
$\mw$ is left Kan extendable in $\Top$.
\item The coreflector $ \Top \xra{\omega} \mw_l[\Top]$  takes a topological space $X$ to the topological space $\omega(X)$ having the same underlying set as $X$ and whose topology is the final topology induced by the functions $\abs{V} \st{\abs{\sigma}}\ra \abs{X}$, $\sigma \in \mw/X$.
\item A topological space $X$ is $\mw$-generated iff $X$ has the final topology defined by all continuous maps  $V \ra X,$ $V \in \mw$.\\

\end{enumerate}  
\end{corollary}

Let $\mw$ be a suitable subcategory of $\Top_B$. For $b\in B$, let $E(b)$ in $\mw$ be as in (\ref{eq14}) and let $B^b$ to be as defined in Example \ref{ex3}.
$B^b$ is a retract of  $E(b)$. By Lemma \ref{pl3}, $B^b$ is $\mw$-generated. Therefore by Example \ref{ex3} and Corollary \ref{pc5}.1, every discrete fibrewise space is $\mw$-generated and we have the following. 

  \begin{lemma}\label{l4}	
Let $\mw$ be a suitable subcategory of $\Top_B$. Then every discrete fibrewise space over $B$ is $\mw$-generated. 
 \end{lemma}	

\begin{proposition}\label{p13}	
Let $\mw$ be a suitable subcategory of $\Top_B$. Then: 
\begin{enumerate}
 
	\item The fibrewise quotient of a $\mw$-generated fibrewise space is $\mw$-generated.
	\item  A fibrewise space is $\mw$-generated iff it is the fibrewise quotient of a coproduct of spaces in $\mw$.
\end{enumerate}
\end{proposition}
  
\begin{proof}\mbox{} 

\begin{enumerate}
 
	\item Let $X$ be a $\mw$-generated fibrewise space and $\sim$ a fibrewise equivalence relation on $\abs{X}$. Let $R$ the discrete topological space whose underlying set is the graph of the equivalence relation $\sim$. The space $R$ is a fibrewise space over $B$. The fibrewise quotient quotient space $X/\sim$ is the coequalizer in $\Top_B$
$$\begin{tikzcd}
R \ar[r,shift left=.75ex,"pr_1"]
  \ar[r,shift right=.75ex,swap,"pr_2"]
&
X \ar[r,"q"]  &X/\sim 
\end{tikzcd}$$
 where $pr_1$ and $pr_2$ are induced by the projections on the first and second factors. The fibrewise space $X$ is $\mw$-generated and by Lemma \ref{l4}, $R$ is $\mw$-generated. Therefore by Proposition \ref{pp6}.1, $X/\sim$ is $\mw$-generated.
	\item Straight forward generalization of the Escardó-Lawson proof of the same result when $B$ is a one point space   \cite [Lemma 3.2.(iv)]{ES}. 
\end{enumerate}

 \qedhere   
\end{proof}	
\begin{corollary}\label{c14} 
Let $\mw$ be a suitable subcategory of $\Top_B$. Assume that a fibrewise space $X$ is such that every point of $X$ has a neighborhood which is in $\mw$. Then $X$ is $\mw$-generated.	
 \end{corollary}
   	
\begin{proof} 	
  For each $x\in X$, choose a  neighborhood $V_x$ of $x$ which is in $\mw$ and let 
$i_x:V_x \ra X$ be the inclusion map. Then the map 
\begin{equation}\label{eq71}
 \coprod ^{x\in X}_{\Top_B} V_x \ra X
\end{equation}
whose restriction to $V_x$ is $i_x$, 
is a  fibrewise quotient map. By Proposition \ref{p13}.2, $X$ is $\mw$-generated.
\qedhere   
\end{proof}	
 
\begin{proposition}\label{p14}
Let  $\Topo_B$ be a reflective subcategory of $\Top_B$ that is closed under subobjects and let $\mw$ be a suitable subcategory of $\Topo_B$. Then
\begin{enumerate}

\item The subcategory $\mw$ of $\Topo_B$ is left Kan extendable.
\item $\mw_l[\Topo_B]= \Topo_B\cap\mw_l[\Top_B]$.
\item  A reflection of $\Top_B$ on $\Topo_B$ induces a reflection of $\mw_l[\Top_B]$ on $\mw_l[\Topo_B]$.
 
 \item  A coreflection of $\Top_B$ on $\mw_l[\Top_B]$  induces a coreflection of  $\Topo$ on $\mw_l[\Topo]$. 
 
 \end{enumerate}
\end{proposition}

\begin{proof} By Theorem \ref{t2}.1, $\mw$ is left Kan extendable  subcategory of $\Top_B$. Let $(L,\epsilon,\delta)$ be the density comonad of the inclusion functor $J: \mw \ra \Top_B$.
Let $X_0 \in\Topo_B$, by  Theorem \ref{pt1}, the map $ \epsilon_{X_0}:L(X_0)\ra X_0$ is $\mw$-monic. The subcategory $\mw$ of $\Top_{B}$ is suitable, therefore  $\epsilon_{X_0}$ is monic. The subcategory $\Topo_B$ is closed under subobject, thus $L(X_0)\in \Topo_B$. Therefore $L(\Topo_B)\subset \Topo_B$ and then the points 1-4 follow  from Theorem \ref{pt4}.
\qedhere   
\end{proof}


\section{Cartesian closed category of $\mw$-generated objects}\label{s10}
 Given a left Kan extendable subcategory $\mw$ of a category $\mc$, In this section, we present sufficient conditions for the category $\mw_l[\mc]$  to be cartesian closed.\\
\indent Assume that $Y$ is an exponentiable object in a category $\mc$ and let
$G: \mc\ra\mc$ be a right adjoint of the functor $ .\tm_{\mc} Y:\mc\ra\mc$.
For $Z\in \mc$, the object $G(Z)$ is called an exponential object and denoted by $Z^{Y}$.

\begin{examples}\label{pex4}\mbox{}
\begin{enumerate}

	\item In $\Top$, the exponentiable objects are precisely the core compact spaces \cite{DK, EH,I2}. In particular, locally compact Hausdorff spaces are exponentiable.  
	\item  By Theorem \ref{t1}, every fibrewise compact fibrewise Hausdorff space over $B$ is exponentiable in the category $\Top_B$.
	\item By \cite[Corollary 2.9]{N}, every local homeomorphism $X\ra B$ is an exponentiable object of $\Top_B$. 
	\end{enumerate}
	 \end{examples}

\begin{lemma}\label{pl6}
Let $\mw$ be a left Kan extendable subcategory of a bicomplete category $\mc$. Assume that  
\begin{enumerate}

	\item  Every object in $\mw$ is exponentiable in $\mc$.
	\item For every $V,W \in \mw$, the object $V\times_{\mc} W \in \mw_l[\mc]$.
	\end{enumerate}
	Then for every $V \in \mw$ and every $Y \in \mw_l[\mc]$, $V\times_{\mc} Y$ is a $\mw$-generated object. That is  $V\tm_{\mw_l[\mc]} Y\cong V\times_{\mc} Y$. 
\end{lemma}
 
\begin{proof} 
Let $V \in \mw$ and $Y\in \mw_l[\mc]$. 
By Corollary \ref{pc6}, there exists a functor  $F:\mk \ra\mc$ taking values in $\mw$ such that
$Y\cong\colim F$. Define $V\tm_{\mc} F$ to be the composite functor 
 $\mk \xra{F}\mc\xra{V\times_{\mc}} \mc.$ 
Then
$$ \begin{array}{rclc}
    V\times_{\mc} Y  & \cong & V\times_{\mc} \colim F & \\
						& \cong & \colim  V\times_{\mc} F& (\text{because $V$ is exponentiable in $\mc$})
	\end{array}$$
By 2., $V\tm_{\mc} F$ takes   values in   $\mw_l[\mc]$. Therefore, by Proposition \ref{pp6}.1, 
 $V\times_{\mc} Y\cong \colim V\tm_{\mc} F$ is in $\mw_l[\mc]$. Thus by Proposition \ref{pp6}.2, $V\times_{\mw_l[\mc]}\ Y$ exists and 
 $$V\times_{\mw_l[\mc]}\ Y\cong F_{L}(V\times_{\mc} Y) \cong V\times_{\mc} Y. $$
\qedhere   
 \end{proof}
  Assume next that $\mw$ and $\mc$ are as in Lemma \ref{pl6}. \\
	
	\begin{itemize}

 \item For $X,Y \in \mw_l[\mc]$, let $J_X: \mw/X \ra \mc$ be as defined by (\ref{peq9}) and let 
$J_X \tm_{\mc} Y$ be the composite functor
\begin{equation}\label{peq10}
J_X\tm_{\mc} Y:\mw/X \st{J_X}\ra \mc\xra{-\tm_{\mc} Y} \mc
\end{equation}
By Proposition \ref{pp6}, $\mw_l[\mc]$ is complete. For
 $(V \st{\sigma} \ra X) \in \mw/X$, define 
$$\theta_{\sigma}= \sigma \tm_{\mw_l[\mc]} 1_Y:V \tm_{\mc} Y = V \tm_{\mw_l[\mc]} Y\ra X \tm_{\mw_l[\mc]} Y.$$ 
The maps $\theta_{\sigma}$ define a cone  
\begin{equation}\label{peq11}  
  J_X\tm_{\mc} Y\st{\theta}\Ra X \tm_{\mw_l[\mc]} Y 
\end{equation}
 \item Let 
 $$\Hom(V,.): \mc\ra \mc$$ be a right adjoint of the functor
$.\tm_{\mc} V: \mc\ra \mc$. For $Y,Z \in \mw_l[\mc]$, define 
\begin{equation}\label{peq12}
\begin{array}{rccl}
  S^{Y}_{Z}= \Hom(J_Y(.),Z):& (\mw/Y)^{op}& \ra &\mc \\
     &(V \st{\sigma} \ra Y)&\longmapsto & \Hom(V,Z)
	\end{array}
  \end{equation}
	\end{itemize}
 \begin{definition}\label{pd7}
A left Kan extendable subcategory $\mw$ of a bicomplete category $\mc$ is said to be  \textbf{\textit{closeable}} if 
\begin{enumerate}

\item Every object in $\mw$ is exponentiable in $\mc$.
\item For every $V,W \in \mw$, the object $V\times_{\mc} W \in \mw_l[\mc]$.
\item For all $X,Y \in \mw_l[\mc]$, the cone  $J_X\tm_{\mc} Y\st{\theta}\Ra X \tm_{\mw_l[\mc]} Y$ given by (\ref{peq11}) is a colimiting cone.
\item For all $Y,Z \in \mw_l[\mc]$, the functor  $S^{Y}_{Z}: (\mw/Y)^{op} \ra \mc$ given by (\ref{peq12}) has a limit. 
\end{enumerate}
\end{definition}

For the remainder of this section, we assume that $\mw$ is a closeable left Kan extendable subcategory of a bicomplete category $\mc$.
Define 
\begin{equation}\label{peq13} 
\hom(.,.) : \mw_l[\mc]^{op}\times \mw_l[\mc] \ra  \mc
\end{equation}
by $$\hom(Y,Z)= \lm S^{Y}_{Z}=\underset{(V\st{\sigma}{\rightarrow}Y)\in \mw|Y}\lm \Hom(V,Z)$$
Then for $V \in \mw$, the arrow-object  $1_V$ of $\mw/V$ is terminal, it is therefore an initial object in the opposite category $(\mw/V)^{op}$.   It follows that the limit of the functor 
$$S^{V}_{Z}: (\mw/V)^{op}\ra \mc $$
is just $S^{V}_{Z}(1_V)$ which is $\Hom(V,Z)$. That is $\hom(V,Z)\cong \Hom(V,Z)$.\\

\begin{lemma}\label{pl7}
 Let $V\in \mw$ and $Y,Z \in \mw_l[\mc]$. There exists a natural bijection
$$ \mc (V,\hom(Y,Z))\cong \mc(V\tm_{\mc} Y,Z).$$
\end{lemma}
 
\begin{proof}
   
$$ \begin{array}{rclr}
    \mc (V,\hom(Y,Z))& \cong & \mc (V, \underset{(W\st{\sigma}{\rightarrow}Y)\in \mw|Y}\lm \Hom(W,Z)) & \\
						& \cong &\underset{(W\st{\sigma}{\rightarrow}Y)\in \mw|Y}\lm \mc(V, \Hom(W,Z))& \\
							& \cong &\underset{(W\st{\sigma}{\rightarrow}Y)\in \mw|Y}\lm \mc(V\times_{\mc} W,Z)&\\
						& \cong &\mc(\underset{(W\st{\sigma}{\rightarrow}Y)\in \mw|Y}\colim  V\times_{\mc} W,Z)&	\\
						& \cong & \mc(V\times_{\mc} Y,Z)& (\text{because }$V$ \text{ is exponentiable in $\mc$})        
\end{array}$$
\qedhere   
\end{proof}
Let $F_{L} :\mc \ra \mw_l[\mc]$ be the coreflector.  Define 
\begin{equation}\label{peq14}
\begin{array}{rrll}
(-)^{(-)}:&\mw_l[\mc]^{op}\times \mw_l[\mc] &\ra &\mw_l[\mc]\\
    &(Y,Z) &\mapsto &Z^Y
	\end{array}
	\end{equation}
to be the composite functor
\begin{equation}\label{peq15} 
\mw_l[\mc]^{op}\times \mw_l[\mc] \xra{\hom} \mc \st{F_{L}}\ra \mw_l[\mc]
\end{equation} 
Then for $Y,Z \in \mw_l[\mc]$, $ Z^Y=F_{L}(\hom(Y,Z))$. 

\begin{lemma}\label{pl8}
 Let $V\in \mw$ and $Y,Z \in \mw_l[\mc]$. There exists a natural bijection
$$\mw_l[\mc](V,Z^Y)\cong \mw_l[\mc](V\tm_{\mw_l[\mc]} Y,Z).$$
\end{lemma}
  
\begin{proof}
   
$$ \begin{array}{rclr}
    \mw_l[\mc](V,Z^Y)& \cong &\mw_l[\mc](V,F_{L}( \hom(Y,Z)))& \\
						& \cong &\mc(V, \hom(Y,Z))& (\text{by  Proposition  }\ref{pp17}) \\
							
							& \cong & \mc(V\times_{\mc} Y,Z)& (\text{by Lemma }\ref{pl7})\\
						& \cong &\mc( V\tm_{\mw_l[\mc]} Y,Z)& (\text{by Lemma }\ref{pl6})\\
						& \cong &\mw_l[\mc]( V\tm_{\mw_l[\mc]} Y,Z)&        
\end{array}$$
\qedhere   
\end{proof}
\begin{theorem}\label{pt10} 
$\mw_l[\mc]$ is cartesian closed with internal hom functor the functor
$$(-)^{(-)} : \mw_l[\mc]^{op}\times \mw_l[\mc] \ra \mw_l[\mc] $$
defined by (\ref{peq15}).
\end{theorem}
  
\begin{proof}
 Let $X,Y,Z \in \mw_l[\mc]$.  
$$ \begin{array}{rclr}
    \mw_l[\mc](X,Z^Y)& \cong &\mw_l[\mc](\underset{(V\st{\sigma}{\rightarrow}X)\in \mw|X}\colim  V,Z^Y)& \\
						& \cong &\underset{(V\st{\sigma}{\rightarrow}X)\in \mw|X}\lm \mw_l[\mc](V,Z^Y)& \\
							& \cong &\underset{(V\st{\sigma}{\rightarrow}X)\in \mw|X}\lm \mw_l[\mc](V\tm_{\mw_l[\mc]} Y,Z)& (\text{by Lemma }\ref{pl8})\\
							& \cong &\underset{(V\st{\sigma}{\rightarrow}X)\in \mw|X}\lm \mw_l[\mc](J_X\tm_{\mc} Y(\sigma),Z)&  \\

							& \cong & \mw_l[\mc](\colim J_X\tm_{\mc} Y,Z)& \\
						& \cong &\mw_l[\mc]( X\tm_{\mw_l[\mc]} Y,Z)&        (\text{by Definition }\ref{pd7}.3)
\end{array}$$
\qedhere   
\end{proof}
  The next result is a generalization of that of Escardó-Lawson \cite[Corollary 5.5]{ES}. 

\begin{corollary}\label{pc11} Let $\mw'$ be another closeable, left Kan extendable subcategory of $\mc$ which is contained in $\mw$.  Then the inclusion functor $\mw'_l[\mc]\ra \mw_l[\mc]$ preserves finite products.

\end{corollary}
  
\begin{proof}
Let $J':\mw' \ra \mc$ be the inclusion functor, $X,Y \in \mw'_l[\mc]$ and $J'_X\tm_{\mc} Y:\mw'/X \ra \mc$ be as in (\ref{peq10}).
The functor $J'_X\tm_{\mc} Y$ factors through $\mw_l[\mc]$ as follows:
$$\begin{tikzcd}
\mw'/X\arrow{rr}{J'_X\tm_{\mc} Y}\arrow{dr}[swap]{H}&&\mc\\
&\mw_l[\mc]\arrow{ur}[swap]{U}
 \end{tikzcd}$$

$$ \begin{array}{rclr}
    X\tm_{\mw'_l[\mc]}Y & \cong &\colim J'_X\tm_{\mc} Y& (\text{by Definition }\ref{pd7}.3)\\
		& \cong & \colim UH &\\
						& \cong & \colim H &    (\text{by Proposition \ref{pp6}.1)}\\
						& \cong &X\tm_{\mw_l[\mc]}Y& (\text{by Lemma \ref{pl6} and Theorem \ref{pt10}} )
\end{array}$$
\qedhere   
\end{proof}
\section{A fibrewise Day's theorem}\label{s11}

The aim of this section is to use the notion of Kan extendable subcategories to provide a fibrewise version of Day's theorem (\cite[Theorem 3.1]{D}).  We begin with the following simple observation. 

\begin{remark}\label{r4} Let $\mi\st{F}\ra \mc$, $\mj\st{G}\ra \mc$ and $\mi\st{P}\ra \mj$  be functors. Assume that $F$ and $G$ have colimits and let $F\st{\alpha}\Ra GP$ be a natural transformation.

$$\begin{tikzcd}
\mi \arrow[dr, "P"' ] \arrow[rr, ""{name=U, below}, "F" ]{}
& & \mc  \\
& \mj \arrow[Rightarrow, from=U, "\alpha"']  \arrow[ur, "G"' ]
\end{tikzcd}$$
Then there exists a unique map 
$h:\colim F\ra\colim G$ rendering commutative the  diagram  
$$\begin{tikzcd}
F(i) \arrow[r, "\alpha_i"] \arrow[d]
& GP(i) \arrow[d] \\
\colim F \arrow[r,  "h"' ]
& \colim G
\end{tikzcd}$$
for all $i\in I$. 
\end{remark}

\begin{theorem}\label{t6} Assume that
\begin{enumerate}

\item The space $B$ is a $\T_1$-space.
\item The subcategory $\mw$ of $\Top_B$ is suitable (See Definition \ref{pd10}).  
\item Every fibrewise space in $\mw$ is exponentiable as an object of $\Top_B$.
\item For every $V,W \in \mw$, the fibrewise space  $V\times_{\Top_B} W$ is $\mw$-generated. 
\end{enumerate}
 Then $\mw$ is left Kan extendable. Moreover, $\mw_l[\Top_B]$ is a cartesian closed subcategory of $\Top_B$.  
\end{theorem}
   
\begin{proof}
By Theorem \ref{t2}, $\mw$ is left Kan extendable.	
 In the light of Theorem \ref{pt10}, we just need to prove that conditions 3 and 4 of Definition \ref{pd7} are satisfied.\\
Let $X, Y \in \mw_l[\Top_B]$, $J_X:\mw/X \ra \Top_B$ be the functor defined by  (\ref{eq1}) and $J_X\tm_{\Top_B} Y:\mw/X  \ra \Top_B $ be the composite functor
\begin{equation}\label{eq2}
J_X\tm_{\Top_B} Y:\mw/X \xra{J_X} \Top_B\xra{-\tm_{\Top_B} Y} \Top_B
\end{equation}
By Theorem \ref{t2}, the functor $J_X$ has a colimit. Therefore by Lemma \ref{l3}, the functor $\abs{J_X}:\mw/X \ra \Set_{\abs{B}}$ has a colimit.
 $\Set_{\abs{B}}$ is cartesian closed, thus the functor 
 $-\tm_{\Set_B}\abs{Y}:\Set_{\abs{B}}\ra   \Set_{\abs{B}}$ is left adjoint and preserves colimits. It follows that the composite of these last two functors, which is $\abs{J_X\tm_{\Top_B} Y}$, has a colimit. Again by Lemma \ref{l3}, the functor $J_X\tm_{\Top_B} Y$ has a colimit.
Let  $ J_{X\tm_{\Top_B} Y} \st{\lambda}\Ra X\tm_{\mw_l[\Top_B]} Y$ and  
$J_ X\tm_{\Top_B} Y \st{\mu}\Ra \colim (J_ X\tm_{\Top_B} Y)$  be colimiting cones. Observe that for $  (f: V\ra X\tm_{\Top_B} Y) \in \mw/X\tm_{\Top_B} Y $, the component $\lambda_f$ of the cone $\lambda $ along $f$ is the map
\begin{equation}\label{eq70}
  \lambda_f= F_{L}(f): V\ra X\tm_{\mw_l[\Top_B]} Y
\end{equation} 
where $F_{L}: \Top{\ra}\mw_l[\Top_B]$ is the coreflector.\\
\indent The cone $\theta:J_X\tm_{\Top_B} Y\Ra X \tm_{\mw_l[\Top_B]} Y $  defined by 
 (\ref{peq11}) induces a map
\begin{equation}\label{eq3}
			 \colim (J_X\tm_{\Top_B} Y)\st{\tilde{\theta}}\ra X \tm_{\mw_l[\Top_B]} Y
\end{equation}
It is such that for every $(V\st{\sigma}\ra X) \in \mw/X$, the  diagram commutes
\begin{equation}\label{eq4}
 \begin{tikzcd}
 V\tm_{\Top_B} Y=V\tm_{\mw_l[\Top_B]}Y  \arrow[d,"\mu_{\sigma}"']   \arrow[dr, "\sigma\tm_{\mw_l[\Top_B]}1_Y"] &\\
 \colim (J_X\tm_{\Top_B} Y)\arrow[r, "\tilde{\theta}"']& X\tm_{\mw_l[\Top_B]} Y
\end{tikzcd} 
\end{equation} 
We need to prove that $\tilde{\theta}$ is an isomorphism.				
Let $P:\mw /X\tm_{\Top_B} Y \ra \mw /X$  be the functor which takes an object in $\mw /X\tm_{\Top_B} Y$, which is an arrow  $f=(\sigma, \tau):V\ra X\tm_{\Top_B} Y$, to its first component $\sigma:V\ra X$, which is an object in $\mw /X$.
Define a natural transformation 
\begin{equation}\label{eq69}
			 J_{X\tm_{\Top_B} Y} \st{\alpha}\Ra (J_X\tm_{\Top_B} Y)P
\end{equation}
as follows:\\
\indent  For $f=(\sigma, \tau):V\ra X\tm_{\Top_B} Y$,  $\alpha _f=(1_V,\tau): V \ra   V\tm_{\Top_B} Y$
\begin{equation}\label{eq5}
 \begin{tikzcd}
\mw /X\tm_{\Top_B} Y \arrow[dr, "P"' ] \arrow[rr, ""{name=U, below}, "J_{X\tm_{\Top_B} Y}" ]{}
& & \Top_B  \\
& \mw /X  \arrow[Rightarrow, from=U, "\alpha"']  \arrow[ur, "J_ X \tm_{\Top_B} Y"' ]
\end{tikzcd}
\end{equation} 
The natural transformation $\alpha$ is such that the following diagram commutes
\begin{equation}\label{eq6}
 \begin{tikzcd} 
 & V \arrow[dr, "f"]\arrow[dl,"\alpha _f"'] \\
V\tm_{\Top_B} Y  \arrow[rr,"\sigma\tm_{\Top_B} 1_Y"'] && X\tm_{\Top_B} Y
\end{tikzcd} 
\end{equation}
Applying the coreflector $F_{L}: \Top{\ra}\mw_l[\Top]$ to (\ref{eq6}), we get a new commutative diagram
\begin{equation}\label{eq7}
 \begin{tikzcd} 
 & V \arrow[dr, "\lambda_f"]\arrow[dl,"\alpha _f"'] \\
V\tm_{\Top_B} Y  \arrow[rr,"\sigma\tm_{\mw_l[\Top]} 1_Y"'] && X\tm_{\mw_l[\Top_B]} Y
\end{tikzcd} 
\end{equation}
By Remark \ref{r4}, the natural transformation $J_{X\tm_{\Top_B} Y} \st{\alpha}\Ra (J_X\tm_{\Top_B} Y)P$ induces a map $X\tm_{\mw_l[\Top_B]} Y \st{h}\ra \colim J_X\tm_{\Top_B} Y$. It is such that 
for every $(f=(\sigma, \tau):V\ra X\tm_{\Top_B} Y) \in \mw/X\tm_{\Top_B} Y $, the diagram commutes
\begin{equation}\label{eq8}
 \begin{tikzcd}
V \arrow[r, "\alpha_f"] \arrow[d,"\lambda_f"']
& V\tm_{\Top_B} Y \arrow[d,"\mu_{\sigma}"] \\
X\tm_{\mw_l[\Top_B]} Y \arrow[r,  "h"' ]
&  \colim  J_X\tm_{\Top_B} Y
\end{tikzcd}
\end{equation} 
Gluing together diagrams (\ref{eq8}) and (\ref{eq4}) along their common edge, we get the following commutative diagram

\begin{equation}\label{eq9}	
\begin{tikzcd}
V \arrow[r, "\alpha_f"] \arrow[d,"\lambda_f"']
& V\tm_{\Top_B} Y \arrow[d,"\mu_{\sigma}"]  \arrow[dr, "\sigma\tm_{\mw_l[\Top]} 1_Y"] &\\
X\tm_{\mw_l[\Top_B]} Y \arrow[r,  "h"' ]
& \colim  J_X\tm_{\Top_B} Y\arrow[r, "\tilde{\theta}"']& X\tm_{\mw_l[\Top_B]} Y
\end{tikzcd}
\end{equation}
By (\ref{eq7}), $(\sigma\tm_{\mw_l[\Top_B]} 1_Y) \alpha_f=\lambda_f$. Therefore $\tilde{\theta} h= 1_{X\tm_{\mw_l[\Top_B]} Y}$. The maps $\tilde{\theta}$ and $h$ induce isomorphisms on the underlying sets, therefore, we also have $h\tilde{\theta}= 1_{\colim  J_X\tm_{\Top_B} Y}$. It follows that $\tilde{\theta}$ is an isomorphism and condition 3 of Definition \ref{pd7} is fulfilled. Condition 4 results from Lemma \ref{l6} below.
\qedhere   
\end{proof}

\begin{lemma}\label{l17} 
 Let $W,Y \in \Top_B$ with $W$ exponentiable in $\Top_B$ and 
$\Hom(W,.):\Top_B \ra \Top_B$ a right adjoint of the functor
$W\tm_{\Top_B} .:\Top_B \ra \Top_B$.  Then 
 $$\abs{\Hom(W,Y)_b}\cong  \Top(W_b,Y_b), \quad \forall b\in B.   $$
 \end{lemma}
   
\begin{proof} 
 Let   $b\in B$, $B^b$ be the fibrewise space over $B$ defined by Example \ref{ex3}. Then
 $$\abs{\Hom(W,Y)_b}\cong \Top_B(B^b,\Hom(W,Y))\cong \Top_B(B^b\tm_{\Top_B}W,Y)\cong \Top(W_b,Y_b).$$
\qedhere   
 \end{proof}
\begin{lemma}\label{l6} 
Assume that $B$ is $\T_1$, $\mw$ is a suitable subcategory of $\Top_B$ and that every object of $\mw$ is exponentiable in $\Top_B$.  Let $Y,Z \in \mw_l[\Top_B]$,  then the functor 
 
$$\begin{array}{rccl}
  S^{Y}_{Z}:& \mw/Y& \ra &\Top_B \\
     &(W \st{\sigma} \ra Y)&\longmapsto & \Hom(W,Z)
	\end{array}$$
	has a limit. 
 \end{lemma}
  
\begin{proof} Let $T_Y:\mw/Y \ra \Top_B$ be as in  (\ref{eq1}). Then $\colim T_Y\cong Y$. Let $b\in B$ and let
 $$\pi^{s}_{b}:\Set_{\abs{B}} \ra \Set \quad  \mbox{ and } \quad \pi^{t}_{b}:\Top_B \ra \Top$$ be the functors defined by (\ref{eq27}) and (\ref{eq28}) respectively.
$$ \begin{array}{rclr}
 \Top(Y_b,Z_b)&\cong&\Top(\pi^{t}_{b}(\colim T_Y),Z_b)  &\\
              &\cong&\Top(\colim  \pi^{t}_{b}T_Y,Z_b)  &\mbox{(by Lemma \ref{l8}.2)}\\  
              &\cong&\Top(\underset{(W\st{\sigma}{\rightarrow}Y)\in \mw|Y}\colim W_b,Z_b)  &\\
       &\cong& \underset{(W\st{\sigma}{\rightarrow}Y)\in \mw|Y}\lm \Top( W_b,Z_b)&\\
         &\cong& \underset{(W\st{\sigma}{\rightarrow}Y)\in \mw|Y}\lm \left|\Hom(W,Z)\right|_b&\mbox{(by Lemma \ref{l17})}\\
				     &\cong& \underset{(W\st{\sigma}{\rightarrow}Y)\in \mw|Y}\lm \left|S^{Y}_{Z}(\sigma)\right|_b& \\
						&\cong& \lim \pi^{s}_{b}\abs{S^{Y}_{Z}}\\
	\end{array}$$
	By Lemma \ref{l7}.1, $\left|S^{Y}_{Z}\right|$ has a limit.
	Therefore  by Lemma \ref{pl6}, $S^{Y}_{Z}$ has a limit.
	\qedhere   
 \end{proof}
\begin{remark}\label{r3}
Let $\mw$ be as in Theorem \ref{t6} and 
\begin{equation}\label{eq65}
 \hom(.,.) : \mw_l[\Top_B]^{op}\times \mw_l[\Top_B] \ra  \Top_B  
\end{equation}
be the functor defined by (\ref{peq13}). Let $Y,Z\in \mw_l[\Top_B]$.
$$ \begin{array}{rclr}
 \abs{\hom(Y,Z)}_b&\cong&\pi^{s}_{b}(\abs{\hom(Y,Z)})  &\\
              &\cong&\pi^{s}_{b}(\abs{\lim S^{Y}_{Z}})  &\\ 
							&\cong&\pi^{s}_{b}(\lim \abs{S^{Y}_{Z}}) & (\absl{} \mbox{ preserves limits}) \\
              &\cong&\lim  \pi^{s}_{b}(\abs{S^{Y}_{Z}}) &\mbox{(by Lemma \ref{l7}.1)}
	\end{array}$$
That is, $\lim  \pi^{s}_{b}(\abs{S^{Y}_{Z}})\cong \Top(Y_b,Z_b)$. It follows from Lemma \ref{l16}.2  that $\hom(Y,Z)$ is the topological space whose underlying set is $\underset{b\in B}{\coprod}\Top(Y_b,Z_b)$ and whose topology is the initial topology induced from the spaces $\Hom(W,Z)$ by the maps
$\underset{b\in B}{ \coprod}\sigma_b:\underset{b\in B}{\coprod}\Top(Y_b,Z_b) \ra \underset{b\in B}{\coprod}\Top(W_b,Z_b)=\abs{\Hom(W,Z)}$, where $ (W\st{\sigma}\sra Y) \in \mw/Y$. 
\end{remark}

  By substituting $Pt$ for $B$,  Theorem \ref{t3} corresponds under the isomorphism $P$ of (\ref{peq16}) to the following celebrated theorem of Day. 

\begin{corollary}\label{pc13}(\cite[Theorem 3.1]{D})\\
 Assume that:
\begin{enumerate}

\item The subcategory $\mw$ of $\Top$ is suitable.  
\item Every  space in $\mw$ is exponentiable as an object of $\Top$.
\item For every $V,W \in \mw$, the  space  $V\times_{\Top} W$ is $\mw$-generated.
\end{enumerate}
 Then $\mw$ is left Kan extendable. Furthermore, $\mw_l[\Top]$ is a cartesian closed subcategory of $\Top$.  
\end{corollary}

\begin{remark}\label{pr9}
Let $\mw$ be as in Corollary \ref{pc13} and $Y,Z\in \mw_l[\Top]$. By Remark \ref{r3},
$\lm \abs{S^{Y}_{Z}}$ exists and is isomorphic to $\Top(Y,Z)$. 
Therefore by Lemma \ref{l13}, $\hom(Y,Z)= \lm S^{Y}_{Z}$  is the topological space whose underlying set is $\Top(Y,Z)$ and whose topology is the initial topology defined by the functions
\begin{equation}\label{peq21}
	 \Top(Y,Z)\xra{\Top(\sigma,Z)} \abs{S^{Y}_{Z}(\sigma)}=\Top(W,Z) 
	\end{equation}
By (\ref{peq15}), the exponential object $Z^Y$ is given by $Z^Y\cong F_{L}(\hom(Y,Z))$, where $F_{L}: \Top \ra \mw_l[\Top]$ is the coreflector. \\
 \end{remark}

\begin{examples}\label{pex6}
Let $\Comp$ be the subcategory of  $\Top$ of compact Hausdorff spaces.
\begin{enumerate} 
 	 
\item  By Corollary \ref{pc13}, $\Comp$ is left Kan extendable and $\Comp_{l}[\Top]$ is a cartesian closed coreflective subcategory of $\Top$. The $\Comp$-generated objects of $\Top$ are precisely the compactly generated spaces so that we recover  (\cite[Theorem 3.1]{D}  and \cite[page 49]{MP}). Let $\kTop=\Comp_{l}[\Top]$ and $k:\Top \ra \kTop$  a coreflector. By Corollary \ref{c14}, $\kTop$ contains every locally compact Hausdorff space.
We next give a description of the internal hom functor of $\kTop$.\\ 
Recall that if $K$ is compact Hausdorff and $Z$ is any space, then the exponential object $\Hom(K,Z)$ is the topological space whose underlying set is $\Top(K,Z)$ and whose topology is generated by the subsets 
\begin{equation}\label{peq }
 (C,V)=\{f\in \Top(K,Z)\mid f(C)\subset V\} 
\end{equation}
where $C$ is closed in $K$ and $V$ is open in $Z$ \cite[Proposition A.14.]{HA}.\\
For $\sigma: K  \ra Y$ continuous, the pull back of the subsets $(C,V)$ of $\Top(K,Z)$ by the maps 

\begin{equation}\label{peq17}
\Top(Y,Z)\xra{\Top(\sigma,Z)} \Top(K,Z)
\end{equation}
are the subsets  
\begin{equation}\label{eq18}
 (C,\sigma,V)=\{f\in \Top(Y,Z)\mid f\sigma(C) \subset V\}
\end{equation}
where $C$ is any compact Hausdorff space, $V$ is any open subset of $Z$ and $\sigma:C\ra Z$ is any continuous map. 
Let $\hom(Y,Z)$ be the topological space whose underlying set is $\Top(Y,Z)$ and whose topology is generated by the subsets $(C,\sigma,V)$.
By Remark \ref{pr9}, the exponential object $Z^Y$ in the cartesian closed category $\kTop$ 
is given by 
\begin{equation}\label{peq19}
 Z^Y=k(\hom(Y,Z))
\end{equation}

\item 
Assume that $B$ is   Hausdorff.
The category $\Comp/B$ is suitable. By Theorems \ref{t5} and \ref{t1}, every object in $\Comp/B$ is exponentiable in $\Top_B$. The base space $B$ is Hausdorff, therefore the diagonal of $B$ is closed. It follows that the product, in $\Top_B$, of two objects of $\Comp/B$ is again in $\Comp/B$.
  By Theorem \ref{t6}, the subcategory $\Comp/B$ of $\Top_B$ is left Kan extendable and $(\Comp/B)_l[\Top_B]$ is cartesian closed. By Proposition \ref{p21}.2, $(\Comp/B)_l[\Top_B]=\kTop/B$. Thus $\kTop/B$ is cartesian closed.
	We therefore recover a theorem of Booth (\cite[Theorem 1.1]{BP}). 
\end{enumerate} 

\end{examples}
\indent We next use the terminology developed in this paper to state another result due to Day and compare it to Theorem \ref{t6}.

\begin{theorem}\label{t10}(Day \cite[Theorem 3.4]{D}).\\
Let $\me$ be a subcategory of $\Top$ such that:
\begin{enumerate}
  
\item The subcategory $\me$ contains the one point space. 
\item  Each object of $\me$ is an exponentiable object of $\Top$.  
\item  For any two fibrewise spaces $p:V\ra B$ and $q:W\ra B$ in
 $\me/B$, the domain of the  product $p\tm_{Top_B}q$ (in  
$Top_B$) is closed in $V\tm_{Top}W$.
  \end{enumerate}
Then $\me/B$ is left Kan extendable in $Top_B$ and $(\me/B)_l[Top_B]$ is cartesian closed. 
\end{theorem}

  Theorems \ref{t6} and \ref{t10} do overlap. Actually, the proof of \cite[Theorem 1.1]{BP} given in Example \ref{pex6}.2, and which uses Theorem \ref{t6}, can also be derived from  Theorem \ref{t10}. There are however some essential differences:

\begin{enumerate}
 
\item The subcategory $\mw$ of $\Top_B$ in Theorem \ref{t6} has the form $\me/B$ in Theorem \ref{t10}. Not any subcategory of $\Top_B$ has this form.
\item  Objects of $\mw$ in Theorem \ref{t6} are assumed to be exponentiable in $\Top_B$, while the objects of $\me$ in Theorem \ref{t10} are assumed to be exponentiable in $\Top$. For instance, Theorem \ref{t6} can be used to prove that the category of fibrewise compactly generated spaces over a $T_1$-space is cartesian closed as shown in a latter section. Theorem \ref{t10} does not apply to prove this fact. 
\item  In Theorem \ref{t6}, $B$ is assumed to be $T_1$. Theorem \ref{t10} uses a different separation condition (condition 3).
\end{enumerate}
Observe that Theorem \ref{t10} can be derived from Theorem \ref{t6} when $B$ is Hausdorff. 
\section{The category $\kTop_B$ of fibrewise compactly generated spaces}\label{s12}
Our objective in this section is to prove that the category of fibrewise compactly generated spaces over a $\T_1$-base is cartesian closed. \\
\indent Let $\Comp_B$ be the subcategory of $\Top_B$ of fibrewise compact, fibrewise Hausdorff spaces over $B$. 

\begin{proposition}\label{p15}
 Assume that $B$ is a $\T_1$-space. Then $\Comp_B$ is left Kan extendable in $\Top_B$.
\end{proposition}
   
\begin{proof} 
$B$ is a $\T_1$-space. Thus   $\Comp_B$ contains the fibrewise spaces $B^b$ of Example \ref{ex3} for all $b\in B$. Therefore $\Comp_B$ is  a suitable subcategory of $\Top_B$. By Theorem \ref{t2}, $\Comp_B$ is left Kan extendable  in $\Top_B$.
\qedhere   
\end{proof}	
  Assume that $B$ is $\T_1$.  Then  $\kTop_B=({\Comp_B})_{_{l}}[\Top_B]$ is a coreflective subcategory of $\Top_B$.
Let
\begin{equation}\label{eq35}
  k:\Top_B\ra \kTop_B  
\end{equation}
 
be a coreflector. An object  in $\kTop_B$ is called a fibrewise compactly generated space over $B$. 
\begin{proposition}\label{pp21}
 Assume that $B$ is a $\T_1$-space and let $X$ be a fibrewise Hausdorff space over $B$. Then  the following properties are equivalent:
 
\begin{enumerate}
 
\item The fibrewise space $X$ is fibrewise compactly generated.
\item If a subset $A$ of $X$ is such that $A\cap K$ is open   in $K$ for any subspace $K$ of $X$ which is fibrewise compact over $B$, then $A$ is open   in $X$.
\item If a subset $A$ of $X$ is such that $A\cap K$ is closed  in $K$ for any subspace $K$ of $X$ which is fibrewise compact over $B$, then $A$ is  closed  in $X$. 
\end{enumerate}

\end{proposition}	
   
\begin{proof}
Let $u:K\sra X$  be a continuous fibrewise map with $K$ fibrewise compact  over $B$. 
By Proposition \ref{p4}.2, $u(K)$ is fibrewise compact fibrewise Hausdorff. The result then follows from Theorem \ref{t2}.3. 
\qedhere   
\end{proof}
Recall that if $X$ is a fibrewise space over $B$ with projection $p:X\ra B$ and $W\subset B$, then the subspace $p^{-1}(W)$ of $X$ is denoted by $X_W$. 

\begin{definition}\label{pd9}(\cite[Definition  10.1]{J2})\\
 Let $X$ be fibrewise space over $B$. Then a subset $A$ of
$X$ is said to be quasi-open (resp. quasi-closed) if the following condition is satisfied:\\
For each point $b\in B$ and each neighborhood $V$ of $b$, there exists a neighborhood 
 $W \subset  V$ of $b$ such that whenever $K \subset X_W$ is fibrewise compact over $W$, then
$A\cap K$ is open (resp. closed) in $K$. 
\end{definition}	

 \begin{lemma}\label{pl9} Let $X$ be a topological space and $(V_i)_{i\in I}$ a family of subsets of $X$ whose interiors cover $X$. Then a subset $A$ of $X$ is open (resp.closed) iff $A\cap V_i$ is open (resp.closed) in $V_i$ for all $i\in I$. 
\end{lemma}	
   
\begin{proof}
Clear.
\qedhere   
\end{proof}

\begin{corollary}\label{pc20}\footnote{I would like to greatly thank the first referee for explicitly stating this result to me. }
Assume that $B$ is a $\T_1$-space. Let $X$ be a fibrewise compactly generated fibrewise Hausdorff space over $B$. Then every quasi-open (resp. quasi-closed) subset of $X$ is open (resp.closed).
\end{corollary}
  
\begin{proof}
The two claims  concerning quasi-open sets and quasi-closed sets are equivalent. We therefore only need to prove one of them.\\
\indent Let $O$ be a quasi-open subset of $X$. For each $b\in B$, there exists a neighborhood $W_b$ of $b$ such that given any  subspace $K$ of $X_{W_b}$ which is fibrewise compact over $W_b$, $O\cap K$ is open in $K$.\\
\indent Let $K$ be any  subspace of $X$ which is fibrewise compact over $B$ and let $b\in B$. The fibrewise subspace $K\cap X_{W_b}$ is  fibrewise compact over $W_b$. Therefore 
$K\cap X_{W_b}\cap O$ is open in $K\cap X_{W_b}$. The family  $(K\cap X_{W_b})_{b\in B}$ is a family of subsets of $K$ whose interiors in $K$ cover $K$. By Lemma \ref{pl9}, $K\cap O$ is open in $K$. By Proposition \ref{pp21},  $O$ is open in $X$. 
\qedhere   
\end{proof}
Recall that a topological space $X$ is said to be regular if for every $x\in X$ and every neighborhood $V$ of $x$, there exists a closed neighborhood $W$ of $x$ which is contained in $V$. Observe that a regular $\T_1$-space is Hausdorff. 

\begin{proposition}\label{pp22}
 Let $B$ be a regular Hausdorff space and $X$ a fibrewise Hausdorff space over $B$.  Assume that every quasi-open (resp.quasi-closed) subset of $X$ is open (resp.closed) in $X$. Then $X$ is fibrewise compactly generated.
\end{proposition}
   	
\begin{proof} Again, we only need to prove the proposition under the quasi-open hypothesis. \\
Let $O\subset X$ be such that $O\cap K$ is open in $K$ for any subspace $K$ of $X$ which is fibrewise compact over $B$. We want to show that $O$ is quasi-open. \\
Let $b\in B$ and let $V$ be any neighborhood of $b$. The space $B$ is regular. There exists a closed neighborhood $W$ of $b$ which is contained in $V$. Let $K$ be any subspace of $X_W$ which is fibrewise compact  over $W$. The subspace $W$ of $X$ is closed. By Theorem \ref{t5}, $K$ is fibrewise compact over $B$. Therefore  $O\cap K$ is open in $K$. It follows that $O$ is a quasi-open subset of $X$, and is therefore open in $X$.  By Proposition \ref{pp21}, $X$ is fibrewise compactly generated. 
\qedhere   
\end{proof}

\begin{remark}\label{pr11}We next compare our notion of
fibrewise compactly generated space to the equally named notion
considered by James in \cite[Definition 10.3]{J2}.

 \begin{enumerate}
 
\item Our notion of fibrewise compactly generated spaces is defined only when the base space $B$ is $T_1$.

\item  A fibrewise space $X$ over $B$ is fibrewise compactly generated in the sense of James iff:
 
\begin{enumerate}
 
\item $X$ is fibrewise Hausdorff.
\item Every quasi-open subset of $X$ is open, or equivalently, if every quasi-closed subset of $X$ is closed. 

\end{enumerate}

\item   Assume that $B$ is a $\T_1$-space and $X$ is a fibrewise Hausdorff space over $B$.
 
\begin{enumerate}
 
\item By Corollary \ref{pc20}, if $X$ is fibrewise compactly generated   in our sense, then it is so   in the sense of James. 
\item Assume further that $B$ is a regular space. Then by Corollary \ref{pc20} and Proposition \ref{pp22}, $X$ is fibrewise compactly generated in our sense iff it is so in the sense of James.   
\end{enumerate}

\end{enumerate}

\end{remark}	
To fit our purposes, we give a definition of fibrewise locally compact spaces which is slightly stronger than the one given by James in \cite[Definition 3.12.]{J2}. 

\begin{definition}\label{d6}
A fibrewise space $X$ over $B$ is said to be fibrewise locally compact if for each $x\in X$, there exists a neighborhood $K$ of $x$ which is  fibrewise compact over $B$. 
\end{definition}

\begin{proposition}\label{p17}
Assume that $B$ is $\T_1$.  Then
every fibrewise locally compact, fibrewise Hausdorff space is fibrewise compactly generated.
\end{proposition}
  
\begin{proof} 	
This is a consequence of Corollary \ref{c14}.
\qedhere  
 \end{proof}

Assume that $B$ is $\T_1$ and let $\Lcomp_B$ be the subcategory of $\Top_B$ of fibrewise locally compact, fibrewise Hausdorff spaces. $\Lcomp_B$ contains the suitable subcategory $\Comp_B$ of $\Top_B$. Therefore  $\Lcomp_B$ is suitable.
By Theorem \ref{t2}, $\Lcomp_B$ is left Kan extendable in $\Top_B$.
Let $\lkTop_B=({\Lcomp_B})_{_{l}}[\Top_B]$.  
  
\begin{corollary}\label{c3}
Assume that $B$ is $\T_1$. Then 
$\lkTop_B=\kTop_B$.
 \end{corollary}
   
\begin{proof} 	
 $\Comp_B$ is a subcategory of  $\Lcomp_B$ and by Proposition \ref{p17}, $\Lcomp_B$ is a subcategory of $\kTop_B$. Therefore by Corollary \ref{pc5}, $\lkTop_B=\kTop_B$.
\qedhere    
 \end{proof}
The next result generalizes Proposition \ref{p17} and is a fibrewise version of \cite[Proposition 2.6]{S}. 

\begin{proposition}\label{p18} Assume that $B$ be is  $\T_1$. Let X be a fibrewise locally compact fibrewise Hausdorff space and $Y$  a fibrewise compactly generated space. Then the product   $X\tm_{\Top_B} Y$ is fibrewise compactly generated.
 \end{proposition}
   
\begin{proof} 	
This follows from  Lemma \ref{pl6} and Corollary \ref{c3}.
\qedhere   
 \end{proof}

\begin{theorem}\label{t7}
 Assume that $B$ is $\T_1$. Then $\kTop_B$ is cartesian closed.
 \end{theorem}
   
\begin{proof}
$B$ is $\T_1$, thus $\Comp_B$ is a suitable subcategory of $\Top_B$. By Theorem \ref{t1}, every fibrewise compact fibrewise Hausdorff space is exponentiable in $\Top_B$. By Corollary \ref{c7}, the product of two fibrewise compact spaces is fibrewise compact. By Examples \ref{ex1}, the subcategory of fibrewise Hausdorff spaces over $B$ is reflective. Therefore by Proposition \ref{pp13}.1.(a),   the product (in $\Top_B$) of two fibrewise Hausdorff spaces is fibrewise Hausdorff. It follows from Theorem \ref{t6} that $\kTop_B$ is cartesian closed.
\qedhere   
 \end{proof}
We next give a description of the internal $\hom$ functor of $\kTop_B$.\\

 Let $K\in \Comp_B$  and $Z \in \Top_B$. By Theorem \ref{t1}, $K$ is exponentiable in $\Top_B$ and the exponential object  $\map_B(K,Z)$  is the topological space whose underlying set is $\underset{b\in B}{\coprod}\Top(K_b,Z_b)$ and whose topology is generated by the subsets 
\begin{equation}\label{eq37}
(C,O,\Omega)=\underset{b\in \Omega}{\coprod}\{\gamma \in \Top(K_b,Z_b) \mid \gamma(C_b)\subset O_b\}
\end{equation}
where $C$ is closed in $K$, $O$ is open in $Z$ and $\Omega$ is open in $B$.\\
Let 
\begin{equation}\label{eq38}
\hom(.,.) : \kTop_B^{op}\times \kTop_B \ra  \Top_B
\end{equation}
be the functor defined as in (\ref{eq65}) and let $Y,Z \in k\Top_B$. By Remark \ref{r3}, $\hom(Y,Z)$ is the topological space whose underlying set 
is $\underset{b\in B}{\coprod}\Top(Y_b,Z_b)$ and whose topology is generated by the subsets 
\begin{equation}\label{eq39}
(\sigma,C,O,\Omega)=\underset{b\in \Omega}{\coprod}\{\gamma\in \Top(Y_b,Z_b)\mid \gamma \sigma_b(C_b)\subset O_b\}
\end{equation}
Where $(\sigma: K\ra Y) \in \Comp_B/Y$, $C$ closed in $K$ and $O$ open in $Z$.
By Theorem \ref{pt10}, the composite functor 
\begin{equation}\label{eq40}
(.)^{(.)}:\kTop_B^{op}\times \kTop_B \xra{\hom} \Top_B \st{k}\ra \kTop_B
\end{equation}
is an internal $\hom$ functor for the cartesian closed category $\kTop_B$,
 where $\Top_B \st{k}\ra \kTop_B$ is a coreflector. 

\begin{proposition}\label{p9} Assume that $B$ is $\T_1$
and let $\Topo_B$ be one of the reflective subcategories 
\begin{equation}\label{eq41}
 \fTop_B,\hTop_B,\uTop_B, \hkTop_B \text{  or  }  \hwTop_B.  
\end{equation} 
Then
\begin{enumerate}
 
\item $\Comp_B$ is left Kan extendable as a subcategory of $\Topo_B$.
 \item $(\Comp_B)_l[\Topo_B]=\Topo_B\cap\kTop_B$.
\item  A reflection of $\Top_B$ on $\Topo_B$ induces a reflection of $\kTop_B$ on $\Comp_l[\Topo_B]$.
 \item  The coreflection of $\Top_B$ on $\kTop_B$ given by Proposition \ref{pp17} induces a coreflection of $\Topo_B$ on $(\Comp_B)_l[\Topo_B]$.
 
 \end{enumerate}
\end{proposition}
  
\begin{proof} $\Topo_B$ is reflective, closed under subobjects subcategory of $\Top_B$ containing the suitable subcategory $\Comp_B$. Properties 1-4 are then consequences of Proposition \ref{p14}. 
\qedhere   
 \end{proof}	
 
\begin{corollary}\label{pc15}
Let  $\Topo$ be one of the reflective subcategories 

\begin{equation}\label{eq37}
 \fTop,\hTop,\uTop,\hcTop,\hkTop  \text{  or  }  \hwTop  
\end{equation} 
of $\Top$. Then
\begin{enumerate}
 
\item $\Comp$ is left Kan extendable as a subcategory of $\Topo$.
\item  A reflection of $\Top$ on $\Topo$ induces a reflection of $\kTop$ on $\Comp_l[\Topo]$.
 \item $\Comp_l[\Topo]=\Topo\cap\kTop$.
 \item  The coreflection of $\Top$ on $\kTop$ given by Proposition \ref{pp17} induces a coreflection of $\Topo$ on $\Comp_l[\Topo]$. 
 
 \end{enumerate}
\end{corollary}

\begin{notation}\label{n2}\mbox{}
\begin{enumerate}
 
\item Let $\kfTop_B=\kTop_B\cap\fTop_B$, $\khTop_B=\kTop_B\cap\hTop_B$, $\kuTop_B=\kTop_B\cap\uTop_B$ and $\khkTop_B=\kTop_B\cap\hkTop_B$.
 \item Similarly, let
 $\kfTop=\kTop\cap\fTop$, $\khTop=\kTop\cap\hTop$, $\kuTop=\kTop\cap\uTop$, $\khcTop=\kTop\cap\hcTop$, $\khkTop=\kTop\cap\hkTop$ and    $\khwTop=\kTop\cap\hwTop$. 
\end{enumerate}
  
 \end{notation}
\section{Cartesian closed subcategories of $\kTop_B$ and $\kTop$}  \label{s13}
In this section, we prove that $\kfTop_B$ is a cartesian closed subcategory of $\kTop_B$. We also prove that the subcategories $\kfTop,$ $\khTop,$ $\kuTop,$ $\khcTop,$     and  $\khwTop $ are cartesian closed. 
  
\begin{proposition}\label{p25} Assume that $B$ is $\T_1$. Then 
the reflective subcategory $\kfTop_B$
of $\kTop_B$ is cartesian closed with internal $\hom$ functor induced by that of $\kTop_B$.
 \end{proposition}
  
\begin{proof}  By Proposition \ref{p9} , $\kfTop_B$ is a reflective subcategory of $\kTop_B$. 
By Remark \ref{pr4}, we just need to prove that if $Y,Z \in \kfTop_B$, then the exponential object $Z^Y$ in $\kTop_B$ defined by (\ref{eq40}) is again an object of $\kfTop_B$. \\
So let $Y,Z\in \kfTop_B$.
\begin{equation}\label{eq46}
\hom(Y,Z)\cong \underset{(K\st{\sigma}{\rightarrow}Y)\in \Comp_B|Y}\lm \map_B(K,Z).
\end{equation}
By Proposition \ref{p26}, the spaces $\map_B(K,Z)$ in (\ref{eq46}) are $\T_1$. The subcategory $\fTop_B$ is a reflective,  by Proposition \ref{pp13}.1.(a), $\hom(Y,Z)$ is $\T_1$. 
Let $\epsilon$ be the counit of the coreflection of $\Top_B$ on $\kTop_B$.
The $\hom(Y,Z)$-component 

 \begin{equation}\label{eq47}
 Z^Y=k(\hom(Y,Z))\xra{\epsilon_{\hom(Y,Z)}}\hom(Y,Z) 
\end{equation} 
\mbox{}\\
of $\epsilon$ is monic. The category $\fTop_B$ is closed under subobjects, therefore $Z^Y\in \fTop_B$. The space
$Z^Y\in \kTop_B$, thus $Z^Y\in \kTop_B \cap \fTop_B=\kfTop_B$.
\qedhere    
\end{proof}

\begin{remark}\label{r10}\mbox{}
\begin{enumerate}

 \item Assume that $B$ be a $\T_1$-space. Let $K$ be a fibrewise compact, fibrewise Hausdorff space and  $Z\in \khTop_B$. Then as observed in Remark \ref{r9}, the space  $\map_B(K,Z)$ may not be fibrewise Hausdorff. Therefore the argument used in the proof of Proposition \ref{p25} cannot be used to prove that $\khTop_B$ is cartesian closed. In fact, this does not seem to be true. 
\item Let $\mk$ be the subcategory of $\Top_B$ of compactly generated spaces 
in the sense of James (\cite[Definition 10.3]{J2} and Remark \ref{pr11}.2). Then $\mk$ is cartesian with binary product $\tm'_B$ defined in (\cite[Page 83]{J2}. For any $X\in \mk$ that is locally sliceable \cite[Definition 1.16]{J2}, the functor  
$$-\tm'_B X:\mk \ra \mk$$
  has a right adjoint which is the functor 	$$\map'_B(X,-):\mk \ra \mk $$
	defined in \cite[Page 84]{J2}. This follows
from the fact that the evaluation functions 
$$\map'_B(X,Z)\tm'_B X \ra Z$$
are continuous \cite[Page 85]{J2}  and that the adjoint of a continuous
function $$h: Y \tm'_B X \ra Z$$ can be regarded as a continuous function
$$ k_B(\widehat{h}):Y\ra \map'_B(X,Z)$$
 as in  \cite[Lemma 10.16]{J2}. 
\end{enumerate} 
 \end{remark}

\begin{proposition}\label{pp14}\cite[Proposition 11.4.]{RC}\\
 Every compactly generated, $k$-Hausdorff space is weak Hausdorff.  
That is,  $$\kTop \cap \hkTop \subset \hwTop.$$
 \end{proposition}
  
\begin{proof}
Let $X$ be a compactly generated $k$-Hausdorff space. Let 
$f:K\ra X$ be continuous where $K$ is compact Hausdorff. Let
$g:L\ra X$ be a continuous map from a compact Hausdorff space $L$ to $X$, 
$f\tm_{\Top} g: K\tm_{\Top} L \ra X\tm_{\Top} X$ and $pr_L:K\tm_{\Top} L \ra L$ is the projection.
The map $pr_L$ is closed, therefore
 $g^{-1}(f(K))=pr_L((f\tm_{\Top} g)^{-1})(\Delta_X)$ is closed. It follows that $f(K)$ is closed and $X$ is weak Hausdorff.
\qedhere   
\end{proof}

\begin{remark}\label{pr8}
Observe that by Propositions \ref{pp14} and \ref{p12}, $\khkTop=\khwTop$.  
 \end{remark}

\begin{proposition}\label{pp15}
The reflective subcategories
\begin{equation}\label{peq30}
 \kfTop,\khTop,\kuTop,\khcTop,\khkTop   \text{  and  }  \khwTop  
\end{equation} 
of $\kTop$ are cartesian closed with internal $\hom$ functor induced by that of $\kTop$. 
 \end{proposition}

\begin{proof} Let $\Topo$ be one of the categories in (\ref{peq30}). By
Proposition \ref{pc15}, $\Topo$ is a reflective subcategory of $\kTop$. 
By Remark \ref{pr4}, we just need to prove that if $Y,Z \in\Topo$, then the exponential object $Z^Y$ in $\kTop$ defined by Examples \ref{pex6}.1  is again in $\Topo$.  
So let $Y,Z\in \Topo$. For $y\in Y$, the  evaluation map $Ev_y:Z^Y \ra Z$  at $y$ and is continuous.
\begin{itemize}
 
\item $\Topo=\kfTop$:\\
Let $f_0 \in Z^Y$. $Ev_y$ is continuous, $Z$ is Fréchet,  thus 
$Ev^{-1}_{y}(f(y))$ is closed in $Z^Y$. It follows that $\{f_0\}= \underset{\rm y\in Y}{\rm \cap}Ev_{y}^{-1}(f_0(y))$ is closed in $Z^Y$ and $Z^Y$ is a Fréchet space.

\item $\Topo =\khTop$:\\ Let $f,g\in \khTop$ with $f\neq g$. Let $y_0 \in Y$ be such that $f(y_0)\neq g(y_0)$. Let $U,V$ be disjoint open neighborhoods of $f(y_0)$ and $g(y_0)$. Then $Ev ^{-1}_{y_0}(U)$ and  $Ev ^{-1}_{y_0}(V)$ are disjoint open neighborhoods of $f$ and $g$. It follows that
$Z^{Y}$ is Hausdorff.

\item $\Topo =\kuTop$:\\ Let $f,g\in \khTop$ with $f\neq g$. Let $y_0 \in Y$ be such that $f(y_0)\neq g(y_0)$. Let $A,B$ be disjoint closed neighborhoods of $f(y_0)$ and $g(y_0)$. Then $Ev ^{-1}_{y_0}(A)$ and  $Ev ^{-1}_{y_0}(B)$ are disjoint closed neighborhoods of $f$ and $g$. It follows that
$Z^{Y}$ is Urysohn.

\item $\Topo =\khcTop$:\\
Let $f,g\in \khcTop$ with $f\neq g$. Let $y_0 \in Y$ be such that $f(y_0)\neq g(y_0)$. $Z$ is completely Hausdorff, thus there exists a continuous fonction $\psi: Z\ra [0,1]$ such that $\psi(f(y_0))=0$ and $\psi(g(y_0))=1$. 
$Ev_{y_0}$ is continuous, thus $\psi Ev_{y_0}: Z^Y\ra [0,1]$ is continuous. $\psi Ev_{y_0}(f)=0$ and  $\psi Ev_{y_0}(g)=1$.
 It follows that $Z^Y$ is completely Hausdorff.
\item $\Topo =\khkTop=\khwTop:$\\
For a topological space $X$, let $\Delta_X$ denote the diagonal of $X$.
Let $$f:K \ra Z^Y\tm_{\Top} Z^Y$$ be a continuous map, where $K$ is compact Hausdorff. Then 
\begin{equation}\label{eq49}
 f^{-1}(\Delta_{Z^Y})=\bigcap_{y\in Y}((Ev_y \tm_{\Top} Ev_y) f)^{-1}(\Delta_Z)
\end{equation}  
is closed in $K$. It follows that $Z^Y$ is $k$-closed.
 \end{itemize}
\qedhere   
 \end{proof}	

 \begin{proposition}\label{p39} Assume that $B$ is Hausdorff. Then   $\kTop/B\subset\kTop_B$.
 \end{proposition}
   
\begin{proof} 	
$B$ is Hausdorff, by Theorem \ref{t5}, $\Comp/B\subset \Comp_B$. By Examples \ref{pex6}.2, $\Comp/B$ is left Kan extendable and $(\Comp/B)_l[\Top_B]=\kTop/B$. It follows from Corollary \ref{pc5} that $$ \kTop/B= (\Comp/B)_l[\Top_B]\subset (\Comp_B)_l[\Top_B]=\kTop_B. $$
\qedhere      
 \end{proof}
\begin{proposition}\label{p33} Assume that $B$ is locally compact Hausdorff space. Then 
 $\kTop_B=\kTop/B$.
 \end{proposition}
  
\begin{proof}
By Proposition \ref{p39}, $\kTop/B\subset\kTop_B$.\\ 
Let $(X\st{p}\ra B)\in \Comp_B$, let $x_0\in K$, $b_0=p(x_0)$ and $K$ a compact neighborhood  of $b_0$.  Then $p^{-1}(K)$ is a neighborhood of $x_0$ which is Hausdorff.  By Proposition \ref{p32}, $p^{-1}(K)$ is compact. Therefore  $(p^{-1}(K)\st{p_{/}}\ra B) \in \Comp/B$. By Corollary \ref{c14}, $(X\st{p}\ra B)\in (\Comp/B)_l[\Top_B]$. It follows that $\Comp_B\subset (\Comp/B)_l[\Top_B]$.
By Corollary \ref{pc5},  
$$\kTop_B=(\Comp_B)_l[\Top_B]\subset (\Comp/B)_l[\Top_B]=\kTop/B.$$ 
Therefore $\kTop_B=\kTop/B.$ 
\qedhere     
 \end{proof}
Assume that $B$ is locally compact Hausdorff. Then by Example \ref{pex6}.1,
 $B\in\kTop$ and by Proposition \ref{p33}, the category $\kTop_B$ is just the slice category $\kTop/B$. The adjunction given by lemma \ref{l2} yields an adjunction   
\begin{equation}\label{eq47}
\begin{tikzcd}
             \kTop_B\arrow[r, shift left=1ex, "P_B"{name=G}] &  \kTop\arrow[l, shift left=.3ex, ".\tm_{\kTop}B" ].
 \end{tikzcd} 
\end{equation}
\section{Fibrewise sequential spaces}\label{s14}
It is a well known fact that the category of sequential spaces is cartesian closed. We here show that this fact extends to the fibrewise setting,  provided that the base space $B$ is Hausdorff.\\
\indent Let $\bn$ be the discrete space of non-negative integers, $\bn^{+}=\bn\cup\{\infty\}$ its one point compactification. Let $\N_B$ be the  subcategory of $\Top_B$ whose objects are continuous maps $\bn^{+} \ra B$.  

\begin{proposition}\label{p27} 
The subcategory $\N_B$ is a left Kan extendable in $\Top_B$.
\end{proposition}
  
\begin{proof} 	
The subcategory $\N_B$ is  suitable. By Theorem \ref{t2},	$\N_B$ is left Kan extendable. 
\qedhere    
 \end{proof}
We will call $\N_B$-generated objects  fibrewise sequential spaces. The category  $(\N_B)_{l}[\Top_B]$ of fibrewise sequential spaces will be  denoted by $\Seq_B$. 

\begin{remark}\label{pr6} 
Let $\Seq$ be the subcategory of $\Top$ that corresponds to $\Seq_{B}$ under the isomorphism $P$ of (\ref{peq16}). Then $\N$ is a dense subcategory of $\Seq$. The objects of $\Seq$ are called sequential spaces. 
\end{remark}
 
\begin{proposition}\label{p28}  
A fibrewise topological space $X \ra B$ is fibrewise sequential iff its domain $X$ is a sequential space.  
 \end{proposition}
  
\begin{proof} 	
 This is a consequence  of Corollary \ref{pc6} and lemma \ref{l2}.1.
\qedhere   
 \end{proof}

\begin{proposition}\label{p22} Assume that $B$ is $\T_2$. Then: 
\begin{enumerate}
 
\item $\Seq_B$ is cartesian closed.
 \item   $\Seq_B$ is a  coreflective subcategory of $\kTop_B$. Furthermore,
the inclusion functor $\Seq_B \hra \kTop_B$ preserves finite products. 
\end{enumerate}  
\end{proposition}
  
\begin{proof} \mbox{} 
\begin{enumerate}
 
\item 	The space $B$ is $\T_2$. By Theorem \ref{t5}, objects of $\N_B$ are fibrewise compact, fibrewise Hausdorff. By Theorem \ref{t1} they are exponentiable in $\Top_B$. The space $\bn^+$ is  a metric space. Let $\bn^+\st{p}\ra B$ and  $\bn^+\st{q}\ra B$ be two objects in $\N_B$. The domain of the product $p\tm_{\Top_B}q$ is a subspace of $\bn^+\tm_{\Top}\bn^+$ and is therefore a metric space.   It follows that the domain of $p\tm_{\Top_B}q$, which is a subspace of $\bn^+\tm_{\Top}\bn^+$,  is a metric space. A metric space is sequential, thus by Proposition \ref{p28}, 
$p\tm_{\Top_B}q \in \Seq_B$.
  By Theorem \ref{t6}, $\Seq_B$ is cartesian closed.
\item The base space $B$ is $\T_2$, therefore $\N_B$ is a subcategory of $\Comp_B$. By Corollary \ref{pc5}.1, $\Seq_B$ is a coreflective subcategory of $\kTop_B$, and by Corollary \ref{pc11}, the inclusion functor $\Seq_B \hra \kTop_B$ preserves finite products.\\
\end{enumerate}
\qedhere   
\end{proof}

\begin{remark}\label{pr7}\mbox{}
\begin{enumerate}
 
\item Assume that $B$ is a sequential space. Then by Proposition \ref{p28},  $\Seq_B$ is just the slice category $\Seq/B$ and the adjunction given by Lemma \ref{l2}.2 yields an adjunction
\begin{equation}\label{eq49}
\begin{tikzcd}
             \Seq_B\arrow[r, shift left=1ex, "P_B"{name=G}] &  \Seq\arrow[l, shift left=.3ex, ".\tm_{\Seq}B" ].
 \end{tikzcd} 
\end{equation}
\item Let $s:\Top \ra \Seq$ be a coreflector. Then the functor 
$\Seq_B \ra \Seq_{s(B)}$ which takes a fibrewise sequential space $X \st{p}\ra B$ to $ X\st{s(p)}\ra s(B)$ is an isomorphism of categories.
\item By the previous points, for any topological space $B$, the functor $P_B:\Seq_B \ra \Seq$
 is left adjoint. Its 
right adjoint	takes  a sequential space $X$ to the fibrewise sequential space $X\tm_{\Seq}s(B)$ whose projection is the composite map $$X\tm_{\Seq}s(B) \st{pr}\ra s(B) \st{\epsilon_B}\ra B,$$ 
where $s:\Top \ra \Seq$ is a coreflector and $\epsilon_B$ is the $B$-component of the counit $\epsilon$ of the coreflection of $\Top$ on  $\Seq$.
\end{enumerate}
 \end{remark}
\section{Fibrewise Alexandroff spaces}\label{s15}
The category of Alexandroff space is known to be equivalent to the cartesian category of preorders and is therefore cartesian closed (Escardó, Lawson \cite[Examples (2), page 114]{ES}). Our objective in this section is to extend this fact to the fibrewise setting.\\
\indent For $b\in B$, let
$$ 
\pi^{t}_{b}:\Top_B \ra \Top
$$
be the functor which takes a fibrewise space over $B$ to its fibre over $b$ as defined by (\ref{eq28}), and let
$$
i_b:\Top \ra \Top_B
$$
be the functor which takes a space $X$ to the fibrewise space whose domain is $X$ and whose projection $X\ra B$ is constant at $b$.

\begin{lemma}\label{l18} Let $b\in B$. Then   
\begin{enumerate}
 
\item The functor $i_b$ is left adjoint to $\pi^{t}_{b}$. 
 \item Assume that $\{b\}$ is closed  in $B$. Then  $\pi^{t}_{b}$ is  left adjoint to the functor

\begin{equation}\label{eq52}
 \begin{array}{rccl}
  \map(B^b,i_b(.)):& \Top & \ra &\Top_B \\
     &Y&\longmapsto & \map_B(B^b,i_b(Y)).
	\end{array} 
\end{equation}
 where $B^b$ is the fibrewise space defined by Example \ref{ex3}.
\end{enumerate}
\end{lemma}
   
\begin{proof}\mbox{} 
\begin{enumerate}
 
\item  Let $X\in \Top$ and $Y\in \Top_B$. Then 
 $\Top_B(i_b(X),Y)\cong \Top(X,\pi^{t}_{b}(Y)).$ 
 \item Let $X\in \Top_B$ and $Y\in \Top$. Then
$$ \begin{array}{rclr}
 \Top(\pi^t_b(X),Y)&\cong& \Top(X_b,Y)  &\\
              &\cong&\Top_B(X\tm_{\Top_B}B^b,i_b(Y))  &\\ 
							&\cong&\Top_B(X,\map_B(B^b,i_b(Y)))  &(\mbox{by  Theorem } \ref{t1}) 
             \end{array}$$
\end{enumerate}
\qedhere   
 \end{proof}
\begin{proposition}\label{p34}   
Let $E\in \Top$ be an exponentiable space and let $b\in B$ be such that $\{b\}$ is closed. Then $i_b(E)$ is an exponentiable object of $\Top_B$.  
\end{proposition}
    
\begin{proof} 
Let $$(.)^E:\Top \ra \Top$$ be a right adjoint of the functor 
$$.\tm_{\Top}E:\Top \ra \Top$$
and let $X,Y\in \Top_B$. We have  $$X\tm_{\Top_B}i_b(E)=i_b(\pi^{t}_{b}(X)\tm_{\Top}E).$$ Therefore 
$$ \begin{array}{rclr}
 \Top_B(X\tm_{\Top_B}i_b(E),Y)&\cong &\Top_B(i_b(\pi^{t}_{b}(X)\tm_{\Top}E),Y) &\\
              &\cong&\Top(\pi^{t}_{b}(X)\tm_{\Top}E,Y_b) & (\mbox{by Lemma \ref{l18}.1})\\ 
							&\cong&\Top(\pi^{t}_{b}(X),Y_b^E)  & \\
								&\cong&\Top_B(X,\map_B(B^b,i_b(Y^{E}_{b})))  & (\mbox{by Lemma \ref{l18}.2}) 
             \end{array}$$
Thus $i_b(E)$ is exponentiable in $\Top_B$. 
\qedhere   
\end{proof}

The Sierpinski space is the topological space denoted by $\bs$, whose underlying set is $\{0,1\}$ and whose set of open sets is $\mo(\bs)=\{\emptyset, \{1\},\bs\}.$ Let $\Sier$ be the  subcategory of $\Top$ having $\bs$ as its unique object and $\Sier_B$ the  subcategory of $\Top_B$ whose objects are all continuous maps  $\bs \sra B$.

\begin{proposition}\label{p35}
$\Sier_B$ is a left Kan extendable subcategory of $\Top_B$.
\end{proposition}
    
\begin{proof} 	
$\Sier_B$ is a suitable subcategory of $\Top_B$. By Theorem \ref{t2},	$\Sier_B$ is left Kan extendable.
\qedhere       
 \end{proof}
The subcategory $\Sier$ of $\Top$ corresponds to the subcategory $\Sier_{\Pt}$ of $\Top_{\Pt}$ under the isomorphism $P$ of (\ref{peq16}).
We therefore have the following.

\begin{corollary}\label{pc16}
$\Sier$ is  left Kan extendable subcategory in $\Top$. 
\end{corollary}

\begin{proposition}\label{pp18}  
A fibrewise topological space $X \ra B$ is $\Sier_B$-generated iff its domain $X$ is $\Sier$-generated. 
 \end{proposition}
	 
\begin{proof} 	
 This is a consequence  of Proposition \ref{pp17}.1  and lemma \ref{l2}.1.
\qedhere   
 \end{proof}

Recall that an Alexandroff space is a topological space in which arbitrary  intersections of open subsets are open. Equivalently, 
an Alexandroff space is a topological space for which arbitrary  unions of closed subsets are closed. Let $\Alex$ be the subcategory of $\Top$ of Alexandroff spaces. A finite topological space has only finitely many open sets, and is therefore an Alexandroff space. \\ 
Let $B$ be a subset of an Alexandroff space $X$ and let $\overline{B}$ denote the topological closure of $B$. The subspace $\underset{b\in B}{\bigcup}\overline{\{b\}}$ is a closed subset of $X$ containing
$B$. It follows that $\overline{B}=\underset{b\in B}{\bigcup}\overline{\{b\}}$. \\
\indent We next provide a simple proof of the following result which is given (without proof) in \cite[Examples (2), page 114]{ES}. 

\begin{proposition}\label{pp16}
 A topological space $X$ is $\Sier$-generated iff it is an Alexandroff space. That is $\Sier_{l}[\Top]=\Alex$.
\end{proposition}
  	
\begin{proof} 	
 Let $X$ be a $\Sier$-generated topological space, $(O_i)_{i\in I}$ a family of open sets in $X$ and $f:\bs \ra X$ a continuous map. Then
$f^{-1}(\underset{i\in I}{\bigcap} O_i)=\underset{i\in I}{\bigcap}f^{-1}(O_i)$ which open in $\bs$ since $\bs$ is an Alexandroff space. By Corollary \ref{pc14}.2, $\underset{i\in I}{\bigcap} O_i$ is open in $X$ and $X$ is an Alexandroff space.  
Conversely, assume that $X$ is an Alexandroff space. Let $B \subset X$ be such that $f^{-1}(B)$ is closed for every continuous map $f:\bs \ra X$  and let $a\in \overline{B}$. There exists $b\in B$ such that $a\in \overline{\{b\}}$. If $a=b$ then $a\in B$, if $a\neq b$, define $g:\bs \ra X$ by $g(0)=a$ and $g(1)=b$. Then 
$$g(\overline{\{1\}})=g(\bs)=\{a,b\}\subset \overline{\{b\}} \subset \overline{g(\{1\})}$$
Therefore $g$ is continuous and $g^{-1}(B)$ is a closed subset of $\bs$ containing $1$. It follows that $g^{-1}(B)=\bs$, in particular, $a=g(0)\in B$ and $B$ is closed in $X$. By Corollary \ref{pc14}.2, $X$ is $\Sier$-generated.
\qedhere   
\end{proof}

It follows that $(\Sier_B)_l[\Top_B]$ is the subcategory $\Alex_B$ of  fibrewise spaces $X \sra B$ whose domain $X$ is an Alexandroff space.
Objects of $\Alex_B$ are called fibrewise Alexandroff spaces over $B$.

\begin{corollary}\label{pc17}\mbox{}
\begin{enumerate}
 
 \item $\Alex$ is a coreflective subcategory of $\Top $ containing $\Sier$ as a dense subcategory.
\item $\Alex_B$ is a coreflective subcategory of $\Top_B $ containing $\Sier_B$ as a dense subcategory.
\end{enumerate}
\end{corollary}
   	
\begin{proof}  	
 This follows from Proposition \ref{pp16}, Proposition \ref{pp17},  Proposition \ref{pp17}.1  and Proposition \ref{pp18}.
\qedhere   
\end{proof}	
We next generalize   \cite[Lemma 4.6.]{ES}. 

\begin{proposition}\label{pp19}\mbox{}
\begin{enumerate}
 
 \item The Sierpinski space $\bs$ is sequential.
\item The category $\Alex$ is a coreflective subcategory of $\Seq$.
\item The category $\Alex_B$ is a coreflective subcategory of $\Seq_B$.
\end{enumerate}  
 \end{proposition}
	  
\begin{proof} 	
Define $q:\bn^{+} \ra  \bs$ by $q(\infty)=0$ and $q(n)=1$ for all $n\in \bn$. The map $q$ is a quotient map, thus  by Proposition \ref{p13}.1, $\bs$ is sequential. By Corollary \ref{pc5}, $\Alex$ is a coreflective subcategory of $\Seq$.  Similarly, $\Alex_B$ is a coreflective subcategory of $\Seq_B$.
\qedhere   
 \end{proof}

\begin{proposition}\label{p37}\mbox{}
\begin{enumerate}
 
\item The subcategory $\Alex$ of $\Top$ is cartesian closed.  
\item If $B$ is $\T_1$, then the subcategory $\Alex_B$ of $\Top_B$ is cartesian closed.
\end{enumerate}  
\end{proposition}
    	
\begin{proof} We just need to prove 2.\\
 \indent Being finite, $\bs$ is a core-compact space. It is therefore an exponentiable object of $\Top$. Let $\bs \st{p}\ra B$ be continuous.
Assume the space $B$ is $\T_1$, therefore $p$ is constant. By	Proposition \ref{p34}, $\bs \st{p}\ra B$ is an exponentiable object of $\Top_B$. The product in $\Top_B$ of two fibrewise Sierpinski spaces is a fibrewise Alexandroff space.  By Theorem \ref{t6},
 $\Alex_B$ is cartesian closed. 
\qedhere   
\end{proof}

\begin{remark}\label{r12}\mbox{}

\begin{enumerate}
 
\item Assume that $B$ is an Alexandroff space. Then $\Alex_B$ is just the slice category $\Alex/B$ and the adjunction given by Lemma \ref{l2}.2 yields an adjunction
\begin{equation}\label{eq54}
\begin{tikzcd}
             \Alex_B\arrow[r, shift left=1ex, "P_B"{name=G}] &  \Alex\arrow[l, shift left=.3ex, ".\tm_{\Alex}B" ].
 \end{tikzcd} 
\end{equation}
\item Let $a:\Top \ra \Alex$ be a coreflector. Then the functor 
$\Alex_B \ra \Alex_{a(B)}$ which takes a fibrewise Alexandroff space $X \st{p}\ra B$ to $ X\st{a(p)}\ra a(B)$ is an isomorphism of categories.
\item By the previous two points, for any topological space $B$, the functor $P_B:\Alex_B \ra \Alex$
 is left adjoint. Its 
right adjoint	takes  an Alexandroff space $X$ to the fibrewise Alexandroff space $X\tm_{\Alex}a(B)$ with projection the composite $$X\tm_{\Alex}a(B) \st{pr}\ra a(B) \st{\epsilon_B}\ra B.$$ 
Where $a:\Top \ra \Alex$ is a coreflector and $\epsilon_B$ is the $B$-component of the counit $\epsilon$ of the of the coreflection of $\Top$ on  $\Sier$. 	 
\end{enumerate}
 \end{remark}

\appendix
 \appendixpage
\addappheadtotoc
\section{Limits in a slice category}\label{A}
 The aim of this section is to prove that if $\mc$ is a bicomplete category and $b\in \mc$, then  the slice category $\mc/b$ of $\mc$  over $b$ is bicomplete.\\
\indent Let $F: \ma \ra \mc$ an $G: \mb \ra \mc$ be two functors. The comma category $F/G$ is defined to be the category whose objects are arrows
$F(a)\st{\alpha}\ra G(b)$ and whose morphisms from $F(a)\st{\alpha}\ra G(b)$
to $F(a')\st{\alpha'}\ra G(b')$ are pairs of morphisms $(f,g)\in \ma(a,a') \tm_{\Set}  \mb(b,b')$ rendering commutative the diagram 
\begin{equation}\label{eq21} 
 \begin{tikzcd}[sep=large]
    F(a)\arrow[d,"\alpha"']\arrow[r,"F(f)"]&F(a')\arrow[d,"\alpha'"]\\
    G(b) \arrow[r,"G(g)"']&G(b')\\
\end{tikzcd}
\end{equation} 
We have functors
\begin{equation}\label{eq22} 
\begin{matrix}
 P: F/G \ra \ma  &\mbox{and} &Q: F/G \ra \mb 
\end{matrix}
\end{equation} 
defined as follows: if $F(a)\st{\alpha}\ra G(b)\in F/G$, then 
$P(\alpha)=a$ and $Q(\alpha)=b$. If $(f,g)$ is a morphism from $F(a)\st{\alpha}\ra G(b)$
to $F(a')\st{\alpha'}\ra G(b')$ as in (\ref{eq21}), then $P((f,g))=f$ and $Q((f,g))=g$. 

\begin{notations}\label{n1} Let $F: \ma \ra \mc$ an $G: \mb \ra \mc$ be two functors.
\begin{enumerate}
 
\item If $\ma$ is a subcategory of $\mc$ and $F: \ma \ra \mc$ is the inclusion functor, then $F/G$ is also denoted by $\ma/G$.
    
\item If $\mb$ is a subcategory of $\mc$ and $G: \mb \ra \mc$ is the inclusion functor, then $F/G$ is also denoted by $F/\mb$. 
\item If $\ma,\mb$ are  subcategories of $\mc$ and $F: \ma \ra \mc$, $G: \mb \ra \mc$ are the inclusion functors, then $F/G$ is denoted by $\ma/\mb$.
 \item  If $\ma$ has just one object $*$ and just one morphisms $id_*$, then $F/G$ is denoted by $c/G$ where $c=F(*)$. If further $\mb=\mc$ and $G$ is the identity functor, then $c/G$ is called the slice category under $c$ and is denoted by $c/\mc$.

 \item  If $\mb$ has just one object $*$ and just one morphisms $id_*$, then $F/G$ is denoted by $F/c$ where $c=G(*)$. If further $\ma=\mc$ and $F$ is the identity functor, then $F/c$ is called the slice category over $c$ and is denoted by $\mc/c$. Observe that this notation is consistent with the one  previously used.

\end{enumerate}
\end{notations}
\indent Let $\mc$ be category, $b\in \mc$, $\mc/b$ the slice category of $\mc$ over $b$ and define 
\begin{equation}\label{eq19} 
 P_b:   \mc/b \ra \mc  
\end{equation}
to be the functor which takes an arrow-object $c\rightarrow b$ to its domain $c$.\\

\begin{lemma}\label{l2}\mbox{}  

\begin{enumerate}
 
	\item The functor $P_b$ creates colimits. In particular, if $\mc$ is cocomplete, then so is $\mc/b$.
	\item  Assume that the categorical product $c\tm_{\mc} b$ exists for every
	$c\in \mc$, then $P_b$ is left adjoint. In particular $P_b$ preserves colimits.
\end{enumerate}

\end{lemma}
   
\begin{proof}\mbox{}  
\begin{enumerate}
 
	\item  This follows from the dual of a straightforward generalization of  \cite[Lemma, page 121]{ML}. 
	\item The functor $\mc \ra \mc/b$ which takes an object $c\in \mc$ to the arrow $c\tm_{\mc} b \rightarrow b$ is a right adjoint of $P_b$. Thus $P_b$ preserves colimits. 
\end{enumerate}
\qedhere   
\end{proof}

Let $\Cat$ be the category of small categories. A poset carries a category structure in the standard way. Thus the ordinal numbers $1=\{0\}$ and $2=\{0,1\}$ may be viewed as small categories. The small category $1$ is a terminal object in $\Cat$.	
   The cone $\mi^{\triangleright}$ of $\mi\in \Cat$ is defined in \cite[Exercice 3.5.iv]{RE}  to be the pushout  in $\Cat$:
$$\begin{tikzcd}[sep=large]
 \arrow[dr, phantom, "\ulcorner", very near end]
    \mi\arrow[r,"*"]&1\arrow[d]\\
    \mi\tm_{\Cat} 2 \arrow[hookleftarrow,u,"i_1"]\arrow[r]&\mi^{\triangleright}
\end{tikzcd}$$
Let $i$ be the composite functor  $  \mi \st{i_0}\hra \mi\tm_{\Cat} 2 \rightarrow \mi^{\triangleright}$. Then $i:\mi \hra \mi^{\triangleright}$ is fully faithful and $\mi$ may be viewed as a full subcategory of $\mi^{\triangleright}$. Furthermore
\begin{enumerate}
 
	  \item  The category $\mi^{\triangleright}$ contains one more object than $\mi$, it is denoted by $*$. 
	\item  The set $\mi^{\triangleright}(i,*)$ contains precisely one morphism denoted by $\sigma_i$, $\forall i \in \mi$.
	\item The set $\mi^{\triangleright}(*,*)$ contains solely the identity morphism.
	\item  The set $\mi^{\triangleright}(*,i)$ is empty,  $\forall i \in \mi$.
\end{enumerate}

\indent Let $X: \mi \ra \mc/b$ be any functor. Define $X^{\triangleright}: \mi^{\triangleright} \ra \mc$ to be the unique functor satisfying the following properties:
\begin{enumerate}
 
	\item  The functor $X^{\triangleright}$ extends $P_bX$ over the category $\mi^{\triangleright}$. That is the following diagram commutes  
	$$\begin{tikzcd}[sep=large]
    \mi\arrow[r,"X"]&\mc/b\arrow[d,"P_b"]\\
    \mi^{\triangleright}\arrow[hookleftarrow,u,"i"]\arrow[r,"X^{\triangleright} "']&\mc\\
\end{tikzcd}
$$
\item   $X^{\triangleright}(*)=b$.
	\item  $X^{\triangleright}(\sigma_i)$ is the arrow $X(i)$ in $\mc$, $ i \in \mi$. 
	 \end{enumerate}
Then one has the following result.

 \label{l2,l9}

\begin{lemma}\label{l9}\mbox{}   

\begin{enumerate}
 
	\item The functor $X$ has a limit if and only if $X^{\triangleright}$ has a limit. Furthermore, a limiting cone $l\st{\lambda}\Ra X^{\triangleright}$
induces a limiting cone from  $\lambda_*:l\ra b$ to $X$, where 
$\lambda_*$ is the $*$-component of the cone $\lambda$.
\item If $\mc$ is complete, then so is $\mc/b$.  
\end{enumerate}
\end{lemma}
   
\begin{proof} Clear.
\qedhere   
\end{proof}
\begin{examples}\label{ex4}
Assume that $\mc$ is complete 
\begin{enumerate}
 
	\item Let $x\st{\sigma}\ra b, y\st{\tau}\ra b \in \mc/b$, $p_1:x\tm_{\mc} y \ra x$, $p_2:x\tm_{\mc} y \ra y$ be the projections  and let 
	$$\begin{tikzcd}
e\ar[r,"i"]&x\tm_{\mc} y \ar[r,shift left=.5ex,"\sigma p_1"]
  \ar[r,shift right=.5ex,swap,"\tau p_2"]
&b
\end{tikzcd}$$
be the equalizer (in $\mc$) of the maps $\sigma p_1$	and $\tau p_2$.
The diagram

\begin{center}
\begin{tikzcd}[sep=large]
e \arrow[r, "p_2i"] \arrow[d, "p_1i"']
\arrow[dr, phantom, "\scalebox{1.5}{$\lrcorner$}" , very near start, color=black]
& y \arrow[d, "\tau"] \\
x \arrow[r, "\sigma"'] & b \\
\end{tikzcd}
\end{center}
is a pullback diagram.
 By Lemma \ref{l9}, the composite $\sigma p_1 i=\tau p_2 i:e\ra b$ is the product of the objects 
$\sigma$ and $\tau$ of $\mc/b$. 
	
\item Using generalized equalizers, the previous example may be extended to the case where one has a family of arrow-objects $x_i\st{\sigma_i}\ra b$ of $\mc/b$ indexed be a small set $I$. 
\end{enumerate}

\end{examples}

\section{Limits in a slice category of sets}\label{B}
The aim of this section is to establish certain properties of limits and colimits in a slice   category of  sets.\\
\indent The category $\Set$ of (small) sets is a bicomplete category. For $X,Y,Z\in \Set$, there is a natural isomorphism
	\begin{equation}\label{eq31} 
\Set(X\tm_{\Set}Y,Z)\cong \Set(X,\Set(Y,Z)) 
\end{equation}
	 so that $\Set$ is cartesian closed.\\ 
Let  $E\in \Set$. The slice category of $\Set$ over $E$ is denoted by $\Set_E$. An object of $\Set_E$ is called a set over $E$. It consists of a set $X$ together with a function $p:X\ra E$ called  projection. A set $X\st{p}\ra E$ over $E$ is often identified with its domain $X$. Let
\begin{equation}\label{eq23} 
 P_E:   \Set_E \ra \Set  
\end{equation}
be the functor defined as in (\ref{eq19}).

\begin{proposition}\label{p20}\mbox{}

\begin{enumerate}
 
	\item  The category $\Set_E$ is bicomplete.
	\item  The functor $P_E$ creates and preserves colimits.
\end{enumerate}
\end{proposition}
   
\begin{proof}
This follows from Lemma \ref{l2}, Lemma \ref{l9} and the fact that $\Set$ is bicomplete.
\qedhere   
\end{proof}

Let $F\subset E$, $J^{s}_{F}:F \sra E$ the inclusion map and
\begin{equation}\label{eq29} 
J^{s,*}_{F}:\Set_E \sra \Set_F
\end{equation}

 the functor given by pulling back along the inclusion map $J^{s}_{F}$.

\begin{lemma}\label{l12}
Let $F\subset E$. Then the functor  
$J^{s,*}_{F}: \Set_E \ra \Set_F$ preserves both limits and colimits. 
 \end{lemma}
   
\begin{proof}
Clearly, $J^{s,*}_{F}$ is at once a right adjoint and  a left adjoint. Therefore it preserves both limits and colimits.
\qedhere   
\end{proof}
Let $e\in E$. For $X\in \Set_E$, define the fibre of $X$ over $e$ to be the set $X_e=p^{-1}(e)$, where $p$ is the projection of the  set $X$ over $E$. One has a functor 
\begin{equation}\label{eq27} 
 \pi^{s}_{e}:\Set_E \ra \Set
\end{equation}
defined as follows: For $X\in \Set_E$, $\pi^{s}_{e}(X)=X_e$ and for $f\in \Set_E(X,Y)$, $\pi^{s}_{e}(f)$ is the map $f_e:X_e\ra Y_e$ induced by $f$.

\begin{lemma}\label{l7} \mbox{}  
\begin{enumerate}
 
\item  The functors $\pi^{s}_{e}$, $e\in E$, preserve limits and a functor $X:\mi\ra \Set_E$ has a limit iff the composite functor 
$\pi^{s}_{e}X$ has a limit for all $e\in E$.
 \item The functors $\pi^{s}_{e}$, $e\in E$, preserve colimits and a functor  $X:\mi\ra \Set_E$ has a colimit iff the composite functor 
$\pi^{s}_{e}X$ has a colimit for all $e\in E$.
\end{enumerate}
\end{lemma}
   
\begin{proof}
By  Lemma \ref{l12},  $\pi^{s}_{e}$ preserves limits and colimits. The other properties are easy to verify.
\qedhere   
 \end{proof}
Let $X,Y\in \Set_E$.  Then 
\begin{equation}\label{eq34} 
 P_E(X)\cong\underset{\Set}{\overset{e\in E}{\coprod}} X_e
\end{equation}
and a map $f:X\ra Y$ can be written as 

\begin{equation}\label{eq35} 
 f= \underset{\Set}{\overset{e\in E}{\coprod}}f_e: \underset{\Set}{\overset{e\in E}{\coprod}} X_e\ra  \underset{\Set}{\overset{e\in E}{\coprod}} Y_e
\end{equation}
so that one has a natural isomorphism
\begin{equation}\label{eq32} 
 \Set_E(X,Y)\cong\Set_E( \underset{\Set}{\overset{e\in E}{\coprod}} X_e, \underset{\Set}{\overset{e\in E}{\coprod}} Y_e)\cong  \underset{\Set}{\overset{e\in E}{\prod}}\Set(X_e,Y_e)  
\end{equation} 
Define 
\begin{equation}\label{eq33} 
 \begin{array}{rrcl}
(.)^{(.)}:&\Set^{op}_E\tm \Set_E&\ra &\Set_E\\
    &(Y,Z) &\mapsto &Z^Y= \underset{\Set}{\overset{e\in E}{\coprod}}\Set(Y_e,Z_e)
	\end{array}
\end{equation} 
That is, $Z^Y$ is the set over $E$ whose fibre over $e \in E $ is $\Set(Y_e,Z_e)$.\\
Let $X,Y,Z\in \Set_E$. By Lemma \ref{l7}.1,
$(X\tm_{\Set_E} Y)_e=X_e\tm_{\Set} Y_e$. Therefore
$$ \begin{array}{rclr}
   \Set_E(X\tm_{\Set_E} Y,Z)&\cong& \underset{\Set}{\overset{e\in E}{\prod}}\Set(X_e\tm_{\Set} Y_e,Z_e)    & (\mbox{by } (\ref{eq32}))\\ 
              &\cong& \underset{\Set}{\overset{e\in E}{\prod}}\Set(X_e,\Set(Y_e,Z_e))  &(\mbox{by } (\ref{eq31}))\\ 
							&\cong&\Set_E(X,Z^Y) & (\mbox{by (\ref{eq32}) and  (\ref{eq33}}))  
          \end{array}$$
It follows that $\Set_E$ is cartesian closed.

\begin{remark}\label{r1}(Category of elements, \cite [ Definition 2.4.1.]{RE} and \cite [(3.35)]{K})\\
\begin{enumerate}
 
\item Let $\Set$ be the category of (small) sets and $T:\mi \ra \Set$ be a functor. Recall that
\begin{enumerate}
 
\item The category $\int T$ of elements of $T$ is the 
category whose objects are pairs $(i,s)$ where $i\in \mi$ and $s\in T(i)$ and morphisms from $(i,s)$ to $(j,t)$ are morphisms
$f$ from $i$ to $j$ satisfying $T(f)(s)=t$.
\item The functor $T$ has a colimit iff the  connected components of the category $\int T$ form an object $\pi_0(\int T)\in \Set$, i.e. a  small set . In this case, the cone  $T\st{\lambda}\Ra\pi_0(\int T)$ whose $i$-component is the map 
 $$\begin{array}{rrll}
\lambda_i:&T(i) &\ra &\pi_0(\int T)\\
    &t &\mapsto &\left[(i,t) \right]
	\end{array}$$
	is a colimiting cone.
\item It follows from the previous point that if a cone $T\st{\lambda}\Ra S$  from the functor $T$ to $S\in \Set$ is such that:  
 
\begin{itemize}
 
\item  $\forall s\in T(i)$,   $\forall t\in T(j)$,   $\lambda_i(s)=\lambda_j(t) \Leftrightarrow$ the objects $(i,s)$ and $(j,t)$ of $\int T$ are in the same connected component.
\item   $ \bigcup_{i\in \mi}  \lambda_i(T(i))=S.$ 
\end{itemize}
Then $\lambda$ is a colimiting cone.
\end{enumerate}
\item Assume now that $T: I \ra \Set_E$ is a functor. Let  $p_i:T(i) \ra E$ be the projection of the  set $T(i)$ over $E$
 and   
 $P_E: \Set_E \ra \Set$ the functor given by (\ref{eq23}).
Then by  Lemma \ref{l2} and Remark \ref{r1}.1.(b), $T$ has a colimit iff the  connected components of the category $\int P_ET$ form a (small) set. When this is the case, then the colimit of $T$ is    
$$\begin{array}{rll}
\pi_0(\int P_ET) &\ra &E\\
    \left[(i,t) \right] &\mapsto & p_i(t)
	\end{array}$$
\end{enumerate}
\end{remark}

\begin{lemma}\label{l11}  
Let $X:\mi\ra \Set$ be a functor and $Y\in \Set$. Assume that $X\st{\lambda}\Ra Y$ is a colimiting cone and let $Y'\subset Y$. Then the functor $X$ induces a functor $X':\mi\ra \Set$ given by  $X'(i)=\lambda_i^{-1}(Y')$, $i\in \mi$. Furthermore, the cone $\lambda':X'\Ra Y'$   induced by $\lambda$  is a colimiting cone.  
 \end{lemma}
  
\begin{proof} Clearly, the functor $X$ induces a functor $X':\mi\ra \Set$
and the cone $\lambda$ induces a cone $\lambda': X'\Ra Y'$. Two objects $(i_1,x_1)$ and $(i_2,x_2)$ in the category $\int X'$ are in the same path-component iff they are in the same path-component of the category $\int X$. That is iff $\lambda'_{i_1}(x_1)=\lambda_{i_1}(x_1)=\lambda_{i_2}(x_2)=\lambda'_{i_2}(x_2)$. Furthermore, $ \bigcup_{i\in \mi}  \lambda'_i(X'(i))=Y'$. By Remark \ref{r1}.3, $\lambda'$ is a colimiting cone.
\qedhere   
\end{proof}

\section{Limits in the category of fibrewise spaces}\label{C}
  Define  $\absl{}:\Top\ra\Set$ to be the underlying set functor. $\abs{.}$  has a left adjoint which is the discrete functor 
$$\Disc : \Set \longrightarrow\Top$$
and has a right adjoint which is the codiscrete functor $$\Codisc :\Set\ra\Top$$ In particular, the underlying functor $\abs{.}$ preserves limits and colimits. \\
\indent For any functor $T:\mi\ra\Top$, define the underlying set functor of $T$ to be the composite functor
\begin{equation}\label{eq25}
 \abs{T}:\mi\st{T}{\ra}\Top\st{\abs{.}}{\ra}\Set. 
\end{equation}
\begin{lemma}\label{l13} 
Let $\mi$ be a (not necessarily small) category and $T:\mi\ra\Top$ a functor. Then: 
\begin{enumerate}
 
	\item $T$ has a limit (resp. colimit) iff $\abs{T}$ has a limit (resp. colimit).
	\item Suppose that $\abs{T}$ has a limit (resp. colimit). Then $\lm T$ (resp. $\colim T$) is the topological space whose underlying set is $\lm\abs{T}$ (resp. $\colim \abs{T}$) and whose topology is the initial (resp. final) topology defined by the limiting (resp. colimiting) cone components $\lm\abs{T}\longrightarrow\abs{T(i)}$ (resp. $\abs{T(i)}\longrightarrow\colim\abs{T}$).
\end{enumerate}
\end{lemma}
   
\begin{proof}
Clear.
\qedhere   
\end{proof}
  It follows from the above lemma that functor $\abs{.}:\Top\ra\Set$ preserves and lifts limits and colimits. $\Set$ is bicomplete, then so is $\Top$.\\
\indent The slice category of $\Top$ over $B$ is denoted by $\Top_B$. An object of $\Top_B$ is called a fibrewise topological space over $B$. It consists of a topological space $X$ together with a continuous map $p:X\ra B$ called projection. A fibrewise topological space $p:X\ra B$ is often identified with its domain $X$. Let
\begin{equation}\label{eq26} 
 P_B:   \Top_B \ra \Top 
\end{equation}
be the functor defined as in (\ref{eq19}).

\begin{proposition}\label{p21}\mbox{}

\begin{enumerate}
 
	\item  $\Top_B$ is bicomplete.
	\item  The functor $P_B$ creates and preserves colimits.
\end{enumerate}
\end{proposition}
   
\begin{proof}
This follows from Lemma \ref{l2} and Lemma \ref{l9}.
\qedhere   
\end{proof}
 If $X$ is a fibrewise space over $B$, then $\abs{X}$ is a set over $\abs{B}$ so that one has an underlying \enquote{fibrewise} set functor 
$$\abs{.}:\Top_B\ra\Set_{\abs{B}}. $$
 The functor $\abs{.}$ has a left adjoint which is the ordinary discrete functor 
$$\Disc : \Set_{\abs{B}} \ra\Top_B,$$
and a right adjoint which is the codiscrete functor $$\Codisc :\Set_{\abs{B}}\ra\Top_B.$$
 It associates to a fibrewise set $S$ over $\abs{B}$ the topological space whose underlying set is $S$ and whose topology is the initial topology defined by the projection $p:S\ra \abs{B}$ of the fibrewise set $S$ on $\abs{B}$. \\ 
\indent It follows that the underlying fibrewise set functor $\abs{.}$ preserves both limits and colimits.  We have a commutative diagram of colimit preserving functors
\begin{equation}\label{eq12}
\begin{tikzcd}[sep=large]
    \Top_B\ar[r,"\abs{ .}"]\ar[d,"P_B"']&\Set_{\abs{B}}\ar[d,"P_{\abs{B}}"]\\
    \Top\ar[r,"\abs{ .}"']&\Set\\  
\end{tikzcd}
\end{equation}
where $P_B$ and $P_{\abs{B}}$ are the functors defined by (\ref{eq26}) and (\ref{eq23}) respectively. For any functor $T:\mi\ra\Top_B$, define the underlying  functor of $T$ to be the composite functor
\begin{equation}\label{eq72}
 \abs{T}:\mi\st{T}{\ra}\Top_B\st{\abs{.}}{\ra}\Set_{\abs{B}} 
\end{equation} 

\begin{lemma}\label{l3}  
Let $\mi$ be a (not necessarily small) category and $T:\mi\ra\Top_B$ a functor.
\begin{enumerate}
 
	\item  If one of the functors
	$T$, $\abs{T}$, $P_BT$, $\abs{P_BT}=P_{\abs{B}}\abs{T} $
	has a colimit, then so do the others. 
	\item  Assume that $T$ has a colimit, then $P_B(\colim T)$ is the topological space whose underlying set is $\colim \abs{ P_BT}$
	and whose topology is the final topology induced by the components
	of the colimiting cone $\abs{P_BT} \Rightarrow \colim \abs{ P_BT}$.
\end{enumerate}
\end{lemma}
   
\begin{proof}\mbox{}
\begin{enumerate}
   
	\item The functors in diagram (\ref{eq12}) are left adjoints, they are therefore colimit preserving, we just need to prove that if $\abs{P_BT}$ has a colimit, then so is $T$. Assume then that $\abs{P_BT}$ has a colimit. By Lemma \ref{l13}, $P_BT$ has a colimit. $P_B$ creates colimits, therefore $T$ has a colimit as desired.
	\item This follows immediately the first property, Lemma \ref{l13} and the fact that $P_B$ preserves colimits.
\end{enumerate}
\qedhere   
\end{proof}

\begin{lemma}\label{l16}  
Let $\mi$ be a (not necessarily small) category and $T:\mi\ra\Top_B$ a functor. Then:
\begin{enumerate}
 
	\item $T$ has a limit iff $\abs{T}$ has a limit.
	\item  Assume that $S$ is a  set over $\abs{B}$ and $S \st{\lambda}\Rightarrow \abs{T}$ is a limiting cone. Let $L$ be the topological space whose underlying set is $S$ and whose topology is the initial topology defined by the components of $\lambda$. Then $L$ is a fibrewise space over $B$ and the cone $L\Rightarrow T$ induced by $\lambda$ is a limiting cone.
\end{enumerate}
\end{lemma}
   
\begin{proof} These are consequences of Lemmas \ref{l9} and \ref{l13}.
\qedhere   
\end{proof}

\begin{lemma}\label{l14}  
  Let $X:\mi\ra \Top$ be a functor and $Y\in \Top$. Assume that $X\st{\lambda}\Ra Y$ is a colimiting cone and let $Y'\subset Y$. Let $X':\mi\ra \Top$ be the functor given by  $X'(i)=\lambda_i^{-1}(Y')$, $i\in \mi$. Then the cone $\lambda':X'\Ra Y'$ induced by $\lambda$ is a colimiting cone, provided that $Y'$ is either open or closed in $Y$. 
\end{lemma}
   
\begin{proof}
By Lemma \ref{l11}, the cone $\abs{\lambda'}:\abs{X'}\Ra \abs{Y'}$ is a colimiting cone in $\Set$. Assume that $Y'$ is closed, then $X'(i)$ is closed in $X(i)$, all $i\in \mi$. Let $C \subset Y'$ such that 
$\lambda'^{-1}_{i}(C)$ is closed in $X'(i)$ for all $i\in I$. The subset $\lambda^{-1}_{i}(C)=\lambda'^{-1}_{i}(C)$ is closed in $X(i)$ for all $i\in I$. By Lemma \ref{l13}, $C$ is closed $Y$ and therefore $C$ is closed $Y'$.
It follows that $Y'$ has the final topology defined by the components of the cone $\abs{\lambda'}$. By Lemma \ref{l13}, $\lambda'$ is a colimiting cone. A similar argument can be used in the case where $Y'$ is open.
\qedhere   
\end{proof}

Let $A\subset B$, $J^t_A:A \sra B$ the inclusion map and
\begin{equation}\label{eq30} 
J^{t,*}_A: \Top_B \ra \Top_A
\end{equation}
The functor given by pulling back along the inclusion map $J^t_A$.

\begin{lemma}\label{l15} \mbox{} 
\begin{enumerate}
 
\item The functor $J_A^{t,*}$ preserves limits.
\item  The functor $J_A^{t,*}$ preserves colimits provided that $A$ is either open or closed in $B$.
\end{enumerate}
 \end{lemma}
  
\begin{proof}
The first point results from the fact that $J_A^{t,*}$ is a right adjoint. The second is a consequence of Lemma \ref{l14} and Proposition \ref{p21}.
\qedhere   
\end{proof}

  Let $b\in B$. For $X\in \Top_B$, define the fibre of $X$ over $b$ to be the subspace $X_b=p^{-1}(b)$, where $p$ is the projection of the fibrewise space $X$.
One has a functor
\begin{equation}\label{eq28} 
\pi^{t}_{b}:\Top_B \ra \Top
\end{equation}
defined as follows: For $X\in \Top_B$, $\pi^{t}_{b}(X)=X_b$ and for $f\in \Top_B(X,Y)$, $\pi^{t}_{b}(f)$ is the map $f_b:X_b\ra Y_b$ is the map induced by $f$. 

\begin{lemma}\label{l8} Let $X:\mi\ra \Top_B$ be a functor.
\begin{enumerate} 
 
\item  The functors $\pi^{t}_{b}$, $b\in B$, preserve limits and  the functor $X$ has a limit iff the composite functor 
$\pi^{t}_{b}X$ has a limit for all $b\in B$.
 \item Assume $B$ is a $\T_1$-space. Then the functors $\pi^{t}_{b}$, $b\in B$, preserve colimits and  the functor $X$ has a colimit iff the composite functor 
$\pi^{t}_{b}X$ has a colimit for all $b\in B$.
\end{enumerate}
\end{lemma}
    
\begin{proof}\mbox{} 
\begin{enumerate}
 
\item  By Lemma \ref{l15}.1, the functors $\pi^{t}_{b}$ preserve limits. The fact that $\pi^{t}_{b}X$ has a limit for all $b\in B$
implies that $X$ has a limit is a consequence of   Lemmas \ref{l13}.1, \ref{l7}.1 and \ref{l16}.1.
 \item By Lemma \ref{l15}.2, the functors $\pi^{t}_{b}$  preserve colimits. The fact that $\pi^{t}_{b}X$ has a colimit for all $b\in B$
implies that $X$ has a colimit is a consequence of  Lemmas \ref{l13}.1, \ref{l7}.2 and \ref{l3}.1.
\end{enumerate}
 \qedhere     
\end{proof}



\end{document}